\documentclass[letterpaper,9pt,oneside,twocolumn,final]{nonumish1IEEEtran}
\usepackage[dvips,final]{graphicx}
\usepackage[dvips,ps,all,color]{xy}
\usepackage{subfigure}
\usepackage{amsmath,amssymb}
\usepackage[dvips,usenames]{color}
\usepackage{enumerate}
\usepackage{dsfont}
\usepackage{yfonts}
\usepackage[algosection,linesnumbered,algoruled,vlined,longend]{algorithm2e}
\usepackage{textcomp}
\usepackage{mathrsfs}
\usepackage{pifont}
\usepackage{pxfonts}
\usepackage{listings}
\usepackage[T1]{fontenc}
\usepackage{newcent}
\usepackage[small,euler-digits,icomma]{eulervm}
\usepackage[T1,OT1]{fontenc}
\usepackage{vaucanson-g}
\usepackage{stmaryrd}
\usepackage{here}
\DeclareMathAlphabet{\mathpzc}{OT1}{pzc}{m}{it}

\newtheorem{notn}{Notation}[section]

\newtheorem{cor}{Corollary}[section]
\newtheorem{lem}{Lemma}[section]
\newtheorem{prop}{Proposition}[section]
\newtheorem{defn}{Definition}[section]

\newtheorem{rem}{Remark}[section]

\newcommand{\etal}{\textit{et} \mspace{3mu} \textit{al.}}

\newcommand{\IPq}{\bigg [ \mathbb{I} -
\mathbf{P}+\Q \bigg ]^{-1}}

\newcommand{\eg}{\geqq_{\textrm{\sffamily \textsc{Elementwise}}}}
\newcommand{\egg}{>_{\textrm{\sffamily \textsc{Elementwise}}}}
\newcommand{\ee}{\equiv_{\textrm{\sffamily \textsc{Elementwise}}}}
\newcommand{\Q}{\mathcal{C}(\Pi)}
\newcommand{\Eps}{\lambda}
\renewcommand{\P}{\mathscr{P}(\Pi)}

\newcommand{\p}{\mathfrak{p} \hspace{0pt}}

\newcommand{\G}{$G=(Q,\Sigma,\delta,\widetilde{\Pi},\chi,\mathscr{C})$ }

\newcommand{\Gt}{$G_{\theta}=(Q,\Sigma,\delta,(1-\theta)\widetilde{\Pi},\chi,\mathscr{C})$ } 
\newcommand{\E}{$\mathscr{E}_{(G,\p)} = (Q_\mathscr{F},\Sigma,
\Delta,\tilde{\pi}_\mathscr{E},\chi_\mathscr{E})$ } 
\newcommand{\Qe}{$Q_\mathscr{F}$} 
\newcommand{\QE}{Q_\mathscr{F}} 
\newcommand{\B}{\big [ \mathbb{I} - (1-\theta) \mathscr{P}(\Pi) \big ]^{-1}}
\newcommand{\infnrm}[1]{ \left \lVert  #1 \right \rVert_\infty}
\renewcommand{\baselinestretch}{1.1}                      
\newcommand{\Crd}{\textrm{\sffamily\textsc{Card}}}
\newcommand{\myar}{\ar@[|(1.5)]}
\newcommand{\myarT}{\ar@[|(3.5)]}
\newcommand{\myarTl}{\ar@[|(2.5)]}
\newcommand{\myarL}{\ar@[|(.8)]}
\newcommand{\myarH}{\ar@[|(1.25)]}
\newcommand{\msp}{\mspace{-8mu}}

\newcommand{\BRed}{\color{BrickRed}}
\renewcommand{\thesection}{\arabic{section}}
\xyoption{arc}\xyoption{poly}\xyoption{crayon}\centerfigcaptionstrue
\title{
\footnotesize \red {Full Version Submitted For Publication in \textbf{\emph{ International Journal of Control}}} \black{}\\
 \Large \textbf{Optimal Control of Infinite Horizon Partially Observable Decision Processes \\ Modeled As Generators of 
Probabilistic Regular Languages$^{\bigstar}$}
\vspace{-2pt}
\thanks{\hrule  \vspace{2pt} $^{\star}$This work has been supported in part by the U.S.
Army Research laboratory and the U.S. Army
Research Office under Grant No.
W911NF-07-1-0376.}
\thanks{$^{\ddag}$The Pennsylvania State University, University Park, PA}}
\author{ \begin{tabular}{ccccccc}
Ishanu Chattopadhyay$^{\ddag}$ & & & Asok Ray$^{\ddag}$\\
{\tt ixc128@psu.edu} & & & {\tt axr2@psu.edu} \\
\end{tabular}
\vspace{-22pt}
\\
} 
\begin{document}

\renewcommand{\thesubsectiondis}{\arabic{section}.\arabic{subsection}.}
\renewcommand{\thesubsubsectiondis}{\arabic{section}.\arabic{subsection}.\arabic{subsubsection}}
\renewcommand{\thesubsection}{\arabic{section}.\arabic{subsection}}
\renewcommand{\thesubsubsection}{\arabic{section}.\arabic{subsection}.\arabic{subsubsection}}

\allowdisplaybreaks{

\maketitle \vspace{-50pt} \allowdisplaybreaks{
\begin{abstract} 
Decision processes with incomplete state feedback have been traditionally modeled as 
Partially Observable Markov Decision Processes. In this paper, we present 
an alternative formulation based on probabilistic regular languages. The proposed approach
generalizes the recently reported work on language measure theoretic optimal control for perfectly observable 
situations and shows that such a framework is far more computationally tractable to the classical alternative.
In particular, we show that the infinite horizon decision problem under partial observation, modeled in the proposed framework, is $\Eps$-approximable
and, in general, is no harder to solve compared to the fully observable case. 
The approach is illustrated via two simple examples.
%
%
%
\end{abstract} 
\begin{keywords}
POMDP; Formal Language Theory; Partial Observation; Language Measure; Discrete Event Systems
\end{keywords}\vspace{0pt}

\section{Introduction \& Motivation}\vspace{0pt}
Planning under uncertainty is one of the oldest  and most studied problems in
research literature pertaining to automated decision making and artificial intelligence.
The central objective is to sequentially choose  control actions for  one or more agents interacting with
the operating environment such that some associated reward function is maximized for 
a pre-specified finite future (finite horizon problems) or for all possible futures (infinite 
horizon problems). Among the various mathematical formalisms studied to model and solve 
such problems, Markov Decision Processes (MDPs) have
received significant attention. A brief overview of the current state of art in MDP-based
decision theoretic planning is necessary to place this work in appropriate context.
\subsection{Markov Decision Processes}
MDP models~\cite{Put90,W93} extend the classical planning framework~\cite{MR91,PW92,PW93,KHW95} to accommodate uncertain 
effects of agent actions with the associated control algorithms attempting to maximize expected
reward and is capable, in theory, of handling  realistic decision scenarios 
arising in operations research, optimal control theory and, more recently, autonomous mission 
planning in probabilistic robotics~\cite{AK01}. In brief, a MDP consists of states and
actions with a set of action-specific probability transition matrices allowing one to compute the distribution
over model states resulting from  the execution of  a particular action sequence. 
Thus the endstate resulting from an action is not known uniquely apriori. However the agent is assumed to occupy
one and only one state at any given time, which is correctly observed, once the action sequence is complete.
Furthermore, each state is associated with a reward value and the performance of a controlled MDP
is the integrated reward over specified operation time (which can be infinite).
A partially observable Markov decision process (POMDP) is a generalization of MDPs 
which assumes  actions  to be  nondeterministic as in a MDP but relaxes the assumption 
of perfect knowledge of the current model state. 

A policy for a MDP is a mapping from the set of states to the set of actions. 
If both sets are assumed to be finite, the number of possible mappings is also finite implying that
an optimal policy can be found by conducting search over this finite set. In a POMDP, on the other hand, 
the current state can be only estimated as a distribution over underlying model states
 as a function 
of operation and observation history. 
The space of all such estimations or \textit{belief states}
is a continuous 
space although the underlying model has only a finite number of states. In contrast to MDPs,  a POMDP policy is a mapping 
from the belief space to the set of actions implying
that computation of the optimal policy demands a search over a continuum
making the problem drastically more difficult to solve. 
\subsection{Negative Results Pertaining to POMDP Solution}\label{secintroneg}
As stated above, an optimal solution to a POMDP is a policy which 
specifies  actions to execute in response to state feedback with the objective of maximizing performance. 
Policies may be \textit{deterministic} with a single action specified at each belief
 state or \textit{stochastic} which 
specify an allowable choice of actions at each state. Policies
can be also categorized as \textit{stationary}, \textit{time dependent} or \textit{history dependent}; stationary policies 
only depend on the current belief state, time dependent policies may vary with the operation time
and history dependent policies vary with the state history.
The current state of art in POMDP solution algorithms~\cite{Z01,CK98} are all
variations of Sondick's original work~\cite{S78} on value iteration  based on Dynamic Programming (DP). Value iterations,
 in general, are required to solve large numbers of linear programs at each DP update and consequently
suffer from exponential worst case complexity. 
Given that it is hard to find an optimal policy, it is natural to try to seek
one that is \textit{good} enough. Ideally, one would be reasonably satisfied to have 
an algorithm guaranteed to be fast which produces a policy that is reasonably close ($\Eps$-approximation)
to the optimal solution. Unfortunately, existence of such  algorithms is unlikely or,
in some cases, impossible. Complexity results show that POMDP solutions are nonapproximable~\cite{BRS96,LGM01,MHC99} with 
the above stated guarantee existing in general only if certain complexity classes collapse. For example,
 the optimal stationary policy for POMDPs of finite state space can be $\Eps$-approximated if and only if P=NP.
 Table~\ref{tabcomplexity} reproduced from \cite{LGM01} summarizes the known complexity results in this context.
\begin{table}[t]
\centering
\caption{ $\Eps$-Approximability  Of Optimal POMDP Solutions}\label{tabcomplexity}
\begin{tabular}{||l|c|l||}
 \hline \textbf{Policy} &  \textbf{Horizon}  &\textbf{Approximability} \\ \hline \hline
Stationary  & K & Not unless P=NP \\ \hline
Time-dependent  & K & Not unless P=NP\\ \hline
Histpry-dependent  & K & Not unless P=PSPACE\\ \hline
Stationary  & $\infty$ & Not unless P=NP\\ \hline
Time-dependent  & $\infty$ & Uncomputable\\\hline
\end{tabular}
\end{table}
Thus finding the history dependent optimal policy for even a  finite horizon POMDP is PSPACE-complete. 
Since this is a broader problem class than NP, the result suggests that POMDP problems 
are even harder than NP-complete problems. Clearly, infinite horizon POMDPs can be no easier to solve 
than finite horizon POMDPs. 
In spite of recent development of new exact and approximate 
algorithms to efficiently compute optimal solutions~\cite{CK98} and machine learning approaches to cope with uncertainty~\cite{Hans98}, 
the most efficient algorithms to date are able to compute near optimal solutions 
only for POMDPs of relatively small state spaces.
\subsection{Probabilistic Regular Language Based Models}
This work investigates  decision-theoretic planning under partial observation 
in a framework distinct from the MDP philosophy.
Decision processes are modeled as Probabilistic
Finite State Automata (PFSA) which act as generators of
probabilistic regular languages~\cite{CR08}. 
\begin{quote}
\textit{It is important to note that the PFSA model
used in this paper is conceptually very different from the notion of probabilistic automata introduced by Rabin, Paz and others~\cite{R63,P71} and 
essentially follows the formulation of p-language theoretic analysis
first reported by Garg $\etal$~\cite{G92,G92-2}.} 
\end{quote}
 The key differences between the MDP framework
and PFSA based modeling can be enumerated briefly as follows:
\begin{enumerate}
\item In both MDP and PFSA formalisms, we have the
notion of states. The notion of actions in the former is analogous to
that of events in the latter. However, unlike actions in the MDP framework, 
which can be executed at will (if defined at the current state), 
generation of events in the context of PFSA models, 
is probabilistic. Also, such events are categorized as being controllable or uncontrollable.
A controllable event can be ``disabled'' so that state change due to 
generation of that particular event is inhibited; uncontrollable events,
on the other hand, cannot be disabled in this sense.
 \item For a MDP, given a state and an action selected for execution, we can only 
compute the probability distribution over
 model states  resulting from the action; although the agent ends up in 
an unique state due to execution of the chosen action, this endstate cannot be determined 
apriori. For a PFSA, on the other hand,  given a state,  we 
only know the probability of occurrence of each alphabet symbol as the next to-be generated 
event each of which causes a transition to a apriori known unique endstate; however
 the next state is still uncertain due to the possible
execution of uncontrollable events defined at the current state. Thus, both formalisms
aim to capture the uncertain  effects of agent decisions; albeit via different mechanisms.
\item Transition probabilities in MDPs are, in general, functions of both the current state
and the action executed; $i.e.$ there are $m$ transition probability matrices where 
$m$ is the cardinality of the set of actions. PFSA models, on the other hand, have only one
transition probability matrix computed from the state based event generation probabilities.
\item It is clear that MDPs emphasize states and state-sequences; while PFSA models
emphasize events and event-sequences. For example, in POMDPs, the
observations are states; while those in the observability model for PFSAs (as adopted
in this paper) are events.
\item In other words, partial observability in MDP directly results
in not knowing the current state; in PFSA models
partial observability results in not knowing transpired events which as an effect causes confusion in the  determination of the current state.
\end{enumerate}
\begin{figure}[t]
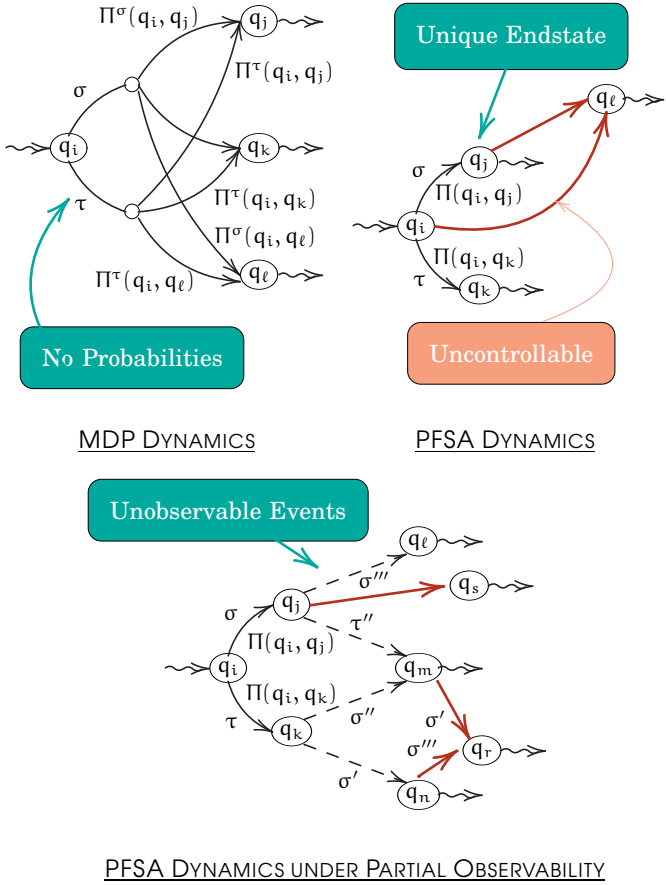

\xy
(0,0)*{\xy 0;/r.2pc/:
(0,0)*+[o][F-]{q_i};%
(10,10)*+[o][F-]{};%
(10,-10)*+[o][F-]{};%
(30,20)*+[o][F-]{q_j};%
(30,0)*+[o][F-]{q_k};%
(30,-20)*+[o][F-]{q_\ell};%
{\myar@{{}{-}{}}@[Black]@/^0.25pc/^{\txt{$\sigma$}}(0,2)*{};(9,10)*{}};
{\myar@{{}{-}{}}@[Black]@/_0.25pc/_{\txt{$\tau$}}(0,-2)*{};(9,-10)*{}};
{\myar@{{}{-}{>}}@[Black]@/^0.5pc/^{\txt{$\Pi^\sigma(q_i,q_j)$}}(11,11)*{};(27,20)*{}};
{\myar@{{}{-}{>}}@[Black]@/_0.5pc/(11,10)*{};(27,0)*{}};
{\myar@{{}{-}{>}}@[Black]@/_0.5pc/(11,9)*{};(27,-20)*{}};
{\myar@{{}{-}{>}}@[Black]@/_0.5pc/_{\txt{$\Pi^\tau(q_i,q_\ell)$}}(11,-11)*{};(27,-21)*{}};
{\myar@{{}{-}{>}}@[Black]@/_0.5pc/(11,-10)*{};(27,0)*{}};
{\myar@{{}{-}{>}}@[Black]@/_0.5pc/(11,-9)*{};(27,20)*{}};
{\myar@{{}{~}{>}}@[Black](32,0)*{};(40,0)*{}};
{\myar@{{}{~}{>}}@[Black](32,20)*{};(40,20)*{}};
{\myar@{{}{~}{>}}@[Black](32,-20)*{};(40,-20)*{}};
{\myar@{{}{~}{>}}@[Black](-11,0)*{};(-3,0)*{}};
(34,12)*{\txt{$\Pi^\tau(q_i,q_j)$}};
(31,-14)*{\txt{$\Pi^\sigma(q_i,q_\ell)$}};
(31,-8)*{\txt{$\Pi^\tau(q_i,q_k)$}};
(10,-33)*+++{\txt{No Probabilities}}*[Emerald]%
           \frm<5pt>{**}%
           *+++[white]{\txt{No Probabilities}};
{\myarTl@{{}{-}{>}}@/^1.2pc/@[Emerald](0,-35)*{};(0,-8)*{}};
\endxy};
(45,0)*{\xy
0;/r.2pc/:
(0,0)*+[o][F-]{q_i};%
(10,10)*+[o][F-]{q_j};%
(10,-10)*+[o][F-]{q_k};%
(30,20)*+[o][F-]{q_\ell};%
{\myar@{{}{-}{>}}@[Black]@/^0.25pc/^{\txt{$\sigma$}}(0,2)*{};(7,10)*{}};
{\myar@{{}{-}{>}}@[Black]@/_0.25pc/_{\txt{$\tau$}}(0,-2)*{};(7,-10)*{}};
{\myar@{{}{~}{>}}@[Black](12,10)*{};(20,10)*{}};
{\myar@{{}{~}{>}}@[Black](12,-10)*{};(20,-10)*{}};
{\myar@{{}{~}{>}}@[Black](-11,0)*{};(-3,0)*{}};
{\myarTl@{{}{-}{>}}@[BrickRed](12,12)*{};(27,20)*{}};
{\myarTl@{{}{-}{>}}@/_1.5pc/@[BrickRed](3,0)*{};(30,18)*{}};
{\myar@{{}{~}{>}}@[Black](32,20)*{};(39,20)*{}};
(10,5)*{\txt{$\Pi(q_i,q_j)$}};
(10,-5)*{\txt{$\Pi(q_i,q_k)$}};
(15,30)*+++{\txt{Unique Endstate}}*[Emerald]%
           \frm<5pt>{**}%
           *+++[white]{\txt{Unique Endstate}};
{\myarTl@{{}{-}{>}}@[Emerald](15,27)*{};(10,14)*{}};
(15,-20)*+++{\txt{Uncontrollable}}*[Melon]%
           \frm<5pt>{**}%
           *+++[white]{\txt{Uncontrollable}};
{\ar@{{}{-}{>}}@/_2pc/@[Melon](20,-15)*{};(22,4)*{}};
\endxy};
(20,-60)*{\xy
0;/r.2pc/:
(0,0)*+[o][F-]{q_i};%
(10,10)*+[o][F-]{q_j};%
(10,-10)*+[o][F-]{q_k};%
(30,20)*+[o][F-]{q_\ell};%
(30,0)*+[o][F-]{q_m};%
(30,-20)*+[o][F-]{q_n};%
(40,-13)*+[o][F-]{q_r};%
(38,13)*+[o][F-]{q_s};%
{\myar@{{}{-}{>}}@[Black]@/^0.25pc/^{\txt{$\sigma$}}(0,2)*{};(7,10)*{}};
{\myar@{{}{-}{>}}@[Black]@/_0.25pc/_{\txt{$\tau$}}(0,-2)*{};(7,-10)*{}};
{\myar@{{}{~}{>}}@[Black](-11,0)*{};(-3,0)*{}};
{\myarTl@{{}{-}{>}}@[BrickRed]^{\txt{$\sigma'''$}}(13,10)*{};(34,13)*{}};
{\myarTl@{{}{-}{>}}@[BrickRed]^{\txt{$\sigma'''$}}(30,-17)*{};(36,-13)*{}};
{\myarTl@{{}{-}{>}}@[BrickRed]_{\txt{$\sigma'$}}(33,-2)*{};(38,-11)*{}};
{\myar@{{}{--}{>}}_{\txt{$\sigma'$}}@[Black](12,-12)*{};(28,-18)*{}};
{\myar@{{}{--}{>}}@[Black](12,12)*{};(28,18)*{}};
{\myar@{{}{--}{>}}_{\txt{$\sigma''$}}@[Black](12,-8)*{};(28,-2)*{}};
{\myar@{{}{--}{>}}^{\txt{$\tau''$}}@[Black](12,8)*{};(28,2)*{}};
{\myar@{{}{~}{>}}@[Black](32,0)*{};(40,0)*{}};
{\myar@{{}{~}{>}}@[Black](32,20)*{};(40,20)*{}};
{\myar@{{}{~}{>}}@[Black](32,-20)*{};(40,-20)*{}};
{\myar@{{}{~}{>}}@[Black](40,13)*{};(48,13)*{}};
{\myar@{{}{~}{>}}@[Black](42,-13)*{};(50,-13)*{}};
(10,4)*{\txt{$\Pi(q_i,q_j)$}};
(10,-4)*{\txt{$\Pi(q_i,q_k)$}};
(0,25)*+++{\txt{Unobservable Events}}*[Emerald]%
           \frm<5pt>{**}%
           *+++[white]{\txt{Unobservable Events}};
{\myarTl@{{}{-}{>}}@[Emerald](5,22)*{};(14,16)*{}};
\endxy};
(45,-33)*{\txt{\sffamily \underline{\textsc{PFSA Dynamics}}}};
(0,-33)*{\txt{\sffamily \underline{\textsc{MDP Dynamics}}}};
(25,-90)*{\txt{\sffamily \underline{\textsc{PFSA Dynamics under Partial Observability}}}};
\endxy
\vspace{5pt}
\caption{Comparison of modeling semantics for
MDPs and PFSA}\label{figcompdyn}
\end{figure}
This paper presents an efficient algorithm for computing
the history-dependent~\cite{LGM01} optimal supervision policy for infinite horizon decision problems
modeled in the PFSA framework. The key tool used is the recently reported 
concept of a rigorous language  measure for probabilistic finite state language generators~\cite{CR06}.
This is a generalization of the 
work on language measure-theoretic optimal control for the fully observable case~\cite{CR07} and we show in this paper, that
the partially observable scenario is no harder to solve in this modeling framework.

The rest of the organized in five additional sections and two brief appendices.
Section~\ref{secprelim} introduces the preliminary concepts and relevant results from reported literature. Section~\ref{seconlineS} presents an online implementation of the language measure-theoretic supervision policy for perfectly observable plants which 
lays the framework for the subsequent development of the proposed 
optimal control policy for partially observable systems in Section~\ref{optip-intro}.
The theoretical development is verified and validated in two simulated examples 
in Section~\ref{secvalid}. The paper is summarized and concluded in Section~\ref{secsum}
with recommendations for future work.
\section{Preliminary Concepts \& Related Work}\label{secprelim}
This section presents the formal definition of the PFSA model and summarizes the concept of signed real measure of regular
languages; the details are reported in~\cite{R05}~\cite{RPP05}~\cite{CR06}. Also, we briefly review the computation of the unique maximally permissive  optimal control policy for
probabilistic finite state automata ($PFSA$)~\cite{CR07} via maximization of the language measure.
In the sequel, this measure-theoretic approach will be generalized  to address partially observable cases and is thus critical to the development presented in this paper.
\subsection{The PFSA Model}\label{PFSAmodel}
Let $G_i = ( Q,\Sigma,\delta,q_{i},Q_{m})$ be a
 finite-state automaton
model that encodes all possible evolutions of the discrete-event dynamics of a physical
plant, where $Q=\{q_k: k\in \mathcal{I}_Q\}$ is the set\index{set}
of states and $\mathcal{I}_Q\equiv\{1, 2, \cdots, n\}$ is the
index set of states; the automaton starts with the initial state
$q_{i}$; the alphabet of events is $\Sigma=\{\sigma_k: k\in
\mathcal{I}_\Sigma\}$, having $\Sigma \bigcap \mathcal{I}_Q =
\emptyset$ and $\mathcal{I}_\Sigma\equiv\{1, 2, \cdots, \ell\}$ is
the index set of events; $\delta:Q\times\Sigma \rightarrow Q$ is
the (possibly partial) function of state transitions; and
$Q_{m}\equiv \{q_{m_1},q_{m_2},\cdots ,q_{m_l}\} \subseteq Q$ is
the set of marked (i.e., accepted) states with $q_{m_k}=q_j$ for
some $j \in \mathcal{I}_Q$. Let $\Sigma^{*}$  be the Kleene closure of $\Sigma$, i.e., the set of all finite-length strings made of the events
belonging to $\Sigma$ as well as the empty string $\epsilon$ that
is viewed as the identity  of the monoid $\Sigma^{*}$ under the
operation of string concatenation, i.e., $\epsilon s=s=s\epsilon$.
The state transition map $\delta$ is recursively extended to its reflexive and transitive closure  $\delta:Q \times \Sigma^* \rightarrow Q$ by defining
\begin{subequations}
\begin{gather}
\forall q_j \in Q, \ \delta(q_j,\epsilon) = q_j \\
\forall q_j \in Q, \sigma \in \Sigma,  s \in \Sigma^\star, \ \delta(q_i,\sigma s) = \delta(\delta(q_i,\sigma),s)
\end{gather}
\end{subequations}
\vspace{-10pt}
\begin{defn}\label{Lgen}
The language $L(q_i)$ generated by a DFSA $G$ initialized at the
state $q_i\in Q$ is defined as:
\begin{equation} \label{LG-q}
    L(q_i) = \{s \in \Sigma^* \ | \ \delta^*(q_i, s) \in Q \}
\end{equation}
The language $L_m(q_i)$ marked by the DFSA $G$ initialized at the
state $q_i\in Q$ is defined as:
\begin{equation} \label{LmG-q}
    L_m(q_i) = \{s \in \Sigma^* \ | \ \delta^*(q_i, s) \in Q_m \}
\end{equation}
\end{defn}
\begin{defn}
For every $q_j\in Q$, let $L(q_i, q_j)$ denote the set of all
strings that, starting from the state $q_i$, terminate at the
state $q_j$, i.e.,
\begin{equation}
L_{i,j} = \{ s \in \Sigma^* \ | \ \delta^*(q_i, s) = q_j \in Q \}
\end{equation}
\end{defn}
To complete the specification of a probabilistic finite state automata,
 we need to specify the event generation probabilities and the state characteristic
weight vector; which we define next.
\begin{defn} \label{pitildefn}
The event generation
probabilities are specified by the function $\tilde {\pi }:Q \times \,\Sigma^\star
\to [0,\,1]$ such that $\forall q_j \in Q,\forall
\sigma _k \in \Sigma , \forall s \in \Sigma^\star ,$
\begin{enumerate}
\item[(1)] $\tilde{\pi}(q_j,{\sigma_k}) \triangleq
\tilde{\pi}_{jk} \in [0, 1)$; \ $\sum_{k} \tilde{\pi}_{jk} = 1
 - \theta, \ \mathrm{with} \  \theta \in (0,1) $; \item[(2)] $\tilde{\pi}(q_j,\sigma) = 0$ if
$\delta(q_j,\sigma)$ is undefined; $\
\tilde{\pi}(q_j,\epsilon) = 1$; \item[(3)]
$\tilde{\pi}(q_j,{\sigma_k s})= \tilde{\pi}(q_j,{\sigma_k})\
\tilde{\pi}(\delta(q_j,\sigma_k),s)$.
\end{enumerate}
\end{defn}\vspace{3pt}
\begin{notn}
The $n\times \ell$ event cost matrix $\widetilde{\Pi}$ is defined as:
$
\widetilde{\Pi}\vert_{ij}=\tilde{\pi}(q_i,\sigma_j)
$ 
\end{notn}
\begin{defn} \label{pifn}
The state transition probability $\pi: Q \times Q \rightarrow [0, 1)$,
of the DFSA\index{DFSA} $G_i$ is defined as follows:
\begin{eqnarray}
\forall q_i, q_j \in Q, \pi_{ij} =
    \displaystyle \sum_{\sigma\in\Sigma \ \mathrm{s.t.} \  \delta(q_i,\sigma)=q_j } \tilde{\pi}(q_i, \sigma)
\end{eqnarray}
\end{defn}\vspace{3pt}
\begin{notn}

The $n\times n$ state transition probability matrix $\Pi$ is
defined as
$
\Pi\vert_{ij} = \pi(q_i,q_j)
$
\end{notn}

The set $Q_m$ of  marked states is partitioned into $Q_m^+$ and
$Q_m^-$, i.e.,  $Q_m = Q_m^+ \cup Q_m^-$ and $Q_m^+ \cap Q_m^- =
\emptyset$, where $Q_m^+$ contains all \textit{good} marked states
that we desire to reach, and $Q_m^-$ contains all \textit{bad}
marked states that we want to avoid, although it may not always be
possible to completely avoid the \textit{bad} states while
attempting to reach the \textit{good} states. To characterize
this, each marked state is assigned a real value based on the
designer's perception of its impact on the system performance.
\begin{defn} \label{charfn}
The characteristic function $\chi:Q \rightarrow [-1, 1]$ that
assigns a signed real weight to state-based sublanguages
$L(q_i,q)$ is defined as:
\begin{equation}\label{chi}
    \forall q \in Q, \quad \chi(q) \in \left\lbrace
        \begin{array}{cc}
            [-1, 0), & q \in Q_m^-\\
            \{ 0 \}, & q \notin Q_m\\
            \rm{(0, 1]}, & \it{q} \in Q_m^+
        \end{array}
    \right.
\end{equation}
The state weighting vector, denoted by $\boldsymbol{\chi} =
[\chi_1 \ \chi_2 \ \cdots \ \chi_{n}]^T$, where $\chi_j\equiv
\chi(q_j)$ $\forall j \in \mathcal{I}_Q$, is called the
$\boldsymbol{\chi}$-vector. The $j$-th element $\chi_j$ of
$\boldsymbol{\chi}$-vector is the weight assigned to the
corresponding terminal state $q_j$.
\end{defn}
\begin{rem}\label{remcompreward}
 The state characteristic function $\chi:Q \rightarrow [-1,1]$ or equivalently 
the characteristic vector $\boldsymbol{\chi}$ is analogous to the notion 
of the reward function in MDP analysis. However, unlike MDP models, where
the reward (or penalty) is put on individual state-based actions, in our model, the
characteristic is put on the state itself. The similarity of the two notions is
 clarified by noting that just as MDP performance can be evaluated as
the total reward garnered as actions are executed sequentially, the
performance of a PFSA can be computed by summing the characteristics of the states
visited due to transpired event sequences.
\end{rem}
Plant models considered in this paper are
\textit{deterministic} finite state automata (plant) with
well-defined event occurrence \textit{probabilities}. In other
words, the occurrence of events is probabilistic, but the
state at which the plant ends up, \textit{given a particular
event has occurred}, is deterministic. No
emphasis is laid on the initial state of the plant $i.e.$ we allow for the fact that the plant may start from any state. Furthermore,
having defined the characteristic state weight vector
$\boldsymbol{\chi}$, it is not necessary to specify the set of
marked states, because if $\chi_i = 0$, then $q_i$ is not
marked and if $\chi_i \neq 0$, then $q_i$ is marked.

\begin{defn}\label{contapp}{(Control Philosophy)}
If $q_i \xrightarrow[\sigma]{} q_k$,
and the  event $\sigma$ is disabled at state
$q_i$, then the  supervisory action is to
prevent the plant from making a transition to the state $q_k$, by
 forcing it to stay at the original state
$q_i$.  Thus disabling any transition $\sigma$ at a given state $q$ results
in   deletion of the original transition and appearance of the self-loop $\delta(q,\sigma) = q$ with the occurrence
probability of $\sigma$ from the state $q$ remaining unchanged
in the supervised and unsupervised plants.
\end{defn}\vspace{0pt}
\begin{defn}\label{optifull-contdef}{(Controllable Transitions)}
For a given plant, transitions that can be disabled in the
sense of Definition~\ref{contapp} are defined to be
\textit{controllable} transitions. The set of
controllable transitions in a plant is denoted $\mathscr{C}$.
\textit{Note controllability is state-based.}
\end{defn}

It follows that plant models can be specified by the
sextuplet:
\begin{gather}\label{sextuple}
G=(Q,\Sigma,\delta,\widetilde{\Pi},\boldsymbol{\chi},\mathscr{C})
\end{gather} \vspace{-10pt}

\subsection{Formal Language Measure for Terminating Plants}  \label{BriefReview_term}
The formal language measure is first defined for terminating
plants~\cite{G92} with sub-stochastic event generation
probabilities, i.e., the event generation probabilities at
each state summing to strictly less than unity.
%
In general, the marked language $L_m(q_i)$ consists of both good
and bad event strings that, starting from the initial state $q_i$,
lead to $Q_m^+$ and $Q_m^-$ respectively. Any event string
belonging to the language $L^0(q_i) = L(q_i) - L_m(q_i)$ leads to one
of the non-marked states belonging to $Q - Q_m$ and $L^0$ does not
contain any one of the good or bad strings. Based on the
equivalence classes defined in the Myhill-Nerode Theorem~\cite{HMU01}, the
regular languages $L(q_i)$ and $L_m(q_i)$ can be expressed as:
\begin{equation} \label{LGq}
L(q_i) = \bigcup_{q_k \in Q} L_{i,k}
\end{equation}
\begin{equation} \label{LmGq}
L_m(q_i) = \bigcup_{q_k \in Q_m} L_{i,k} =  L_m^+ \cup L_m^-
\end{equation}
where the sublanguage  $L_{i,k}\subseteq L(q_i)$ having the initial
state $q_i$ is uniquely labelled by the terminal state $q_k, k \in
\mathcal{I}_Q$ and $L_{i,j} \cap L_{i,k} = \emptyset$ $\forall j
\neq k$; and $L_m^+\equiv\bigcup_{q_k \in Q_m^+} L_{i,k}$ and
$L_m^-\equiv\bigcup_{q_k \in Q_m^-} L_{i,k}$ are good and bad
sublanguages of $L_m(q_i)$, respectively. Then, $L^0 =
\bigcup_{q_k \notin Q_m} L_{i,k}$ and $L(q_i) = L^0 \cup L_m^+
\cup L_m^-$.\vspace{3pt}

A signed real measure $\mu^i:{2^{L(q_i)}} \rightarrow
\mathbb{R}\equiv(-\infty,+\infty)$ is constructed on the
$\sigma$-algebra $2^{L(q_i)}$ for any $i \in  \mathcal{I}_Q$;
interested readers are referred to~\cite{R05}~\cite{RPP05} for the
details of measure-theoretic definitions and results. With the
choice of this $\sigma$-algebra, every singleton set made of an
event string $s \in L(q_i)$ is a measurable set. By Hahn
Decomposition Theorem~\cite{R88}, each of these measurable sets
qualifies itself to have a numerical value based on the above
state-based decomposition of $L(q_i)$ into $L^0$(null),
$L^{+}$(positive), and $L^-$(negative) sublanguages.\vspace{3pt}

\begin{defn} \label{measurefn}
Let $\omega \in L(q_i, q_j)\subseteq 2^{L(q_i)}$. The signed
real measure $\mu^i$ of every singleton string set $ \{ \omega
\} $ is defined as:
\begin{equation}
\mu^i(\{\omega \})=\tilde{\pi}(q_i, \omega)\chi(q_j)
\end{equation}
The signed real measure of a sublanguage $L_{i,j} \subseteq
L(q_i)$ is defined as:
\begin{equation}\label{measureEqn}
\mu_{i,j} = \mu^i(L(q_i, q_j)) = \left( \sum_{\omega\in
L(q_i, q_j)} \tilde{\pi} (q_i, \omega)\right)\chi_j
\end{equation}
\end{defn}

Therefore, the signed real measure of the language  of a DFSA
$G_i$ initialized at $q_i \in Q$, is defined as
\begin{equation} \label{fullObsMeasure}
\mu_i= \mu^i(L(q_i)) = \sum_{j\in \mathcal{I}_Q}
\mu^i(L_{i,j})
\end{equation}

It is shown in \cite{R05}~\cite{RPP05} that the language measure
in Eq. (\ref{fullObsMeasure}) can be expressed as
\begin{gather} \label{algbrameasure}
\mu_i = \sum_{j\in \mathcal{I}_Q} {\pi}_{ij} \mu_j +
\chi_i
\end{gather}

The language measure vector, denoted as \mbox{\boldmath $\mu$} =
$[\mu_1 \ \mu_2 \ \cdots \ \mu_n]^{T}$, is called the
\mbox{\boldmath $\mu$}-vector.  In vector form, Eq.
(\ref{algbrameasure}) becomes
\begin{equation}
\boldsymbol{\mu} = \Pi\boldsymbol{\mu}+\boldsymbol{\chi}
\end{equation}
whose solution is given by
\begin{equation} \label{fullObservation}
\boldsymbol{\mu} = (\mathbb{I} - \Pi)^{-1}
\boldsymbol{\chi}
\end{equation}
The inverse in Eq. (\ref{fullObservation}) exists for
terminating plant models~\cite{G92}\cite{G92-2} because
$\Pi$ is a contraction
operator~\cite{R05}~\cite{RPP05} due to the strict inequality
$ \sum_j \pi_{ij} < 1$. The residual $\theta_i = 1 - \sum_j
\pi_{ij}$ is  referred to as the termination probability for
state $q_i \in Q$. We extend the analysis to non-terminating
plants~\cite{G92}\cite{G92-2} with stochastic transition probability matrices ($i.e.$
with $\theta_i = 0,\ \forall q_i \in Q$) by renormalizing the
language measure~\cite{CR06} with respect to the uniform
termination probability of a limiting terminating model as
described next.

Let $\widetilde{\Pi}$ and $\Pi$ be the stochastic event
generation and transition probability matrices for a
non-terminating plant $G_i=(
Q,\Sigma,\delta,q_{i},Q_{m})$. We consider the
terminating plant $G_i(\theta)$ with the same DFSA structure $
( Q,\Sigma,\delta,q_{i},Q_{m})$ such that the
event generation probability matrix is given by
$(1-\theta)\widetilde{\Pi}$ with $\theta \in (0,1)$ implying
that the state transition probability matrix is  $(1 -
\theta)\Pi$.
\begin{defn}{(Renormalized Measure) } \label{defrenormmeas}
The renormalized measure {\small $\nu^i_\theta : 2^{L(q_i)}
\rightarrow [-1,1]$} for the $\theta$-parametrized terminating plant $G_i(\theta)$ is  defined as:
\begin{gather}
\forall \omega \in L(q_i), \ \nu^i_\theta(\{\omega\}) = \theta \mu^i(\{\omega\})
\end{gather}
The corresponding matrix form is given by
\begin{gather} \label{renormalizedMeasure}
\boldsymbol{\nu_\theta} = \theta \ \boldsymbol{\mu}=
\theta \ [I-(1-\theta)\Pi]^{-1} \boldsymbol{\chi} \
\mathrm{with} \ \theta \in (0,1)
\end{gather}
We note that the vector representation allows for the following notational simplification
\begin{gather}
  \nu_\theta^i(L(q_i)) = \boldsymbol{\nu_\theta} \big \vert_i
\end{gather}
The renormalized measure for the non-terminating plant $G_i$ is defined to be $\displaystyle \lim_{\theta \rightarrow 0^+}\nu_\theta^i$.
\end{defn}
The following results are retained for the sake of completeness. Complete proofs can be found in \cite{CR06}\cite{C-PhD}.
\vspace{0pt}
\begin{prop}\label{Proposition2.1}
The limiting  measure vector $\mbox{\boldmath
$\nu$}_0\triangleq \lim_{\theta \rightarrow 0^+}
\mbox{\boldmath $\nu$}_\theta$ exists and $||\mbox{\boldmath
$\nu$}_0||_{\infty} \leq 1$.
\end{prop}\vspace{0pt}
\begin{prop} \label{Proposition2.2}
Let $\Pi$ be the stochastic transition matrix of a non-terminating PFSA~\cite{G92,G92-2}.
Then, as the parameter $\theta \rightarrow 0^+$, the limiting
 measure vector is obtained as: $
\boldsymbol{\nu}_0 = \Q\boldsymbol{\chi} $ where the
matrix operator $\displaystyle \Q\triangleq \lim_{k
\rightarrow \infty}
\frac{1}{k}\sum_{j=0}^{k-1}\Pi^{j}$  is the 
Cesaro limit~\cite{BR97,Berman1979} of the
stochastic transition matrix $\Pi$.
\end{prop}\vspace{0pt}
\begin{cor} (to Proposition~\ref{Proposition2.2}) \label{Corollary2.1}
The expression $\Q\boldsymbol{\nu}_\theta$ is independent of
$\theta$. Specifically, the following identity holds for all
$\theta \in (0, 1)$.
\begin{equation}\label{thetaIndependence01}
\Q\boldsymbol{\nu}_\theta = \Q\boldsymbol{\chi}
\end{equation}
\end{cor}\vspace{0pt}
\begin{notn}\label{notpure}
The  linearly independent orthogonal
set $ \{ v^i \in \mathbb{R}^{\Crd(Q)}: v^i_j = \delta_{ij} \}$ is denoted as $\mathcal{B}$ where $\delta_{ij}$ denotes the Kr\"{o}necker delta function.
We note that there is a one-to-one onto mapping between the states $q_i \in Q$ and the elements of $\mathcal{B}$, namely,
\begin{gather}\label{eqpure}
 q_i \mapsto \alpha \iff \alpha_k = \left \{ \begin{array}{ll}
                                              1 & \textrm{if} \ k = i \\
					      0 & \textrm{otherwise}
                                             \end{array}
\right.
\end{gather}
 \end{notn}

\begin{defn}
 For any non-zero vector $v \in \mathbb{R}^{\Crd(Q)}$, the normalizing function $\mathscr{N} :\mathbb{R}^{\Crd(Q)}\setminus\boldsymbol{0} \rightarrow \mathbb{R}^{\Crd(Q)}$ is defined 
as $\mathscr{N}(v) = \frac{v}{\sum_i v_i}$.
\end{defn}
\subsection[Optimal Supervision Problem]{The Optimal Supervision Problem: Formulation \& Solution}\label{subsecformulatesoln}
A supervisor disables a subset of the set $\mathscr{C}$ of
controllable transitions and hence there is a bijection
between the set of all possible supervision policies and the
power set $2^{\mathscr{C}}$. That is, there exists $2^{\vert
\mathscr{C} \vert}$ possible supervisors and each supervisor
is uniquely identifiable with a subset of $\mathscr{C}$ and
the corresponding language measure $\boldsymbol{\nu}_\theta$ allows a quantitative comparison of
different policies.
\begin{defn}\label{superior}
For an unsupervised  plant
\G, let
$G^{\dag}$ and $G^{\ddag}$ be the supervised plants with sets
of disabled transitions, $\mathscr{D}^{\dag}\subseteq
\mathscr{C}$ and $\mathscr{D}^{\ddag}\subseteq \mathscr{C}$,
respectively, whose measures are $\boldsymbol{\nu}^{\dag}$ and
$\boldsymbol{\nu}^{\ddag}$. Then, the supervisor that disables
$\mathscr{D}^{\dag}$ is defined to be superior to the
supervisor that disables $\mathscr{D}^{\ddag}$  if
$\boldsymbol{\nu}^{\dag} \eg \boldsymbol{\nu}^{\ddag}$ and
strictly superior if $\boldsymbol{\nu}^{\dag} \egg
\boldsymbol{\nu}^{\ddag}$.
\end{defn}
\vspace{0pt}
\begin{defn}\label{pdef}{(Optimal Supervision Problem)}
Given a (non-terminating) plant
\G, the
problem is to compute a supervisor that disables a subset
$\mathscr{D}^\star \subseteq \mathscr{C}$, such that $
\forall \mathscr{D}^{\dag} \subseteq \mathscr{C}, \boldsymbol{\nu}^{\star} \eg \boldsymbol{\nu}^{\dag} $ where
$\boldsymbol{\nu}^{\star}$ and $\boldsymbol{\nu}^{\dag}$ are
the measure vectors of the supervised plants $G^{\star}$ and
$G^{\dag}$ under $\mathscr{D}^{\star}$ and $\mathscr{D}^\dag$,
respectively.
\end{defn}
\begin{rem}\label{remtheta}
The solution to the optimal supervision problem  is obtained
in \cite{CR07,C-PhD} by designing an optimal policy for a
\textit{terminating} plant~\cite{G92,G92-2} with a
substochastic transition probability matrix
$(1-\theta)\widetilde{\Pi}$ with $\theta \in (0,1)$. To ensure
that the computed optimal policy coincides with the one for
$\theta = 0$, the suggested algorithm chooses a \textit{small}
value for $\theta$ in each iteration step of the design
algorithm. However, choosing $\theta$ too small may cause
numerical problems in convergence. Algorithm~\ref{Algorithm01} (See Appendix~\ref{appen})
computes the critical lower bound $\theta_\star$ (i.e., how small a $\theta$ is actually required).  In conjunction with
Algorithm~\ref{Algorithm01}, the optimal supervision
problem is solved by use of Algorithm~\ref{Algorithm02} for a generic $PFSA$ as reported in~\cite{CR07}\cite{C-PhD}.
\end{rem}
The following  results in Proposition~\ref{proposition-4.2} are critical to development in the sequel and hence are
presented here without proof. The
complete proofs are available in~\cite{CR07}\cite{C-PhD}.
\begin{prop}\label{proposition-4.2}
\begin{enumerate}
 \item {(Monotonicity) }Let $\boldsymbol{\nu}^{[k]}$ be the language measure vector
computed in the $k^{th}$ iteration of
Algorithm~\ref{Algorithm02}. The measure vectors computed by
the algorithm form an elementwise non-decreasing sequence,
i.e., $\boldsymbol{\nu}^{[k+1]} \eg \boldsymbol{\nu}^{[k]} \
\forall k$.
\item {(Effectiveness) }
Algorithm~\ref{Algorithm02} is an effective
procedure~\cite{HMU01}, i.e., it is guaranteed to terminate.
\item {(Optimality) }
The supervision policy computed by Algorithm~\ref{Algorithm02}
is optimal in the sense of Definition~\ref{pdef}.
\item {(Uniqueness) }
Given an unsupervised plant $G$, the optimal supervisor
$G^{\star}$, computed by Algorithm~\ref{Algorithm02}, is
unique in the sense that it is maximally permissive among all
possible supervision policies with optimal performance. That
is, if $\mathscr{D}^{\star}$ and $\mathscr{D}^{\dag}$ are the
disabled transition sets, and $\boldsymbol{\nu}^{\star}$ and
$\boldsymbol{\nu}^{\dag}$ are the language measure vectors for
$G^{\star}$ and an arbitrarily supervised plant $G^{\dag}$,
respectively, then $ \boldsymbol{\nu}^{\star} \ee
\boldsymbol{\nu}^{\dag} \Longrightarrow \mathscr{D}^{\star}
\subset \mathscr{D}^{\dag}\subseteq \mathscr{C} $
\end{enumerate}
\end{prop}

\begin{defn}\label{defthetamin}
Following Remark~\ref{remtheta}, we note that
Algorithm~\ref{Algorithm01} computes a lower bound for the
critical termination probability for each iteration of
Algorithm~\ref{Algorithm02} such that the disabling/enabling
decisions for the terminating plant coincide with the given
non-terminating  model. We define
\begin{gather}
\theta_{min} = \min_{k} \theta^{[k]}_\star
\end{gather}
where $\theta^{[k]}_\star$ is the termination probability
computed by Algorithm~\ref{Algorithm01} in the $k^{th}$
iteration of Algorithm~\ref{Algorithm02}.
\end{defn}
\vspace{3pt}
\begin{defn}\label{defnustarmeas}
If $G$ and $G^\star$ are the unsupervised and optimally supervised PFSA
respectively then we denote the renormalized measure of the
terminating plant $G^\star(\theta_{min})$ as
$\nu_{\star}^i:2^{L(q_i)} \rightarrow [-1,1]$ (See
Definition~\ref{defrenormmeas}). Hence, in vector notation we
have:
\begin{gather}
 \boldsymbol{\nu_{\star}} = \boldsymbol{\nu_{\theta_{min}}}= \theta_{min}  [I-(1-\theta_{min})\Pi^\star]^{-1} \boldsymbol{\chi}
\end{gather}
where $\Pi^\star$ is the transition probability matrix of the supervised plant $G^\star$.
\end{defn}
\begin{rem}
Referring to Algorithm~\ref{Algorithm02}, it is noted that
$\boldsymbol{\nu_{\star}}=\nu^{[K]}$ where $K$ is the total
number of iterations for Algorithm~\ref{Algorithm02}.
\end{rem}
%
\subsection{The Partial Observability Model}\label{secunobs}
%
The observation model used in this paper is defined by the
so-called  unobservability maps
developed in \cite{CRg07} as a generalization of natural projections in discrete event systems.
It is important to mention that while some authors refer to unobservability as the case 
where no transitions are observable in the system; we use the terms ``unobservable'' and ``partially observable''
interchangeably in the sequel.
The relevant concepts developed in \cite{CRg07} are enumerated in this section for the sake of completeness.
\vspace{3pt}
\subsubsection{Assumptions \& Notations}\label{optip-ass} 
 We make two key
assumptions:
\begin{itemize}
\item The unobservability situation
in the model is specified by a bounded memory
unobservability map $\p$ which is available to the
supervisor. 
\item Unobservable transitions are uncontrollable
\end{itemize}
\begin{defn}\label{defunobs}
An unobservability map $\p : Q \times\Sigma^\star
\longrightarrow
\Sigma^\star$ for a given model
$G=(Q,\Sigma,\delta,\widetilde{\Pi},\boldsymbol{\chi},\mathscr{C})$ is defined recursively as follows:
$\forall q_i \in Q, \sigma_j \in \Sigma \ \textrm{and} \ \sigma_j\omega \in L(q_i)$,
\begin{subequations}
\begin{align}
 &\p(q_i,\sigma_j) &&= \left \{
\begin{array}{cl} \epsilon , & \textrm{if} \ \sigma_j \ \textrm{is unobservable from } q_i \\
\sigma_j , & \textrm{otherwise}
\end{array}\right.\\
&\p(q_i,\sigma_j\omega) &&= \p(q_i,\sigma_j)\p(\delta(q_i,\sigma),\omega)
\end{align}
\end{subequations}
\end{defn}
\vspace{3pt}
 We can indicate transitions to be unobservable in the
graph for the automaton \G  as
\textit{unobservable} and this would suffice for a complete
specification of the unobservability map acting on the plant. The assumption of bounded memory
of the unobservability maps implies that although we may need to unfold the automaton graph to unambiguously indicate the unobservable
transitions; there exists a finite unfolding that suffices for our purpose. Such unobservability maps were referred to as \textit{regular} in \cite{CRg07}.
\begin{rem}
 The unobservability maps considered in this paper are state based as opposed to being event 
based 
observability considered in \cite{RW87}.
\end{rem}
\begin{defn}\label{poobs-compunobs}A string $\omega \in
\Sigma^{\star}$ is called unobservable at the supervisory level
 if at least one of the
events in $\omega$ is unobservable $i.e.$
$ \p(q_i,\omega) \neq \omega $
Similarly, a string $\omega \in
\Sigma^{\star}$ is called completely unobservable if each of
the events in $\omega$ is unobservable $i.e.$
$ \p(q_i,\omega) = \epsilon $
Also, if there are no unobservable strings, we denote the unobservability map $\p$ as trivial.
\end{defn}
The subsequent analysis requires the notion of the phantom
automaton introduced in \cite{CR06a}. The following definition is 
included for the sake of completion.
\begin{defn}\label{defPhantom}
 Given a model $G=(Q,\Sigma,\delta,\widetilde{\Pi},\boldsymbol{\chi},\mathscr{C})$ 
and an unobservability map $\p$, the phantom automaton 
$\mathscr{P}(G)=(Q,\Sigma,\mathscr{P}(\delta),\mathscr{P}(\widetilde{\Pi}),\boldsymbol{\chi},\mathscr{P}(\mathscr{C}))$
is defined as follows:
\begin{subequations}
\begin{align}
&\mathscr{P}(\delta)(q_i,\sigma_j) = \left \{ \begin{array}{ll}
                                              \delta(q_i,\sigma_j)  &, \mathrm{if} \ \p(q_i,\sigma_j) = \epsilon \\
\mathrm{Undefined}&,  \mathrm{otherwise}
                                             \end{array}
\right. \\
&\mathscr{P}(\widetilde{\Pi})(q_i,\sigma_j) = \left \{ \begin{array}{ll}
                                              \widetilde{\Pi}(q_i,\sigma_j)  &, \mathrm{if} \ \p(q_i,\sigma_j) = \epsilon \\
0&,  \mathrm{otherwise}
                                             \end{array}
\right.\\
&\mathscr{P}(\mathscr{C}) = \varnothing \label{eqnunobsuncont}
\end{align}
\end{subequations}
\end{defn}

\begin{rem}
The phantom automata in the sense of  Definition~\ref{defPhantom} is a finite state machine description of 
the language of completely unobservable strings resulting from the unobservability map $\p$ acting on the model
$G=(Q,\Sigma,\delta,\widetilde{\Pi},\boldsymbol{\chi},\mathscr{C})$. Note that Eqn.(\ref{eqnunobsuncont}) is a consequence of the
assumption that unobservable transitions are uncontrollable. Thus 
no transition in the phantom automaton is controllable.
\end{rem}
Algorithm~\ref{AlgoPh} (See Appendix~\ref{appen}) computes the transition
probability matrix for the phantom automaton of a given plant
$G$ under a specified  unobservability map $\p$ by
deleting all {observable} transitions from $G$.
\vspace{5pt}
\subsubsection{The Petri Net Observer}\label{subsecPNO}
For a given model $G=(Q,\Sigma,\delta,\widetilde{\Pi},\boldsymbol{\chi},\mathscr{C})$ and a non-trivial
unobservability map $\p$, it is, in general, impossible to pinpoint the current state from 
an observed event sequence at the supervisory level. However, it is possible to estimate
the set of plausible states from a knowledge of the phantom automaton $\mathscr{P}(G)$.
\begin{defn}\label{defnqw}{(Instantaneous State Description :)}
 For a given plant $G_0=(Q,\Sigma,\delta,\widetilde{\Pi},\boldsymbol{\chi},\mathscr{C})$ initialized at state $q_0 \in Q$
and a non-trivial
unobservability map $\p$,  the instantaneous state description is defined to be the image of an observed event sequence $\omega \in \Sigma^\star$
under the map $\overline{Q} : p(L(G_0)) \longrightarrow 2^Q$ 
as follows:
\begin{gather*}
\overline{Q}(\omega) = \{  q_j \in Q : \exists s \in \Sigma^\star \ \mathrm{s.t.} \ \delta(q_0,s) = q_j \bigwedge \p(q_0,s) = \omega \}
\end{gather*}
\end{defn}
\vspace{3pt}
\begin{rem}
 Note that for a trivial unobservability map $\p$ with $\forall \omega \in \Sigma^\star, \p(\omega) = \omega$, we have
$\overline{Q}(\omega) = \delta(q_0,\omega)$ where $q_0$ is the initial state of the plant. 
\end{rem}

The instantaneous state description $\overline{Q}(\omega)$ can be estimated on-line by constructing a Petri Net observer
with flush-out arcs~\cite{MA98}~\cite{G01}. The advantage of using a Petri net
description is the compactness of representation and simplicity of
the on-line execution algorithm that we present next. Our preference of a Petri net description over a subset construction for finite state machines is motivated by the following:
The Petri
net formalism is natural, due to its ability to model
transitions of the type $ q_1 \rightarrow \mspace{-5mu} \vert \mspace{-1.5mu}
^{\nearrow^{ \boldsymbol{ \ q_2}}}_{\searrow_{\boldsymbol{\ q_3}}}$,
which reflects the condition \textit{"the plant can
possibly be in states $q_2$ or $q_3$ after an observed transition
from $q_1$"}.
One can avoid introducing an exponentially large number of "combined states" of the form $[q_2,q_3]$ as involved in the subset construction and more importantly  preserve the  state description of the underlying plant. Flush-out arcs were introduced by Gribaudo $\etal$~\cite{G01} in
the context of fluid stochastic Petri nets. We apply this notion
to ordinary nets with similar meaning: a flush-out arc is
connected to a labeled transition, which, on firing, removes a
token from the input place (if the arc weight is one). Instantaneous descriptions can be computed
on-line efficiently due to the following result:
\vspace{0pt}
\begin{prop}\label{petriobs}
\begin{enumerate}
\item Algorithm~\ref{Alg2} has polynomial complexity. \item Once the Petri net observer has been computed off line, the current possible states for any observed
sequence can be computed by executing Algorithm~\ref{Alg3}
on-line:
\end{enumerate}
\end{prop}
\vspace{3pt}\begin{proof} 
Given in \cite{CRg07}.
\end{proof}\vspace{0pt}
%
\section{Online Implementation of Measure-theoretic Optimal Control under Perfect Observation}\label{seconlineS}
This section  devises an online implementation scheme for the language measure-theoretic optimal control 
algorithm 
which will be later extended  to handle plants with non-trivial unobservability maps.
%
Formally, a supervision policy $S$ for a given plant \G
specifies the control in the terms of disabled controllable
transitions at each state $q_i \in Q$ $i.e.$ $S = (G,\phi)$
where
\begin{gather}\label{eqphisup}
\phi : Q \longrightarrow \{0,1\}^{Card(\Sigma)}
\end{gather}
The map $\phi$ is referred to in the literature as the state
feedback map~\cite{RW87} and it specifies the set of disabled
transitions as follows: If at state $q_i \in Q$, events
$\sigma_{i_1},\sigma_{i_r}$ are disabled by the
particular supervision policy, then $\phi(q_i)$ is a binary
sequence on $\{0,1\}$  of length equal to the cardinality of
the event alphabet $\Sigma$ such that
\begin{gather*}
\begin{array}{ccccccc}\mspace{205mu}
\left \downarrow \begin{array}{cc} i_1^{th} \ \textbf{\small
element}
\end{array}
\right. \boldsymbol{\cdots} \mspace{30mu}\left \downarrow \begin{array}{cc}
i_r^{th} \ \textbf{\small element}
\end{array} \right.
\\ \phi(q_i) = \left [
\begin{array}{ccccccccccccc} 0 & \cdots & 1 & \cdots & 0 & \cdots
& 0 &\mspace{0mu} 1 & \cdots
\end{array}\right ]
\end{array}
\end{gather*}
\vspace{-3pt}
\begin{rem}
If it is possible to partition the alphabet $\Sigma$ as
$\Sigma = \Sigma^c \bigsqcup \Sigma^{uc}$, where $\Sigma^c$ is
the set of controllable transitions and $\Sigma^{uc}$ is the
set of uncontrollable transitions, then it suffices to
consider $\phi$ as a map $\phi:Q \longrightarrow
\{0,1\}^{Card(\Sigma^c)}$. However, since we consider
controllability to be state dependent ($i.e.$ the possibility
that an event is controllable if generated at a state $q_i$
and uncontrollable if generated at some other state $q_j$),
such a partitioning scheme is not feasible.
\end{rem}
%
Under perfect observation, a computed supervisor $(G,\phi)$ responds
to the report of a generated event as follows:
\begin{itemize}
\item The current state of the plant model is computed as
$q_{current} = \delta(q_{last},\sigma)$, where $\sigma$ is the
reported event and $q_{last}$ is the state of the plant model
before the event is reported. 
\item All events
specified by $\phi(q_{current})$ is disabled.
\end{itemize}
Note that such an approach requires the supervisor to remember
$\phi(q_i) \forall q_i \in Q$, which is equivalent to keeping
in memory a $n \times m$ matrix, where $n$ is the number of
plant states and $m$ is the cardinality of the event alphabet.
We show that there is a alternative simpler implementation.
\begin{algorithm}[!htb]
 \small \SetLine
  \SetKwData{Left}{left}
  \SetKwData{This}{this}
  \SetKwData{Up}{up}
  \SetKwFunction{Union}{Union}
  \SetKwFunction{FindCompress}{FindCompress}
    \SetKw{Tr}{true}
   \SetKw{Tf}{false}
  \SetKwInOut{Input}{input}
  \SetKwInOut{Output}{output}
  \caption{Online Implementation of Optimal Control}\label{Algoonlineopt}
\Input{\G,$\p$, Initial state $q_0$}
\Output{Optimal Control Actions} \Begin{
Compute $G^{opt}$ by $G\xrightarrow[\mathscr{A}_O]{}G^{opt}$\;
Set $\theta_{\star\star} = \min \theta_\star$\tcc*[r]{Min.
$\theta_\star$ for all iterations}
Set $\boldsymbol{\mu} =
\boldsymbol{\mu}^{G^{opt}_{\theta_{\star\star}}}$\;Set
$q_{current} = q_0$\tcc*[r]{initial state}
\While(\tcc*[f]{Infinite Loop}){\Tr} { Observe event
$\sigma_j$\tcc*[r]{Perfect Observation} Compute $q_{current} =
\delta(q_{current},\sigma_j)$\; \For(\tcc*[f]{$m$ =
Cardinality of $\Sigma$}){$k=1$ \textbf{to} $m$ } {
Compute $q_{next} = \delta(q_{current},\sigma_k)$\;
\eIf(\tcc*[f]{If $q_{Test} == q_j$ then
$\boldsymbol{\mu}(q_{Test}) =\mu_j$
}){$(q_{current},\sigma_k,q_{next})\in \mathscr{C}$}{
\If(\tcc*[f]{If $q_{current} == q_i$ then
$\boldsymbol{\mu}(q_{current}) =\mu_i$ }){ $
\boldsymbol{\mu}(q_{Test}) \geqq \boldsymbol{\mu}(q_{current})
$}{ Disable $\sigma_k$\;}}{ Enable $\sigma_k$\;}
    }
   }
  }
\end{algorithm}
\begin{lem}\label{lemremmeas}
For a given finite state plant \G  and the corresponding 
optimal language measure $\boldsymbol{\nu_\star}$,
 the pair
$(G,\boldsymbol{\nu_\star})$
completely specifies the optimal supervision policy.
\end{lem}
\vspace{0pt}
\begin{proof}
The optimal configuration
$G^\star$ is characterized as follows~\cite{CR07,C-PhD}:
\begin{itemize}
\item if for states $q_i,q_j \in Q$,
$\boldsymbol{\nu_\star}\big \vert_i >
\boldsymbol{\nu_\star}\big \vert_j$, then all controllable
transitions $q_i \xrightarrow[q_i]{} q_j$ are disabled. \item
if for states $q_i,q_j \in Q$,
$\boldsymbol{\nu_\star}\big \vert_i \leqq
\boldsymbol{\nu_\star}\big \vert_j$, then all controllable
transitions $q_i \xrightarrow[q_i]{} q_j$ are enabled.
\end{itemize}
It follows that if the supervisor has access to the
unsupervised plant model $G$ and the language measure vector
$\boldsymbol{\nu_\star}$, then the
optimal policy can be implemented by the following procedure:
\begin{enumerate}
\item Compute the current state of the plant model as
$q_{current} = \delta(q_{last},\sigma)$, where $\sigma$ is the
reported event and $q_{old}$ is the state of the plant model
before the event is reported. Let $q_{current} = q_i$. \item
Disable all controllable transitions $q_i
\xrightarrow[\sigma_j]{} q_k$ if
$\boldsymbol{\nu_\star}\big \vert_i >
\boldsymbol{\nu_\star}\big \vert_k$ for all $q_k \in Q$.
\end{enumerate}
This completes the proof. The procedure is summarized in
Algorithm~\ref{Algoonlineopt}.
\end{proof}
\vspace{0pt}

The approach
given in Lemma~\ref{lemremmeas} is important from the
perspective that it forms the intuitive basis for extending the
optimal control algorithm derived under the assumption of
perfect observation 
to situations where one or more transitions are unobservable
at the supervisory level.
%
\section{Optimal Control under Non-trivial Unobservability}\label{optip-intro}
This section makes
use of the  unobservability analysis presented in
Section~\ref{secunobs} to derive a modified online-implementable
control algorithm for partially observable probabilistic finite state plant
models. 
%
%
%
\subsection{The Fraction Net Observer}\label{secinstmeas}
In Section~\ref{secunobs} the notion of instantaneous description
of  was introduced
as a map $\overline{Q} : p(L(G_i)) \longrightarrow 2^Q$  from
the set of observed event traces to the power set of the state set $Q$, such that given an observed event
trace $\omega$, $\overline{Q}(\omega) \subseteq Q$ is the set
of states that the underlying deterministic finite state plant
can possibly occupy at the given instant. 
We constructed a Petri Net observer (Algorithm~\ref{Alg2}) and showed that the instantaneous
description can be computed online with polynomial complexity.
However, for a plant modeled by a probabilistic regular
language, the knowledge of the event occurrence probabilities
allows us not only to compute the set of possible current
states ($i.e.$ the instantaneous description) but also the
probabilistic cost of ending up in each state in the instantaneous
description. To achieve this objective, we modify the Petri
Net Observer introduced in Section~\ref{subsecPNO} by assigning
(possibly) fractional weights computed as functions of the
event occurrence probabilities to the input arcs. The output
arcs are still given unity weights. In the sequel, the Petri
Net observer with possibly fractional arc weights is referred
to as the \textbf{Fraction Net Observer} (FNO). 
\vspace{0pt}
First we need to
formalize the notation for the Fraction Net observer.
\vspace{3pt}
\begin{defn}\label{notfno}
Given a finite state terminating plant model \Gt, and an
unobservability map $\p$, the Fraction Net observer (FNO), denoted as $\mathscr{F}_{(G_\theta,\p)}$, is
a labelled Petri Net
$(Q,\Sigma,A^{\mathcal{I}},A^{\mathcal{O}},w^{\mathcal{I}},
x^0)$
 with fractional arc weights and possibly fractional markings,
 where $Q$ is the set of places, $\Sigma$ is the event label
 alphabet, $A^{\mathcal{I}}\subseteqq Q \times \Sigma \times Q$ and
 $A^{\mathcal{O}}\subseteqq Q\times \Sigma$ are the sets of input and output arcs,
 $w^{\mathcal{I}}$ is the 
 input  weight assignment function
and $x^0 \in \mathcal{B}$ (See Notation \ref{notpure}) is the
 initial marking. The output arcs are defined to have unity weights.
\end{defn}
\vspace{0pt}
The algorithmic construction of a FNO is derived next. We assume that the Petri Net observer has already
been computed (by Algorithm~\ref{Alg2}) with $Q$ the set of
places, $\Sigma$ the set of transition labels,
$A^{\mathcal{I}} \subseteqq Q \times \Sigma \times Q$ the set
of input arcs and $A^{\mathcal{O}} \subseteqq Q \times \Sigma$
the set of output arcs. \vspace{0pt}

\begin{defn}\label{defnwtinput}
The input weight assigning function
$w^{\mathcal{I}}:A^{\mathcal{I}} \longrightarrow (0,\infty)$
for the Fraction Net observer  is defined as :
\begin{align*}
&\forall q_i \in Q, \forall \sigma_j \in \Sigma, \forall q_k
\in Q, \nonumber \\ & \delta(q_i,\sigma_j) = q_\ell
\Longrightarrow w^{\mathcal{I}}(q_i,\sigma_j,q_k) =
\mspace{-50mu}
\sum_{\begin{subarray}{c}\omega \in \Sigma^\star \ \textrm{s.t.} \\
\delta^\star(q_\ell,\omega)=q_k \bigwedge p (q_\ell,\omega) =
\epsilon \end{subarray}} \mspace{-50mu}
(1-\theta)^{\vert \omega \vert}\tilde{\pi}(q_\ell,\omega)
\end{align*}
where  $\delta:Q \times \Sigma \rightarrow Q$ is the
transition map of the underlying DFSA and $\p$ is the given
unobservability map and $\tilde{\pi}$ is the event
cost ($i.e.$ the occurrence probability) function \cite{R05}.
It follows that the weight on an input arc from transition
$\sigma_j$ (having an output arc from place $q_i$) to place
$q_k$ is the sum total of the conditional probabilities of all completely unobservable paths by which the
underlying plant can reach the state $q_k$ from state $q_\ell$
where $q_\ell=\delta(q_i,\sigma_j)$.
\end{defn}
\vspace{3pt}
Computation of the input arc weights for the
Fraction Net observer requires the notion of the phantom
automaton (See Definition~\ref{defPhantom}).  The
computation of the arc weights for the FNO is  summarized
in Algorithm~\ref{AlgoFNOarcwt}.
\begin{prop}\label{propinputarcw}
Given a Petri Net observer
$(Q,\Sigma,A^{\mathcal{I}},A^{\mathcal{O}})$, the event
occurrence probability matrix $\widetilde{\pi}$ and the
transition probability matrix for the phantom automaton
$\mathscr{P}(\Pi)$, Algorithm~\ref{AlgoFNOarcwt}
computes the arc weights for the fraction net observer as
stated in Definition~\ref{defnwtinput}.
\end{prop}
\vspace{0pt}
\begin{proof}
Algorithm~\ref{AlgoFNOarcwt} employs the following identity to
compute input arc weights:
\begin{align*}
&\forall q_i \in Q, \forall \sigma_j \in \Sigma, \forall q_k
\in Q,\nonumber\\ &w^{\mathcal{I}}(q_i,\sigma_j,q_k) = \left
\{\begin{array}{ll}\Big [ \mathbb{I} - (1-\theta)\mathscr{P}
(\Pi) \Big ]^{-1} \bigg \vert_{\ell k}, & \\ \mspace{60mu}
\textrm{if}
 \ (q_i,\sigma_j,q_k) \in A^{\mathcal{I} } \wedge \delta(q_i,\sigma_j) = q_{\ell} & \\
0, & \\ \mspace{60mu}\textrm{otherwise} & \end{array}\right.
\end{align*}
which follows from the following argument. Assume that for the
given unobservability map $\p$, $G^{\mathscr{P}}$ is the
phantom automaton for the underlying plant $G$. We observe that the measure of the
language of all strings initiating from state $q_\ell$ and
terminating at state $q_k$ in the phantom automaton
$G^{\mathscr{P}}$ is given by
$\Big [ \mathbb{I} -
\mathscr{P} (\Pi) \Big ]^{-1} \bigg \vert_{\ell
k}$.
Since every string generated by the phantom automaton is
completely unobservable (in the sense of
Definition~\ref{defPhantom}), we
conclude
\begin{gather}
\Big [ \mathbb{I} - (1-\theta)\mathscr{P} (\Pi) \Big ]^{-1}
\bigg \vert_{\ell k} = \mspace{-20mu}
\sum_{\begin{subarray}{c}\omega \in \Sigma^\star \ \textrm{s.t.} \\
\delta^\star(q_\ell,\omega)=q_k \bigwedge p (q_\ell,\omega) =
\epsilon \end{subarray}} \mspace{-50mu}
(1-\theta)^{\vert \omega \vert}\tilde{\pi}(q_\ell,\omega)
\end{gather}
This completes the proof.
\end{proof}
\vspace{0pt}
\begin{algorithm}[t]
\footnotesize
  \SetLine
  \SetKwData{Left}{left}
  \SetKwData{This}{this}
  \SetKwData{Up}{up}
  \SetKwFunction{Union}{Union}
  \SetKwFunction{FindCompress}{FindCompress}
  \SetKwInOut{Input}{input}
  \SetKwInOut{Output}{output}
  \SetKw{Tr}{true}
   \SetKw{Tf}{false}
  \caption{Computation of Arc Weights for FNO}\label{AlgoFNOarcwt}
\Input{Petri Net Observer $(Q,\Sigma,A^{\mathcal{I}},A^{\mathcal{O}})$,
Event Occurrence probability Matrix $\widetilde{\pi}$,
$\mathscr{P}(\Pi)$} \Output{$w^{\mathcal{I}}$,
$w^{\mathcal{O}}$} \Begin{
\tcc*[h]{Computing Weights for Input Arcs}\\
\For{$i=1$ \textbf{to} $n$}{
\For{$j=1$ \textbf{to} $m$}{
\For{$k=1$ \textbf{to} $n$}{
\If{$(q_i,\sigma_j,q_k) \in A^{\mathcal{I} }$}{
Compute $q_\ell = \delta(q_i,\sigma_j)$\;
$w^{\mathcal{I}}(q_i,\sigma_j,q_k) = \Big [ \mathbb{I} -
\mathscr{P} (\Pi) \Big ]^{-1} \bigg \vert_{\ell
k}$\;
}}} }
}
\end{algorithm}
In the Section~\ref{subsecPNO}, we presented Algorithm~\ref{Alg3}
to compute the Instantaneous State Description $\overline{Q}(\omega)$ online
without referring to the transition probabilities. The approach consisted of
firing all enabled transitions (in the Petri Net observer)
labelled by $\sigma_j$ on observing the event $\sigma_j$ in
the underlying plant. The set of possible current states then
consisted of all states which corresponded to places with one
or more tokens. For the Fraction Net observer we use a
slightly different approach which involves computation of a set
of  event-indexed state transition matrices.
\vspace{0pt}
\begin{defn}\label{defnevindM}
For a Fraction Net observer
$(Q,\Sigma,A^{\mathcal{I}},A^{\mathcal{O}},w^{\mathcal{I}},x^0)$
the set of event-indexed state transition matrices
$\boldsymbol{\Gamma} = \{\Gamma^{\sigma_j} : \sigma_j \in
\Sigma \}$ is a set of $m$  matrices each of dimension $n
\times n$ (where $m$ is the cardinality of the event alphabet
$\Sigma$ and $n$ is the number of places), such that on
observing event $\sigma_j$ in the underlying plant, the
updated marking $x^{[k+1]}$ for the FNO (due to firing of all
enabled $\sigma_j$-labelled transitions in the net) can be
obtained from the existing marking $x^{[k]}$ as follows:
\begin{gather}
x^{[k+1]} = x^{[k]}\Gamma^{\sigma_j}
\end{gather}
\end{defn}
\vspace{0pt} The procedure for computing $\boldsymbol{\Gamma}$
is presented in Algorithm~\ref{AlgGammaFNO}. Note that the
only inputs to the algorithm are the transition matrix for the
phantom automaton, the unobservability map $\p$ and the
transition map for the underlying plant model. The next
proposition shows that the algorithm is correct.
\begin{algorithm}[t]
 \footnotesize \SetLine
  \SetKwData{Left}{left}
  \SetKwData{This}{this}
  \SetKwData{Up}{up}
  \SetKwFunction{Union}{Union}
  \SetKwFunction{FindCompress}{FindCompress}
  \SetKwInOut{Input}{input}
  \SetKwInOut{Output}{output}
  \caption{Derivation of Transition Matrices $\Gamma^{\sigma_j}$}\label{AlgGammaFNO}
\Input{$\mathscr{P}(\Pi)$, $\delta$, $\p$}
\Output{$\Gamma^{\sigma_j} \ \forall  \sigma_j  \in \Sigma$}
\Begin{
      \For(\tcc*[f]{$m$ = No. of events}){$j  \ \in \ \{1,\cdots,m\} \ $} {
      \For(\tcc*[f]{$n$ = No. of places}){$i  \ \in \ \{1,\cdots,n\} \ $} {
\eIf{ $ \delta(q_i,\sigma_j) \ \textrm{is undefined} \
\textbf{OR} \ p(q_i,\sigma_j)=\epsilon $}{$\textrm{Set} \
i^{th} \ \textrm{row of} \ \Gamma^j =
[0,\cdots,0]^T$\;}{$\textrm{Compute} \ \boldsymbol{r} =
\delta(q_i,\sigma_j) \ $\; $\textrm{Set} \ i^{th} \
\textrm{row of} \ \ \Gamma^j =  \boldsymbol{r}^{th}
\ \textrm{row of} \ [\mathbb{I} -\mathscr{P}(\Pi)
]^{-1}  $\; }
    }
   }
  }
\end{algorithm}
\vspace{0pt}
\begin{prop}\label{proptransM}
Algorithm~\ref{AlgGammaFNO}  correctly computes the set of event-indexed
transition matrices $\boldsymbol{\Gamma} = \{\Gamma^{\sigma_j}
: \sigma_j \in \Sigma \}$ for a given fraction net observer
$(Q,\Sigma,A^{\mathcal{I}},w^{\mathcal{I}},x^0)$ in the sense
stated in Definition~\ref{defnevindM}.
\end{prop}
\vspace{0pt}
\begin{proof}
Let the current marking of the Fraction Net observer specified
as $(Q,\Sigma,A^{\mathcal{I}},$
$A^{\mathcal{O}},w^{\mathcal{O}},w^{\mathcal{I}})$ be denoted
by $x^{[k]}$ where $x^{[k]} \in [0,\infty)^n$ with $n =
Card(Q)$. Assume event $\sigma_j \in \Sigma$ is observed in
the underlying plant model. To obtain the updated marking of
the Fraction Net observer, we need to fire all transitions
labelled by $\sigma_j$ in the FNO. Since the graph of the FNO
is identical with the graph of the Petri Net observer
constructed by Algorithm~\ref{Alg2}, it
follows that if $\delta(q_i,\sigma_j)$ is undefined or  the
event $\sigma_j$ is unobservable from the state $q_i$ in the
underlying plant, then there is a flush-out arc to a
transition labelled $\sigma_j$ from the place $q_i$ in the
graph of the Fraction Net observer. This implies that the
content of place $q_i$ will be flushed out and hence will not
contribute to any place in the updated marking $x^{[k+1]}$
i.e.
\begin{gather}
x^{[k]}_i\Gamma^{\sigma_j}_{i\ell} = 0 \forall \ i \in
\{1,\cdots,n\}
\end{gather}
implying that the $i^{th}$ column of the matrix
$\Gamma^{\sigma_j}$ is $[0,\cdots,0]^T$. This justifies Line 5
of Algorithm~\ref{AlgGammaFNO}.If $\sigma_j$ is defined and
observable from the state $q_i$ in the underlying plant, then
we note that the contents of the place $q_i$ end up in all
places $q_\ell \in Q$ such that there exists an input arc
$(q_i,\sigma_j,q_\ell)$ in the FNO. Moreover, the contribution
to the place $q_\ell$ coming from place $q_i$ is weighted by
$w^{\mathcal{I}}(q_i,\sigma_j,q_\ell)$. Denote this
contribution by $c_{i\ell}$. Then we have
%
%
%
%
%
\begin{align}
&c_{i\ell} = w^{\mathcal{I}}(q_i,\sigma_j,q_\ell)x^{[k]}_i\nonumber\\
\Longrightarrow & \sum_ic_{i\ell} =
\sum_iw^{\mathcal{I}}(q_i,\sigma_j,q_\ell)x^{[k]}_i \nonumber\\
\Longrightarrow & x^{[k+1]}_\ell =
\sum_iw^{\mathcal{I}}(q_i,\sigma_j,q_\ell)x^{[k]}_i
\end{align}
Note that $ \sum_ic_{i\ell}=x^{[k+1]}_\ell$ since
contributions from all places to $q_\ell$ sum to the value of
the updated marking in the place $q_\ell$. Recalling from
Proposition~\ref{propinputarcw}, that
\begin{gather}
w^{\mathcal{I}}(q_i,\sigma_j,q_\ell) = \Big [ \mathbb{I} - (1-\theta)
\mathscr{P} (\Pi) \Big ]^{-1} \bigg \vert_{r\ell}
\end{gather}
where $q_r = \delta(q_i,\sigma_j)$ in the underlying plant,
the result follows.
\end{proof}
\vspace{0pt}
Proposition~\ref{proptransM} allows an alternate computation of the
Instantaneous State Description. We assume that the initial state
of the underlying plant is known and hence the initial marking
for the FNO is assigned as follows:
\begin{gather}
x^{[0]}_i = \left \{ \begin{array}{ll} 1 & \textrm{if} \ q_i \
\textrm{is  the initial state} \\
0 & \textrm{otherwise}
\end{array}\right.
\end{gather}
It is important to note that since the underlying plant is a
deterministic finite state automata (DFSA) having only one
initial state, the initial marking of the Fraction Net
observer has only one place with value 1 and all remaining
places are empty. It follows from
Proposition~\ref{proptransM}, that for a given initial marking
$x^{[0]}$ of the FNO,  the marking after observing a string
$\omega = \sigma_{r_1}\cdots\sigma_{r_k}$ where $\sigma_{j}\in
\Sigma$, is obtained as:
\begin{gather}
x^{[k]} = 
x^{[0]}\prod_{j=r_1}^{j=r_k}\Gamma^{\sigma_j}
\end{gather}
Referring to the notation for instantaneous description
introduced in Definition~\ref{defnqw}, we have
\begin{gather}
\overline{Q}(\omega) = \big \{q_i \in Q : x^{[ \vert \omega \vert ] }_i > 0 \big
\}
\end{gather}\vspace{-15pt}
\begin{rem}
We observe that to solve the State Determinacy problem, we
only need to know if  the individual marking values are
non-zero. The specific values of the entries in the marking
$x^{[k]}$  however allow us to estimate the cost of occupying
individual states in the instantaneous description
$\overline{Q}(\omega)$.
\end{rem}
\vspace{-5pt}
\subsection{State Entanglement Due to Partial Observability}\label{subsecentanglement}
The  markings of the FNO $\mathscr{F}_{(G_\theta,\p)}$ for the plant \Gt in case of perfect observation is of the
following form:
\begin{gather*}
\forall k \in \mathbb{N}, \ x^{[k]} = [ 0 \cdots 0 \ 1 \ 0 \cdots 0]^T \ i.e. \ x^{[k]} \in \mathcal{B} \ (\mathrm{See \ Notation~\ref{notpure}})
\end{gather*}
It follows that for a perfectly observable system,  $\mathcal{B}$ is an enumeration of the
state set $Q$ in the sense
$x^{[k]}_i = 1$ implies that the current
state is $q_i\in Q$. Under a non-trivial unobservability
map $\p$, the set of all possible FNO markings proliferates and 
 we can interpret $x^{[k]}$ after the $k^{th}$ observation instance as the current states of the
observed dynamics. This follows from the fact that no previous knowledge beyond that
of the current FNO marking $x^{[k]}$ is required to define the future evolution
of $x^{[k]}$. The effect
of partial observation can then be interpreted as adding new
states to the model with each new state a linear combination
of the underlying states enumerated in $\mathcal{B}$. 

Drawing an analogy with 
the phenomenon of state entanglement in quantum mechanics, we refer to $\mathcal{B}$
as the set of \textit{pure} states; while all other occupancy estimates that may appear are
referred to as \textit{mixed} or \textit{entangled} states. Even for a finite state plant model, the
cardinality of the set of all possible entangled states is not guaranteed to be finite. 
\begin{lem}\label{lemobsstate}
Let $\mathscr{F}_{(G_\theta,\p)}$ with initial marking $x^{[0]} \in \mathcal{B}$ be the FNO for 
the underlying terminating plant \Gt with uniform termination probability $\theta$. Then
for any observed string $\omega = \sigma_{r_1}\cdots\sigma_{r_s} $ of length $s \in \mathbb{N}$
with $\sigma_{r_j}\in
\Sigma \ \forall r_j \in \{1,\cdots,k\}$, the occupancy estimate $x^{[k]}$, after occurrence of the
$k^{th}$ observable transition,  satisfies:
\begin{subequations}
\begin{gather}
 x^{[k]} \in \left [0,\frac{1}{\theta} \right ]^{\Crd(\Sigma)}\setminus \boldsymbol{0} \label{eq35a}
\end{gather}
 \end{subequations}
\end{lem}

\vspace{3pt}
\begin{proof}
Let the initial marking $x^{[0]} \in \mathcal{B}$ be given by 
\begin{gather}
 \begin{array}{cccc}[0 \cdots &1 &\cdots 0]\\
  \mathrm{(i^{th} \ element)} & \uparrow &
 \end{array}
\end{gather}
 Elementwise non-negativity of $x^{[k]}$ for all $k \in \mathbb{N}$ follows from the fact that $x^{[0]} \in \mathcal{B}$ is elementwise non-negative and  each 
$\Gamma^{\sigma}$ is a non-negative matrix for all $\sigma \in \Sigma$. We also need to show that $x^{[k]}$ cannot be the zero vector. The argument is as follows:
Assume if possible $x^{[\ell]}\Gamma^{\sigma} = \boldsymbol{0}$ where $x^{[\ell]} \neq \boldsymbol{0}$ and $\sigma \in \Sigma$ is the current observed event. It follows from the construction of the transition matrices that $ \forall q_i \in Q, x^{[\ell]}_i \neq 0 $ implies that either $ \delta(q_i,\sigma) $ is undefined or $\p(q_i,\sigma) = \epsilon$. In either case, it is impossible
to observe the event $\sigma$ with the current occupancy estimate $x^{[\ell]}$ which is a contradiction.  Finally, we need to prove the
elementwise upper bound of $\frac{1}{\theta}$ on $x^{[k]}$. We note that that $x^{[k]}_j$ is the sum total of the conditional probabilities of all strings $u \in \Sigma^\star$
initiating from state $q_i \in Q$ (since $\forall j, x^{[0]}_j = \delta_{ij} $) that terminate on the state $q_j \in Q$ and satisfy
\begin{gather}
 \p(u) = \omega
\end{gather}
It follows that $x^{[k]}_j \leqq x^{[0]}[ \mathbb{I} - (1-\theta)\Pi  ]^{-1}\big \vert_j$ since the righthand-side is the sum of the conditional probabilities of all strings
that go to $q_j$ from $q_i$ irrespective of observability. Hence we conclude:
\begin{gather*}
 \vert \vert x^{[k]} \vert \vert_\infty \leqq \vert \vert x^{[0]} [ \mathbb{I} - (1-\theta)\Pi  ]^{-1}\vert \vert_\infty \leqq 1 \times \frac{1}{\theta}
\end{gather*}
which completes the proof.

\end{proof}
\begin{rem}\label{remcompact}
It follows from Lemma~\ref{lemobsstate} that the entangled states belong to a compact subset of $\mathbb{R}^{\Crd(Q)}$.
\end{rem}
\vspace{3pt}
\begin{defn}\label{defentstate}
 (Entangled State Set:) For a given \G and $\p$, the entangled state set $Q_\mathscr{F} \subset  \mathbb{R}^{\Crd(Q)} \setminus \boldsymbol{0}$ is 
 the set of all possible markings of the FNO initiated at any of the
pure states $x^{[0]} \in \mathcal{B}$. 
\end{defn}
\subsection{An Illustrative Example of State Entanglement}
We consider the plant model as presented in the lefthand plate of 
Figure~\ref{optip-figex1}. The finite state plant model with
the unobservable transition (marked in red dashed) along with
the constructed Petri net observer is shown in
Figure~\ref{optip-figex1}. The event occurrence probabilities
assumed are shown in Table~\ref{optip-tabpit} and the transition probability matrix \textbf{P} is shown in
Table~\ref{optip-tabpi}. Given $\theta=0.01$, we apply Algorithm~\ref{AlgoPh} to obtain:
\begin{gather}
\big [ \mathbb{I} - (1-\theta)\mathscr{P}(\Pi) \big ]^{-1} = \left [
\begin{array}{cccc} 1 & 0.2 & 0 & 0 \\0 & 1 &
0 & 0 \\0 & 0 & 1 & 0 \\0 & 0 & 0 & 1
\end{array} \right ]
\end{gather}
\begin{table}[!h]
\begin{minipage}{1.75in}
\caption{Event Occurrence Probabilities}\label{optip-tabpit}
\hspace{15pt}\begin{tabular}{||c|ccc|}
\hline & e &r & a \\ \hline \hline $00$ & $0.2$ & $0.8$ & $0$
\\\hline $01$ & $0.2$
& $0.5$ & $0.3$ \\ \hline $11$ & $0.6$ & $0.4$ & $0$
\\ \hline $10$ & $0.3$ & $0.5$
& $0.2$ \\ \hline
\end{tabular}
\end{minipage}
\begin{minipage}{1.6in}
 \caption{Transition Probability Matrix $\Pi$}\label{optip-tabpi}
\hspace{-5pt}\begin{tabular}{||c|cccc|}
\hline & $00$ & $01$ & $11$ & $10$ \\ \hline \hline $00$ &
$0.8$ & $0.2$ & $0$ & 0
\\\hline $01$ & $0.5$
& $0.2$ & $0.3$ & 0 \\ \hline $11$ & 0 & 0 &$0.6$ & $0.4$
\\ \hline $10$ & $0.2$ & $0$
& $0.3$ & 0.5 \\ \hline
\end{tabular}
\end{minipage}
\end{table}\vspace{0pt}
\begin{figure}[t]
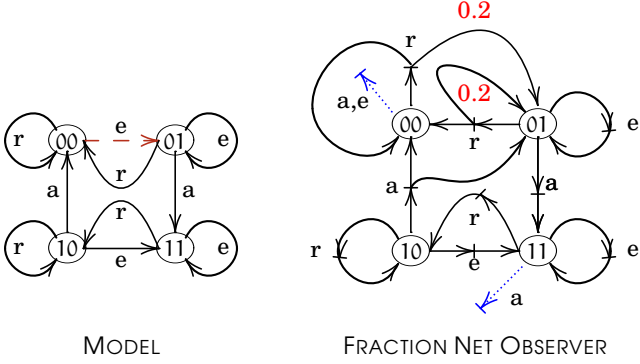

\begin{center}
\begin{minipage}{3.25in}\vspace{-10pt}
\xy (-2,0)*{ \xy 0;/r.17pc/:
(-40,0)*+[o][F]{00};(-20,0)*+[o][F-]{01};
(-20,-20)*+[o][F-]{11};(-40,-20)*+[o][F-]{10};
{\myar@{-->}^{\txt{e}}@[BrickRed](-37,0)*{};(-23,0)*{}};
{\myar@{->}^{\txt{a}}(-20,-2)*{};(-20,-17)*{}};
{\myar@{->}_{\txt{e}}(-37,-20)*{};(-23,-20)*{}};
{\myar@{->}^{\txt{a}}(-40,-17)*{};(-40,-2)*{}};
{\myar@{<-}^{\txt{r}}@/_1.5pc/(-37,0)*{};(-23,0)*{}};
{\myar@{<-}_{\txt{r}}@/^1.5pc/(-37,-20)*{};(-23,-20)*{}};
(-42,2)*{}="X1";(-42,-2)*{}="X2"; "X1";"X2"
**[thicker]\crv{(-45,8) & (-58,0) & (-45,-8) };
{\myar@{{}{}{>}}(-43.75,-4)*{};(-42,-2)*{}};
(-18,2)*{}="X1";(-18,-2)*{}="X2"; "X1";"X2"
**[thicker]\crv{(-15,8) & (-2,0) & (-15,-8) };
{\myar@{{}{}{>}}(-16.25,-4)*{};(-18,-2)*{}};
(-42,-18)*{}="X1";(-42,-22)*{}="X2"; "X1";"X2"
**[thicker]\crv{(-45,-12) & (-58,-20) & (-45,-28) };
{\myar@{{}{}{>}}(-43.75,-24)*{};(-42,-22)*{}};
(-18,-18)*{}="X1";(-18,-22)*{}="X2"; "X1";"X2"
**[thicker]\crv{(-15,-12) & (-2,-20) & (-15,-28) };
{\myar@{{}{}{>}}(-16.25,-24)*{};(-18,-22)*{}};
(-49,0)*{\txt{r}};(-11,0)*{\txt{e}};
(-49,-20)*{\txt{r}};(-11,-20)*{\txt{e}};
\endxy};
(45,5)*{\xy 0;/r.2pc/:
(-40,0)*+[o][F]{00};(-20,0)*+[o][F-]{01};
(-20,-20)*+[o][F-]{11};(-40,-20)*+[o][F-]{10};
{\myar@[blue]@{{}{.}{>|}}^{\txt{a,e}}(-43,2)*{};(-48,8)*{}};
{\myar@[blue]@{{}{.}{>|}}^{\txt{a}}(-22,-22)*{};(-29,-29)*{}};
{\myar@{->}^{\txt{a}}(-20,-2)*{};(-20,-17)*{}};
{\myar@{->}^{\txt{a}}(-20,-2)*{};(-20,-17)*{}};
{\myar@{->}_{\txt{e}}(-37,-20)*{};(-23,-20)*{}};
{\myar@{{}{-}{>|}}(-40,-17)*{};(-40,-10)*{}};
{\myar@{{}{-}{>}}(-40,-10)*{};(-40,-2)*{}};
{\myar@{{}{-}{>}}(-22.75,-3)*{};(-22,-2)*{}};
(-40,-10)*{}="X1";(-22,-2)*{}="X2"; "X1";"X2"
**[thicker]\crv{(-37,-5) & (-30,-15) };
{\myar@{{}{-}{>}}(-30,0)*{};(-37,0)*{}};
{\myar@{{}{-}{>|}}(-23,0)*{};(-30,0)*{}};
(-30,0)*{}="X1";(-22,2)*{}="X2"; "X1";"X2"
**[thicker]\crv{(-40,10) & (-30,10) & (-25,5) };
{\myar@{{}{}{>}}(-23,3)*{};(-22,2)*{}};
{\myar@{<-}@/^1.8pc/(-37,-20)*{};(-23,-20)*{}};
(-40,9.5)*{}="X1";(-42,-2)*{}="X2"; "X1";"X2"
**[thicker]\crv{(-47,18) & (-65,0) & (-45,-8) };
{\myar@{{}{}{>}}(-43.75,-4)*{};(-42,-2)*{}};
{\myar@{{}{-}{>|}}(-40,3)*{};(-40,9)*{}};
{\myar@/^1.8pc/@{{}{-}{>}}(-40,9)*{};(-20,3)*{}};
{\myar@{{}{}{>|}}(-9,-19)*{};(-8.8,-21)*{}};
{\myar@{{}{}{>|}}(-51,-19)*{};(-51.2,-21)*{}};
{\myar@{{}{}{>|}}(-9,1)*{};(-8.8,-1)*{}};
{\myar@{{}{}{>|}}(-20,-10)*{};(-20,-11)*{}};
{\myar@{{}{}{>|}}(-31,-20)*{};(-30,-20)*{}};
{\myar@{{}{}{>|}}(-28,-12)*{};(-29,-11)*{}};
(-18,2)*{}="X1";(-18,-2)*{}="X2"; "X1";"X2"
**[thicker]\crv{(-15,8) & (-2,0) & (-15,-8) };
{\myar@{{}{}{>}}(-16.25,-4)*{};(-18,-2)*{}};
(-42,-18)*{}="X1";(-42,-22)*{}="X2"; "X1";"X2"
**[thicker]\crv{(-45,-12) & (-58,-20) & (-45,-28) };
{\myar@{{}{}{>}}(-43.75,-24)*{};(-42,-22)*{}};
(-18,-18)*{}="X1";(-18,-22)*{}="X2"; "X1";"X2"
**[thicker]\crv{(-15,-12) & (-2,-20) & (-15,-28) };
{\myar@{{}{}{>}}(-16.25,-24)*{};(-18,-22)*{}};
(-40,13)*{\txt{r}};(-5,0)*{\txt{e}};
(-55,-20)*{\txt{r}};(-5,-20)*{\txt{e}};
(-43,-10)*{\txt{a}};(-30,-15)*{\txt{r}}; (-30,-3)*{\txt{r}};(-30,18)*{\txt{\red
0.2}};(-30,5)*{\txt{\red 0.2}};
\endxy};
(-2,-20)*{\txt{\sffamily \textsc{Model}}};
(45,-20)*{\txt{\sffamily \textsc{Fraction Net Observer}}};
\endxy\vspace{3pt}
\caption{Underlying plant and  Petri Net
Observer}\label{optip-figex1}
\end{minipage}
\end{center}
\end{figure}\vspace{0pt}
The arc weights are then computed for the Fraction Net
Observer and the result is shown in the righthand plate of Figure~\ref{optip-figex1}.
Note that the arcs in red are the ones with fractional weights
in this case; all other arc weights are unity. The set of
transitions matrices $\boldsymbol{\Gamma}$ are now computed
from Algorithm~\ref{AlgGammaFNO} as:
\begin{gather*}
\Gamma^e = \left [ \begin{array}{cccc} \mspace{0mu} 0 & \mspace{-15mu} 0 & \mspace{-15mu} 0 & \mspace{-15mu} 0  \\
\mspace{0mu} 0 & \mspace{-15mu} \textbf{\blue 1} & \mspace{-15mu} 0 & \mspace{-15mu} 0  \\
\mspace{0mu} 0 & \mspace{-15mu} 0 & \mspace{-15mu} \textbf{\blue 1} & \mspace{-15mu} \textbf{\blue 1}  \\
\mspace{0mu} 0 & \mspace{-15mu} 0 & \mspace{-15mu} 0 & \mspace{-15mu} 0
\end{array} \right ], \
\Gamma^r = \left [ \begin{array}{cccc}
\mspace{0mu} \textbf{\blue 1} & \mspace{-15mu} \textbf{\blue 1} & \mspace{-15mu} 0 & \mspace{-15mu} 0  \\
\mspace{0mu} \textbf{\red 0.2} & \mspace{-15mu} \textbf{\red 0.2} & \mspace{-15mu} 0 & \mspace{-15mu} 0  \\
\mspace{0mu} 0 & \mspace{-15mu} 0 & \mspace{-15mu} 0 & \mspace{-15mu} 0  \\
\mspace{0mu} 0 & \mspace{-15mu} 0 & \mspace{-15mu} \textbf{\blue 1} & \mspace{-15mu} \textbf{\blue 1}
\end{array} \right ] 
%
%
\Gamma^a = \left [ \begin{array}{cccc}
 \mspace{0mu} 0 & \mspace{-15mu} 0 & \mspace{-15mu} 0 & \mspace{-15mu} \textbf{\blue 1}  \\
\mspace{0mu} 0 & \mspace{-15mu} 0 & \mspace{-15mu} 0 & \mspace{-15mu} \textbf{\red 0.2}  \\
\mspace{0mu} 0 & \mspace{-15mu} \textbf{\blue 1} & \mspace{-15mu} 0 & \mspace{-15mu} 0  \\
\mspace{0mu} 0 & \mspace{-15mu} 0 & \mspace{-15mu} 0 & \mspace{-15mu} 0
 \end{array} \right ]
\end{gather*}
We consider three different observation sequences $rr,re,ra$
assuming that the initial state in the underlying plant is
$00$ in each case ($i.e.$ the initial marking of the FNO is
given by $\alpha^0 = [ 1 \ 0 \ 0 \ 0]^T$. The final markings ($i.e.$ the entangled states) are
given by:
\begin{gather}
\alpha\Gamma^r \Gamma^r  = \left [
\begin{array}{c} 1.20\\ 0.24\\ 0\\
0\end{array} \right],  \alpha\Gamma^r \Gamma^e  = \left [
\begin{array}{c} 0\\ 0.2\\ 0 \\ 0 \end{array} \right ],
\alpha \Gamma^r \Gamma^a
= \left [
\begin{array}{c} 0\\ 0\\ 0.2\\
0\end{array} \right]
\end{gather}
Note that while in the
case of the Petri Net observer, we could only say that
$\overline{Q}(rr) = \{q_1,q_2\}$, for the fraction net
observer, we have an estimate of the cost of occupying each
state ($1.2$ and $0.24$ respectively for the first case).
\begin{table}[t]
\begin{minipage}{3.5in}
\centering
\caption{Event Occurrence Probabilities For Model 2}\label{table3}
\hspace{0pt}\begin{tabular}{||c|ccc|}
\hline & e &r & a \\ \hline \hline $00$ & $0.2$ & $0.79$ & \textbf{\BRed 0.01} $\boldsymbol{\leftarrow}$
\\\hline $01$ & $0.2$
& $0.5$ & $0.3$ $\phantom{\boldsymbol{\leftarrow}}$\\ \hline $11$ & $0.6$ & $0.39$ & \textbf{\BRed 0.01} $\boldsymbol{\leftarrow}$
\\ \hline $10$ & $0.3$ & $0.5$
& $0.2$ $\phantom{\boldsymbol{\leftarrow}}$\\ \hline
\end{tabular}
\end{minipage}
\end{table}\vspace{0pt}

Next we consider a slightly modified underlying plant 
with the event occurrence probabilities as tabulated in Table~\ref{table3}.
The modified plant (denoted as Model 2) is shown in
the righthand plate of Figure~\ref{fig2}.
\begin{figure}[!h]
\begin{center}
\begin{minipage}{3.25in}\hspace{0pt}
\xy (0,0)*{\xy 0;/r.18pc/:
(-40,0)*+[o][F]{00};(-20,0)*+[o][F-]{01};
(-20,-20)*+[o][F-]{11};(-40,-20)*+[o][F-]{10};
{\myar@{-->}^{\txt{e}}@[BrickRed](-37,0)*{};(-23,0)*{}};
{\myar@{->}^{\txt{a}}(-20,-2)*{};(-20,-17)*{}};
{\myar@{->}_{\txt{e}}(-37,-20)*{};(-23,-20)*{}};
{\myar@{->}^{\txt{a}}(-40,-17)*{};(-40,-2)*{}};
{\myar@{<-}^{\txt{r}}@/_1.5pc/(-37,0)*{};(-23,0)*{}};
{\myar@{<-}_{\txt{r}}@/^1.5pc/(-37,-20)*{};(-23,-20)*{}};
(-42,2)*{}="X1";(-42,-2)*{}="X2"; "X1";"X2"
**[thicker]\crv{(-45,8) & (-58,0) & (-45,-8) };
{\myar@{{}{}{>}}(-43.75,-4)*{};(-42,-2)*{}};
(-18,2)*{}="X1";(-18,-2)*{}="X2"; "X1";"X2"
**[thicker]\crv{(-15,8) & (-2,0) & (-15,-8) };
{\myar@{{}{}{>}}(-16.25,-4)*{};(-18,-2)*{}};
(-42,-18)*{}="X1";(-42,-22)*{}="X2"; "X1";"X2"
**[thicker]\crv{(-45,-12) & (-58,-20) & (-45,-28) };
{\myar@{{}{}{>}}(-43.75,-24)*{};(-42,-22)*{}};
(-18,-18)*{}="X1";(-18,-22)*{}="X2"; "X1";"X2"
**[thicker]\crv{(-15,-12) & (-2,-20) & (-15,-28) };
{\myar@{{}{}{>}}(-16.25,-24)*{};(-18,-22)*{}};
(-49,0)*{\txt{r}};(-11,0)*{\txt{e}};
(-49,-20)*{\txt{r}};(-11,-20)*{\txt{e}};
(-30,-30)*{\txt{\sffamily \textsc{Model 1}}};
\endxy };
(45,0)*{
\xy 0;/r.18pc/:
(-40,0)*+[o][F]{00};(-20,0)*+[o][F-]{01};
(-20,-20)*+[o][F-]{11};(-40,-20)*+[o][F-]{10};
{\myar@{-->}^{\txt{e}}@[BrickRed](-37,0)*{};(-23,0)*{}};
{\myar@{->}^{\txt{a}}(-20,-2)*{};(-20,-17)*{}};
{\myar@{->}_{\txt{e}}(-37,-20)*{};(-23,-20)*{}};
{\myar@{->}^{\txt{a}}(-40,-17)*{};(-40,-2)*{}};
{\myar@{<-}^{\txt{r}}@/_1.5pc/(-37,0)*{};(-23,0)*{}};
{\myar@{<-}_{\txt{r}}@/^1.5pc/(-37,-20)*{};(-23,-20)*{}};
(-42,2)*{}="X1";(-42,-2)*{}="X2"; "X1";"X2"
**[thicker]\crv{(-45,8) & (-58,0) & (-45,-8) };
{\myar@{{}{}{>}}(-43.75,-4)*{};(-42,-2)*{}};
(-18,2)*{}="X1";(-18,-2)*{}="X2"; "X1";"X2"
**[thicker]\crv{(-15,8) & (-2,0) & (-15,-8) };
{\myar@{{}{}{>}}(-16.25,-4)*{};(-18,-2)*{}};
(-42,-18)*{}="X1";(-42,-22)*{}="X2"; "X1";"X2"
**[thicker]\crv{(-45,-12) & (-58,-20) & (-45,-28) };
{\myar@{{}{}{>}}(-43.75,-24)*{};(-42,-22)*{}};
(-18,-18)*{}="X1";(-18,-22)*{}="X2"; "X1";"X2"
**[thicker]\crv{(-15,-12) & (-2,-20) & (-15,-28) };
{\myar@{{}{}{>}}(-16.25,-24)*{};(-18,-22)*{}};
(-48,0)*{\txt{r,a}};(-11,0)*{\txt{e}};
(-49,-20)*{\txt{r}};(-12,-20)*{\txt{a,e}};
(-30,-30)*{\txt{\sffamily \textsc{Model 2}}};
\endxy};
\endxy
\vspace{5pt}
\caption{Underlying models to illustrate effect of unobservability on the 
cardinality of the entangled state set}\label{fig2}
\end{minipage}
\end{center}\vspace{0pt}
\end{figure}
\begin{figure}[!h]
\vspace{0pt}
\centering
 \includegraphics[width=3in]{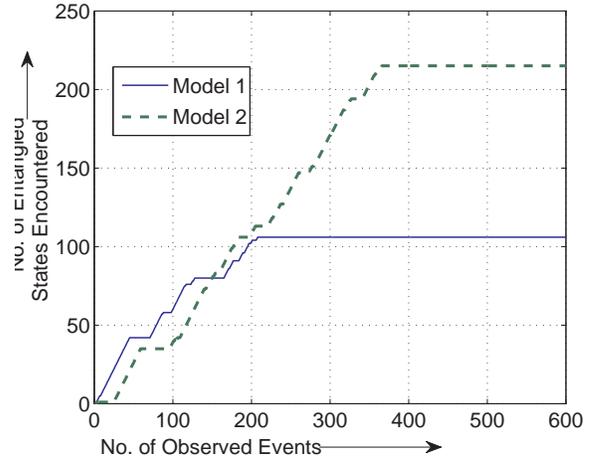}
\vspace{-6pt}
\caption{Total number of distinct entangled states encountered as a function of the number of observation ticks $i.e.$ the number of observed events}\label{figentangled}
\end{figure}
The two models are simulated with the initial pure state set to $[0 \ 0 \ 1\ 0]$
in each case. We note that the number of entangled states in the 
course of simulated operation more than doubles from $106$ for Model 1 to 
$215$ for Model 2 (See Figure~\ref{figentangled}). In the simulation, entangled state vectors were distinguished
with a tolerance of $10^{-10}$ on the max norm.
\subsection{Maximization of Integrated
Instantaneous Measure}
\begin{defn}\label{definstchar}{Instantaneous Characteristic: }
Given a  plant  \Gt,
the instantaneous characteristic $\hat{\chiup}(t)$ is defined as a
function of plant operation time $t\in[0,\infty)$ as follows:
 \begin{gather}
\hat{\chiup}(t) = \boldsymbol{\chiup}\big \vert_i
\end{gather}
where $q_i \in Q$ is the state occupied at time $t$ 
\end{defn}
\begin{defn}\label{defninstm}
{Instantaneous Measure For Perfectly Observable Plants: }
Given a  plant  \Gt,
the instantaneous measure ($\displaystyle \hat{\nu}_\theta(t)$) is defined as a
function of plant operation time $t\in[0,\infty)$ as follows:
\begin{gather}
\hat{\nu}_\theta(t) = \langle \alpha(t), \boldsymbol{\nu}_\theta \rangle
\end{gather}
where $\alpha \in \mathcal{B}$ corresponds to the state that G is observed to occupy at time $t$ (Refer to Eq.~\eqref{eqpure}) and 
$\boldsymbol{\nu}_\theta$ is the renormalized language measure vector for the
underlying plant $G$ with uniform termination probability $\theta$.
\end{defn}
\vspace{0pt}
 Next we show that the
optimal control algorithms presented in
Section~\ref{seconlineS} for perfectly observable
situations can be interpreted as maximizing the expectation of the
time-integrated instantaneous measure for the finite state
plant model under consideration.
\begin{prop}\label{propmuint}
For the unsupervised plant \G with all transitions observable at the supervisory
level, let $G^\star$ be the optimally supervised plant and
$G^{\#}$ be obtained by arbitrarily disabling
controllable transitions. Denoting the instantaneous measures 
for $G^\star$ and $G^{\#}$
by $\hat{\nu}^{\star}_\theta(t)$ and $\hat{\nu}^{{\#}}_\theta(t)$ for some uniform termination probability $\theta \in (0,1)$
respectively, we have
\begin{gather}
\mathbf{E} \left
(\int_0^t\hat{\nu}^{\star}_\theta(\tauup)\mathrm{d}\tauup \right
)\geqq \mathbf{E} \left
(\int_0^t\hat{\nu}^{\#}_\theta(\tauup)\mathrm{d}\tauup \right )
\forall t \in[0,\infty), \forall \theta \in (0,1)
\end{gather}
where $t$ is the plant operation time and
$\mathbf{E}(\cdot)$ denotes the expected value of the
expression within braces.
\end{prop}
\vspace{0pt}
\begin{proof}
Assume that the stochastic transition probability matrix for
an arbitrary finite state plant model  be denoted by
$\Pi$ and denote the Cesaro limit as:
$\displaystyle
\Q=\lim_{k\rightarrow\infty}\frac{1}{k}\sum_{j=0}^{k-1}
\Pi^{j}
$.
Denoting the final stable state probability vector as
$p^i$, where the plant is assumed to initiate operation in
state $q_i$, we claim that
$
 p^i_j = \Q_{ij}
$
which follows immediately from noting that if the initiating
state is $q_i$ then
\begin{gather*}
  \begin{array}{cccccccccc} (p^i)^T = \bigg [ 0\cdots &0& \mspace{-70mu}1 \cdots
&0  \bigg  ]
\lim_{k\rightarrow\infty}\frac{1}{k}\sum_{j=0}^{k-1}
\Pi^{j} \\
&&\uparrow i^{th} \ \textrm{element}&\end{array}
\end{gather*}
$i.e.$ $(p^i)^T$ is the $i^{th}$ row of $\Q$. Hence, we have
\begin{align*}
\mathbf{E} \left (\int_0^t\hat{\chiup}(\tauup)\mathrm{d}\tauup
\right ) = \int_0^t\mathbf{E} \left
(\hat{\chiup}(\tauup)\right
)\mathrm{d}\tauup = t\langle p^i,\boldsymbol{\chi} \rangle
= t\boldsymbol{\nu}_0 \big \vert_i && (\mathrm{Note:} \ \theta=0  ) 
\end{align*}
where finite number of states guarantees
that the expectation operator and the integral can be
exchanged 
Recalling that optimal supervision
 elementwise maximizes the language measure vector
$\boldsymbol{\nu}_0$, we conclude
\begin{gather}\label{eqchimax}
\mathbf{E} \left
(\int_0^t\hat{\chiup}^{\star}(\tauup)\mathrm{d}\tauup \right
)\geqq \mathbf{E} \left
(\int_0^t\hat{\chiup}^{\#}(\tauup)\mathrm{d}\tauup \right
) \forall t \in[0,\infty)
\end{gather}
where the $\hat{\chiup}(t)$ for the plant configurations
$G^\star$ and $G^{\#}$ is denoted as  $\hat{\chiup}^{\star}$
and $\hat{\chiup}^{\#}$ respectively. 
Noting that the construction of the Petri Net observer
(Algorithm~\ref{Alg2}) implies that in
the case of perfect observation, each transition leads to
exactly one place, we conclude that the instantaneous measure
is given by
\begin{gather}
\hat{\nu}_\theta(t) = \boldsymbol{\nu}_\theta \big \vert_i \ \textrm{where the current state at
time $t$ is $q_i$}
\end{gather}
Furthermore, we recall from Corollary \ref{Corollary2.1}
\begin{align}\label{eqexpect}
 \Q\boldsymbol{\nu}_\theta =
\Q\boldsymbol{\chi} \Longrightarrow \mathbf{E} \left (
\hat{\nu}_\theta(t) \right ) =\mathbf{E} \left ( \hat{\chiup}(t)
\right ) \forall t \in [0,\infty)
\end{align} which leads to the following argument:
\begin{align*}
& \phantom{\Longrightarrow}\mathbf{E} \left
(\int_0^t\hat{\chiup}^{\star}(\tauup)\mathrm{d}\tauup \right
)\geqq \mathbf{E} \left
(\int_0^t\hat{\chiup}^{\#}(\tauup)\mathrm{d}\tauup \right
) \ \forall t \in[0,\infty)\\
 &\Longrightarrow
\int_0^t\mathbf{E} \left
(\hat{\chiup}^{\star}(\tauup)\right )\mathrm{d}\tauup \geqq
\int_0^t\mathbf{E}\left (\hat{\chiup}^{\#}(\tauup)\right
)\mathrm{d}\tauup
 \ \forall t \in[0,\infty)\\
&\Longrightarrow \int_0^t\mathbf{E} \left
(\hat{\nu}_\theta^{\star}(\tauup)\right )\mathrm{d}\tauup \geqq
\int_0^t\mathbf{E}\left (\hat{\nu}_\theta^{\#}(\tauup)\right
)\mathrm{d}\tauup
 \ \forall t \in[0,\infty), \forall \theta \in (0,1) \\
 &\Longrightarrow \mathbf{E} \left
(\int_0^t\hat{\nu}_\theta^{\star}(\tauup)\mathrm{d}\tauup \right
)\geqq \mathbf{E} \left
(\int_0^t\hat{\nu}_\theta^{\#}(\tauup)\mathrm{d}\tauup \right ) \ 
\forall t \in[0,\infty), \forall \theta \in (0,1)
\end{align*}
This completes the proof.
\end{proof}
\vspace{0pt}
Next we formalize a procedure of implementing an optimal supervision
policy from a knowledge of the optimal language measure vector for the
underlying plant. 
\begin{figure}[t]
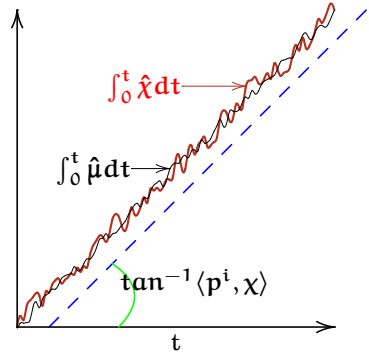

\begin{center}
\begin{minipage}{3.3in}
\hspace{50pt} \xy 0;/r.2pc/:(0,0)*{}="X1"; (50,50)*{}="X2";
"X1";"X2" **[thicker][BrickRed]\crv{ (0.00 , 0.35) & (0.50 , 2.88)
& (1.00 , -0.07) & (1.50 , 3.93) & (2.00 , 4.08) & (2.50 ,
4.72) & (3.00 , 5.49) & (3.50 , 5.50) & (4.00 , 2.03) & (4.50
, 6.65) & (5.00 , 4.73) & (5.50 , 6.95) & (6.00 , 7.26) &
(6.50 , 8.08) & (7.00 , 6.38) & (7.50 , 7.62) & (8.00 , 7.60)
& (8.50 , 8.87) & (9.00 , 7.66) & (9.50 , 10.12) & (10.00 ,
7.97) & (10.50 , 10.25) & (11.00 , 8.55) & (11.50 , 11.81) &
(12.00 , 12.99) & (12.50 , 13.60) & (13.00 , 14.64) & (13.50 ,
14.41) & (14.00 , 12.98) & (14.50 , 14.24) & (15.00 , 17.41) &
(15.50 , 17.89) & (16.00 , 17.78) & (16.50 , 15.76) & (17.00 ,
14.67) & (17.50 , 16.40) & (18.00 , 19.90) & (18.50 , 17.24) &
(19.00 , 17.70) & (19.50 , 18.94) & (20.00 , 21.40) & (20.50 ,
18.28) & (21.00 , 22.83) & (21.50 , 21.01) & (22.00 , 22.38) &
(22.50 , 21.19) & (23.00 , 21.98) & (23.50 , 23.49) & (24.00 ,
23.33) & (24.50 , 22.65) & (25.00 , 27.05) & (25.50 , 27.80) &
(26.00 , 24.15) & (26.50 , 24.26) & (27.00 , 28.59) & (27.50 ,
29.11) & (28.00 , 26.26) & (28.50 , 29.15) & (29.00 , 29.07) &
(29.50 , 30.81) & (30.00 , 27.72) & (30.50 , 28.70) & (31.00 ,
29.26) & (31.50 , 29.20) & (32.00 , 33.62) & (32.50 , 34.34) &
(33.00 , 34.04) & (33.50 , 32.11) & (34.00 , 32.26) & (34.50 ,
36.01) & (35.00 , 34.53) & (35.50 , 33.63) & (36.00 , 38.37) &
(36.50 , 38.98) & (37.00 , 38.77) & (37.50 , 39.95) & (38.00 ,
38.20) & (38.50 , 38.21) & (39.00 , 39.22) & (39.50 , 40.70) &
(40.00 , 37.92) & (40.50 , 42.61) & (41.00 , 40.71) & (41.50 ,
39.19) & (42.00 , 40.24) & (42.50 , 40.25) & (43.00 , 40.50) &
(43.50 , 43.78) & (44.00 , 45.09) & (44.50 , 42.45) & (45.00 ,
47.04) & (45.50 , 43.21) & (46.00 , 44.25) & (46.50 , 45.24) &
(47.00 , 46.72) & (47.50 , 49.37) & (48.00 , 47.43) & (48.50 ,
48.80) & (49.00 , 50.03) & (49.50 , 51.77)
};{\myar@{--}@[blue](5,0)*{};(55,50)*{}};
(0,0)*{}="X1"; (50,50)*{}="X2"; "X1";"X2" **[black]\crv{ (0.00
, -0.52) & (0.50 , 1.09) & (1.00 , 0.43) & (1.50 , 0.85) &
(2.00 , 0.50) & (2.50 , 3.50) & (3.00 , 4.49) & (3.50 , 2.25)
& (4.00 , 4.84) & (4.50 , 4.97) & (5.00 , 5.34) & (5.50 ,
6.18) & (6.00 , 5.78) & (6.50 , 6.52) & (7.00 , 7.88) & (7.50
, 8.01) & (8.00 , 8.37) & (8.50 , 8.14) & (9.00 , 10.10) &
(9.50 , 8.61) & (10.00 , 8.94) & (10.50 , 10.01) & (11.00 ,
10.98) & (11.50 , 10.53) & (12.00 , 11.35) & (12.50 , 12.34) &
(13.00 , 14.27) & (13.50 , 13.65) & (14.00 , 15.26) & (14.50 ,
13.60) & (15.00 , 14.95) & (15.50 , 14.67) & (16.00 , 15.39) &
(16.50 , 16.12) & (17.00 , 16.50) & (17.50 , 17.71) & (18.00 ,
18.51) & (18.50 , 19.32) & (19.00 , 19.76) & (19.50 , 18.32) &
(20.00 , 20.97) & (20.50 , 19.83) & (21.00 , 21.33) & (21.50 ,
22.53) & (22.00 , 22.60) & (22.50 , 23.46) & (23.00 , 21.71) &
(23.50 , 22.86) & (24.00 , 25.01) & (24.50 , 25.64) & (25.00 ,
26.02) & (25.50 , 25.54) & (26.00 , 26.87) & (26.50 , 27.31) &
(27.00 , 26.22) & (27.50 , 28.14) & (28.00 , 28.01) & (28.50 ,
27.11) & (29.00 , 30.25) & (29.50 , 28.71) & (30.00 , 29.16) &
(30.50 , 29.48) & (31.00 , 31.01) & (31.50 , 31.51) & (32.00 ,
33.48) & (32.50 , 31.16) & (33.00 , 31.70) & (33.50 , 34.69) &
(34.00 , 33.95) & (34.50 , 33.96) & (35.00 , 34.16) & (35.50 ,
35.48) & (36.00 , 36.69) & (36.50 , 35.52) & (37.00 , 37.93) &
(37.50 , 36.42) & (38.00 , 37.04) & (38.50 , 38.37) & (39.00 ,
38.87) & (39.50 , 39.07) & (40.00 , 39.40) & (40.50 , 39.54) &
(41.00 , 40.49) & (41.50 , 42.11) & (42.00 , 42.06) & (42.50 ,
42.28) & (43.00 , 42.57) & (43.50 , 42.14) & (44.00 , 44.12) &
(44.50 , 44.74) & (45.00 , 46.18) & (45.50 , 46.66) & (46.00 ,
45.41) & (46.50 , 45.13) & (47.00 , 46.74) & (47.50 , 47.43) &
(48.00 , 47.76) & (48.50 , 47.11) & (49.00 , 47.57) & (49.50 ,
48.85) } ;
{\myar@{->}_{\txt{t}}(0,0)*{};(50,0)*{}};
{\myar@{->}(0,0)*{};(0,50)*{}};
{\myar@{-}@/_0.6pc/@[green](16,0)*{};(15,10)*{}};(28,7)*{\large
\boldsymbol{tan^{-1}\langle p^i,\chi\rangle}};(20,38)*{\large
\boldsymbol{\red \int_0^t\hat{\chiup}dt}};(12,25)*{\large
\boldsymbol{\int_0^t\hat{\muup}dt}};
{\ar@[BrickRed]@{->}(26,38)*{};(36,38)*{}};
{\ar@{->}(18,25)*{};(24,25)*{}};
\endxy
\caption[Time integrals of instantaneous measure \&
instantaneous characteristic]{Time integrals of instantaneous
measure and instantaneous characteristic Vs operation
time}\label{figerg}
\end{minipage}
\end{center}
\end{figure}
\subsection{The Optimal Control Algorithm}\label{secregunobs}
For any finite state underlying plant \Gt and a specified unobservability map $\p$, it is possible to 
define a probabilistic transition system as a possibly infinite state 
generalization of PFSA which we denote as the entangled transition system corresponding
to the underlying plant and the specified unobservability map. In defining the entangled transition system (Definition~\ref{defentangtran}), we
use a similar formalism as stated in Section~\ref{PFSAmodel}, with
the exception of dropping the last argument for controllability specification in Eq.~\eqref{sextuple}. Controllability needs to handled separately to address the issues of partial controllability arising as a result of partial observation.
\begin{defn}\label{defentangtran}(Entangled Transition System:)
 For a given plant \Gt and an unobservability map $\p$, the entangled transition system $\mathscr{E}_{(G,\p)} = (Q_\mathscr{F},\Sigma,
\Delta,\tilde{\pi}_\mathscr{E},\chi_\mathscr{E})$ is defined as:
\begin{enumerate}
 \item The transition map $\Delta : Q_\mathscr{F} \times \Sigma^\star \rightarrow 
Q_\mathscr{F}$ is defined as :
\begin{gather*}
 \forall \alpha \in Q_\mathscr{F}, \ \Delta(\alpha,\omega) = \alpha\prod_{\sigma_1}^{\sigma_m}\Gamma^{\sigma_i} \ \mathrm{where} \ \omega = \sigma_1 \cdots \sigma_m 
\end{gather*}
\item The event generation probabilities  $\tilde{\pi}_\mathscr{E} : Q_\mathscr{F} 
\times \Sigma^\star \rightarrow [0,1]$ are specified as:
\begin{gather*}
 \tilde{\pi}_\mathscr{E}(\alpha,\sigma) = \sum_{i=1}^{i=\Crd(Q)}(1-\theta)\mathcal{N}(\alpha_i) \tilde{\pi}(q_i,\sigma)
\end{gather*}
\item The characteristic function $\chi_\mathscr{E}: Q_\mathscr{F} \rightarrow [-1,1]$ is defined as: 
$ \chi_\mathscr{E}(\alpha) = \langle \alpha,\boldsymbol{\chi} \rangle$
\end{enumerate}
\end{defn}
\begin{rem}\label{rem4p3}
 The definition of $\tilde{\pi}_\mathscr{E} $ is consistent in the  sense:
\begin{gather*}
 \forall \alpha \in Q_\mathscr{F}, \ \sum_{\sigma \in \Sigma}\tilde{\pi}_\mathscr{E}(\alpha,\sigma) = 
 = \sum_{i}\mathcal{N} (\alpha_i) (1-\theta) = 1 - \theta
\end{gather*}
implying that if $Q_\mathscr{F}$ is finite then $\mathscr{E}_{(G,\p)} $ is a \textbf{perfectly observable} terminating model with uniform termination probability $\theta$.
\end{rem}
\vspace{3pt}
\begin{prop}\label{propentangledmu}
The renormalized language measure $\nu^\mathscr{E}_\theta(\alpha)$ for the state $\alpha \in Q_\mathscr{F}$ of the entangled transition system \E can be computed as follows:
\begin{gather}
 \nu^\mathscr{E}_\theta(\alpha) = \langle \alpha , \boldsymbol{\nu}_\theta \rangle
\end{gather}
where $\boldsymbol{\nu}_\theta$ is the language measure vector for the underlying terminating
plant \Gt with uniform termination probability $\theta$.
\end{prop}
\begin{proof}
We first compute the measure of the pure states $\mathcal{B} \subset Q_\mathscr{F}$ of \E denoted by the vector $\boldsymbol{\nu}_\theta^\mathscr{E}$. Since every string generated by the
 Phantom automaton is completely unobservable, it follows that the measure of the empty string $\epsilonup$ from any state $\alpha \in \mathcal{B}$ is given by
$\alpha\big [ \mathbb{I} - (1-\theta)\mathscr{P}(\Pi) 
\big ]^{-1}\boldsymbol{\chi}$. Let $\alpha$ correspond to the state $q_i \in Q$ in the underlying plant.
Then the measure of the set of all strings generated from $\alpha \in \mathcal{B}$ having at least one observable 
transition in the underlying plant is given by 
\begin{gather}
 \sum_j(1-\theta)\bigg [ \mathbb{I} - (1-\theta)\mathscr{P}(\Pi) \bigg ]^{-1}\bigg 
( \Pi - \mathscr{P}(\Pi) \bigg )\bigg\vert_{ij}\big \{ \boldsymbol{\nu}_\theta^\mathscr{E}\big \}_{j}
\end{gather}
which is simply the
measure of the set of all strings of the form $\omega_1\sigma\omega_2$ where $\p(\omega_1\sigma\omega_2) = \sigma\p(\omega_2)$.
It therefore follows from the additivity of  measures that
\begin{align}
 \boldsymbol{\nu}_\theta^\mathscr{E} &= (1-\theta)\bigg [ \mathbb{I} - (1-\theta)\mathscr{P}(\Pi) \bigg ]^{-1}\bigg ( \Pi - \mathscr{P}(\Pi) \bigg )\boldsymbol{\nu}_\theta^\mathscr{E} \notag \\ & \mspace{250mu} +  \bigg [ \mathbb{I} - (1-\theta)\mathscr{P}(\Pi) 
\bigg ]^{-1}\boldsymbol{\chi} \notag \\
\Rightarrow \boldsymbol{\nu}_\theta^\mathscr{E} &= \Bigg [ \mathbb{I} - (1-\theta)\bigg [ \mathbb{I} - (1-\theta)\mathscr{P}(\Pi) \bigg ]^{-1}\bigg ( \Pi -
 \mathscr{P}(\Pi) \bigg ) \Bigg ]^{-1}\notag \\ & \mspace{250mu} \times \bigg [ \mathbb{I} - (1-\theta)\mathscr{P}(\Pi) 
\bigg ]^{-1}\boldsymbol{\chi} \notag \\
\Rightarrow \boldsymbol{\nu}_\theta^\mathscr{E} &= 
\bigg [ \mathbb{I} - (1-\theta)\Pi \bigg ]^{-1}\boldsymbol{\chi} 
=\boldsymbol{\nu}_\theta
\end{align}
 which implies that for any pure state $\alpha \in \mathcal{B}$, we have $\nu_\theta^\mathscr{E}(\alpha) = \langle \alpha,\boldsymbol{\nu}_\theta\rangle$.
The general result then follows from  the following linear relation arising from the
definitions of $\tilde{\pi}_\mathscr{E}$ and $\chi_\mathscr{E}$:
\begin{gather}
 \forall \alpha \in \mathcal{B}, \forall k \in \mathbb{R}, \nu_\theta^\mathscr{E}(k\alpha) = k \nu_\theta^\mathscr{E}(\alpha)
\end{gather}
This completes the proof.
\end{proof}
%
\begin{defn}\label{definstcharentang}(Instantaneous Characteristic for Entangled Transition Systems:) Given an underlying  plant  \Gt and an unobservability map $\p$,
the instantaneous characteristic $\hat{\chiup}_\mathscr{E}(t)$ for the
corresponding entangled transition system \E is defined as a
function of plant operation time $t\in[0,\infty)$ as follows:
 \begin{gather}
\hat{\chiup}_\mathscr{E}(t) = \langle \alpha(t), \boldsymbol{\chi} \rangle
\end{gather}
where $\alpha(t)$ is the entangled state occupied at time $t$ 
\end{defn}
\begin{defn}\label{defninstmentangled}
(Instantaneous Measure For Partially Observable Plants:) Given an underlying  plant 
\Gt and an unobservability map $\p$,
the instantaneous measure ($\displaystyle \hat{\nu}_\theta(t)$) is defined as a
function of plant operation time $t\in[0,\infty)$ as follows:
\begin{gather}
\hat{\nu}_\theta(t) = \langle \alpha(t), \boldsymbol{\nu}^\mathscr{E}_\theta \rangle
\end{gather}
where $\alpha \in \QE$ is the entangled state at time $t$ and 
$\boldsymbol{\nu}^\mathscr{E}_\theta$ is the renormalized language measure vector for the
corresponding entangled transition system \E.
\end{defn}
\begin{cor}\label{lementinst}(Corollary to Proposition~\ref{propentangledmu}) For a given plant \Gt and an unobservability map $\p$, the instantaneous measure 
$\hat{\nu}_\theta:[0,\infty) \rightarrow [-1,1]$ is given by
\begin{gather}
 \hat{\nu}_\theta(t) = \langle \alpha(t), \boldsymbol{\nu}_\theta \rangle
\end{gather}
where $\alpha(t)$ is the current  state of the entangled transition system \E at time $t$ and
 $\boldsymbol{\nu}_\theta$ is the language measure vector for the underlying plant $G$.
\end{cor}
\vspace{3pt}
\begin{proof}
 Follows from Definitions~\ref{defninstmentangled}, \ref{defentangtran} and 
Proposition~\ref{propentangledmu}.
\end{proof}
Proposition~\ref{propentangledmu} has a crucial consequence. It follows that elementwise maximization of the
measure vector $\boldsymbol{\nu}_\theta$ for the underlying plant automatically maximizes the measures of each of the entangled states irrespective
of the particular unobservability map $\p$. This allows us to directly formulate
the optimal supervision policy for cases where the cardinality of the entangled state set is finite. 
However, before we embark upon the construction of such policies, we need to address the 
controllability issues arising due to state entanglement.
We note that 
for a given  entangled state $\alpha \in Q_\mathscr{F}\setminus \mathcal{B}$, an event $\sigma \in \Sigma$
may be controllable from some but not all of the states $q_i \in Q$ that satisfy $\alpha_i > 0$. 
Thus the notion of controllability introduced in Definition~\ref{optifull-contdef}
needs to be generalized; \textit{disabling of a transition $\sigma \in \Sigma$ from an entangled state can still change
the current state}. We formalize the analysis by defining a set of event-indexed disabled transition matrices by
suitably modifying $\Gamma^\sigma$ as follows:
\begin{defn}\label{defgencontro}
 For a given plant \G, the event indexed disabled transition matrices $\Gamma_\mathcal{D}^\sigma$ is defined as
\begin{gather*}
 \Gamma_\mathcal{D}^\sigma\big \vert_{ij} =  \left \{ \begin{array}{ll} \delta_{ij},  & \mathrm{if} \ \sigma \ \mathrm{is \ controllable \ at} \ q_i \ \mathrm{and} \ \p(q_i,\sigma) = \sigma\\
					\Gamma^\sigma_{ij}, & \mathrm{otherwise}
                                                       \end{array}
\right.
\end{gather*}
Evolution of the current entangled state $\alpha$ to $\alpha'$ due to the firing of the 
disabled transition $\sigma \in \Sigma$ is then computed as:
\begin{gather}
 \alpha' = \alpha \Gamma_\mathcal{D}^\sigma
\end{gather}
\end{defn}
\vspace{3pt}
\begin{rem}\label{remcontassump}
 If an event $\sigma \in \Sigma$ is uncontrollable 
 at every state $q_i \in Q$, then 
$\Gamma^\sigma_\mathcal{D} = \Gamma^\sigma$. On the other hand, if event $\sigma$ is always controllable 
(and hence by our assumption always observable), then we have $\Gamma^\sigma_\mathcal{D} = \mathbb{I}$.
In general, we have $\Gamma^\sigma_\mathcal{D} \neq \Gamma^\sigma \neq \mathbb{I}$.
\end{rem}\vspace{3pt}
Proposition~\ref{propmuint} shows that  optimal supervision in the case of perfect observation yields a
 policy that maximizes the time-integral of the
instantaneous measure. We now outline a procedure (See Algorithm~\ref{PrAlg3}) to maximize
$\int_0^t\hat{\nu}_\theta(\tauup)\mathrm{d}\tauup$ when the
underlying plant has a non-trivial unobservability
map.
\begin{algorithm}
\footnotesize
  \SetLine
  \SetKwData{Left}{left}
  \SetKwData{This}{this}
  \SetKwData{Up}{up}
  \SetKwFunction{Union}{Union}
  \SetKwFunction{FindCompress}{FindCompress}
  \SetKwInOut{Input}{input}
  \SetKwInOut{Output}{output}
  \caption{Optimal Control under Partial Observation (Preliminary Procedure For Illustration)}\label{PrAlg3}
\dontprintsemicolon
\Input{\G,$\p$, Initial State $q_0$ for $G$}
\Begin{
\While(\tcc*[f]{Infinite Loop}){true}{
 	\textrm{\sffamily Compute the optimal measure vector} $\boldsymbol{\nu}_\star$ \textrm{\sffamily for}  $G$ \;
	\textrm{\sffamily Set the current entangled state to }$\alpha = q_0\big [ \mathbb{I} - \mathscr{P}(\Pi) \big ]^{-1}$\;
		\If{current entangled state is $\alpha$}{
			\For{$\sigma \in \Sigma$}{
				\If{ $\langle \alpha\Gamma^\sigma,\boldsymbol{\nu}_\star \rangle <  \langle \alpha\Gamma^\sigma_\mathscr{D},\boldsymbol{\nu}_\star \rangle$ }{
					\textrm{\sffamily Disable} $\sigma$\;
				}
			}
		}
		Observe next event $\sigma \in \Sigma$\;
		\eIf{\textrm{\sffamily  $\sigma$ is enabled}}{
			\textrm{\sffamily Update the  entangled state to} $\alpha\Gamma^\sigma$\;
		}
		{
			\textrm{\sffamily Update the  entangled state to} $\alpha\Gamma^\sigma_\mathscr{D}$\;
		}
}
}
\end{algorithm}
\begin{lem}\label{lemassum2}
Let the following condition be satisfied for  a plant \G and an unobservability map $\p$:
 \begin{gather}
 \Crd(Q_\mathscr{F}) < \infty \label{eqcont}
\end{gather}
Then the control actions generated by Algorithm~\ref{PrAlg3}
is optimal in the sense that
\begin{gather}
\mathbf{E} \left
(\int_0^t\hat{\nu}_\theta^{\star}(\tauup)\mathrm{d}\tauup \right
)\geqq \mathbf{E} \left
(\int_0^t\hat{\nu}_\theta^{\#}(\tauup)\mathrm{d}\tauup\right ) \
\forall t \in [0,\infty), \forall \theta \in (0,1)
\end{gather}
where $\hat{\nu}_\theta^\star(t)$ and $\hat{\nu}_\theta^{\#}(t)$ are the
instantaneous measures at time $t$ for control actions generated by
Algorithm~\ref{PrAlg3} and an arbitrary policy
respectively.
\end{lem}
\vspace{0pt}
\begin{proof}
\textbf{Case 1:} First we consider the case where the following condition is true:
\begin{gather}
 \forall \sigma \in \Sigma,  \ \left (\Gamma^\sigma_\mathscr{D} = \Gamma^\sigma\right ) \bigvee \left ( \forall \alpha \in Q_\mathscr{F},\  \alpha\Gamma^\sigma_\mathscr{D} 
= \alpha \right )
\end{gather}
which can be paraphrased as follows: 
\begin{quote}
 \textit{Each event is either uncontrollable at every 
state $q\in Q$ in the underlying plant 
\Gt  or is controllable at
every state at which it is observable.}
\end{quote}
We note that the entangled transition system qualifies 
 as a perfectly observable probabilistic finite state machine
(See Remark~\ref{rem4p3}) since the unobservability
effects have been
eliminated by introducing the entangled states. If the above condition stated in Eq.~\eqref{eqcont} is true, then no generalization 
of the notion of event controllability in \E is required (See Definition~\ref{defgencontro}).
Under this assumption, the claim of the lemma then  follows from Lemma~\ref{lemremmeas} by noting that Algorithm~\ref{PrAlg3} under the above assumption
reduces to the procedure stated in Algorithm~\ref{Algoonlineopt} when we view the entangled system as a perfectly observable PFSA model.\\
\textbf{Case 2:} Next we consider the general scenario where the condition in Eq.~\eqref{eqcont} is relaxed. We note that the key to the online implementation result in 
stated Lemma~\ref{lemremmeas} is the Monotonicity lemma proved in \cite{CR07} which states
that for any given terminating plant \Gt with uniform termination probability $\theta$, the following iteration 
sequence elementwise increases the measure vector monotonically:
\begin{quote}
 1. Compute $\boldsymbol{\nu}_\theta$\\
2. If $\boldsymbol{\nu}_\theta \vert_i < \boldsymbol{\nu}_\theta \vert_j$, then disable all events $q_i \xrightarrow{\sigma} q_j$, otherwise
enable all events $q_i \xrightarrow{\sigma} q_j$ \\ 
3. Go to step 1.
\end{quote}
The proof of the Monotonicity Lemma~\cite{CR07} assumes that ``disabling'' $q_i \xrightarrow{\sigma} q_j$ replaces it with  a self loop at state $q_i$ labelled $\sigma$ with the same generation probability; $i.e.$ $\widetilde{\Pi}(q_i,\sigma)$ remains unchanged.
Now if there exists $\sigma \in \Sigma$ with $\Gamma^\sigma_\mathscr{D} \neq \mathbb{I}$, then we need to consider the fact that on disabling $\sigma$,
the new transition is no longer a self loop, but ends up in some other state $q_k \in Q$. Under this more general situation, we claim that
Algorithm~\ref{PrAlg3} is true; or in other words, we claim that
 the following procedure
elementwise increases the measure vector monotonically:
\begin{quote}
 1. Compute $\boldsymbol{\nu}_\theta$\\
2. Let $q_i \xrightarrow{\sigma} q_j$ (if enabled) and $q_i \xrightarrow{\sigma} q_k$ (if disabled)\\
3. If $\boldsymbol{\nu}_\theta \vert_j < \boldsymbol{\nu}_\theta \vert_k$, then disable  $q_i \xrightarrow{\sigma} q_j$, otherwise
enable  $q_i \xrightarrow{\sigma} q_j$ \\ 
4. Go to step 1.
\end{quote}
which is guaranteed by Proposition~\ref{propmonotone} in Appendix~\ref{appmono}. Convergence of this iterative process and the optimality 
of the resulting supervision policy in the sense of Definition~\ref{pdef} can be worked out exactly on similar lines as shown in \cite{CR07}. This completes the proof. 
\end{proof}
In order to extend the result of Lemma~\ref{lemassum2} to the general case where
the cardinality of the entangled state set can be infinite, we need 
to introduce a sequence of finite state approximations to the
potentially infinite state entangled transition system. This would allow us to
work out the above extension as a natural consequence of
continuity arguments. The finite state approximations are parametrized 
by $\eta \in (0,1]$ which approaches $0$ from above as we derive closer
and closer approximations. The formal definition  of such an 
$\eta$-Quantized Approximation for \E is stated next:
\begin{defn}\label{defetaapprox}
 ($\eta$-Quantized Approximation:) For a plant \Gt, an unobservability map $\p$ and a given $\eta \in (0,1]$, a probabilistic finite state machine
$\mathscr{E}_{(G,\p)}^\eta = (Q_\mathscr{F}^\eta,\Sigma,
\Delta^\eta,\tilde{\pi}_\mathscr{E},\chi_\mathscr{E})$ qualifies as a $\eta$-quantized approximation of the 
corresponding  entangled transition
 system  \E if 
\begin{gather}
 \Delta^\eta(\alpha,\omega) = \zeta_\eta(\Delta(\alpha,\omega))
\end{gather}
where $\zeta_\eta:[0,\frac{1}{\theta}]^{\Crd(Q)} \rightarrow Q_\mathscr{F}^\eta$ is a quantization map satisfying:
\begin{subequations}
\begin{align}
 &\Crd(Q_\mathscr{F}^\eta) < \infty \\
&\forall \alpha \in \mathcal{B}, \ \zeta_\eta(\alpha) = \alpha \\
&\forall \alpha \in Q_\mathscr{F}, \ \vert \vert \zeta_\eta(\alpha) - \alpha \vert \vert_\infty \leqq \eta
\end{align}
\end{subequations}
where $\vert \vert \cdot \vert \vert_\infty$ is the standard max norm.
Furthermore, we denote the language measure of the state $\alpha \in Q_\mathscr{F}^\eta$ as $\nu^\eta_\theta(\alpha)$ and the measure vector
for the pure states $\alpha \in \mathcal{B}$ is denoted as $\boldsymbol{\nu}_\theta^\eta$.
\end{defn}

We note the following:
\begin{enumerate}
\item For a given $\eta \in (0,1]$, there may exist uncountably infinite number of distinct probabilistic finite state machines that
qualify as a $\eta$-quantized approximation to \E; $i.e.$ the approximation is not unique.
 \item $\displaystyle \lim_{\eta \rightarrow 0^+} \mathscr{E}_{(G,\p)}^\eta = \mathscr{E}_{(G,\p)}$
\item The compactness of $[0,\frac{1}{\theta}]^{\Crd(Q)}$ is crucial in the definition.
\item The set of pure states of \E is a subset of $Q_\mathscr{F}^\eta$, $i.e.$, $\mathcal{B} \subset Q_\mathscr{F}^\eta$.
\item The measure of an arbitrary state $\alpha \in Q_\mathscr{F}^\eta$ is given by
$\langle \alpha , \boldsymbol{\nu}_\theta^\eta\rangle$.
\end{enumerate}
\begin{lem}\label{lemcontinuity}
 The language measure vector $\boldsymbol{\nu}_\theta^\eta$ for the set of pure states $\mathcal{B}$ for any
$\eta$-quantized approximation of \E, is upper semi-continuous w.r.t. $\eta$ at $\eta= 0$.
\end{lem}
\begin{proof}
Let $M_k$ be a sequence in $\mathbb{R}^{\Crd(Q)}$ such that $M_k\big \vert_i$ denotes the measure  of the expected state after $k \in \mathbb{N} \cup \{0\}$ observations for the chosen 
$\eta$-quantized approximation $\mathscr{E}^\eta_{(G,\p)}$ beginning from the pure state corresponding to 
$q_i \in Q$. We note that:
\begin{gather}
 \sum_{k=0}^\infty M_k = \boldsymbol{\nu}_\theta^\eta
\intertext{Furthermore, we have:}
 M_0 = A\chi^{[0]}
\end{gather}
where $A = \theta\B$ and $\chi^{[0]}$ is the perturbation of the characteristic vector $\chi$ due to quantization, implying that
\begin{gather}
 \vert \vert M_0 - A\chi \vert \vert_\infty \leqq \vert \vert A \vert \vert_\infty \eta
\end{gather}
Denoting $B = \B(1-\theta)\bigg (\Pi - \mathscr{P}(\Pi)\bigg )$, we note:
\begin{gather}
 M_k = B^kA\chi^{[k]}
\Longrightarrow \infnrm{M_k} \leqq \infnrm{B}^k \infnrm{A} \eta
\intertext{It then follows that we have:}
\infnrm{\boldsymbol{\nu}_\theta^\eta - \boldsymbol{\nu}_\theta} \leqq \bigg (\sum_k \infnrm{B}^k \bigg ) \infnrm{A} \eta \label{eq63}
\end{gather}
We claim that the following bounds are satisfied:
\begin{enumerate}
 \item $\displaystyle\infnrm{A} \leqq 1$
\item $\displaystyle\sum_k \infnrm{B}^k \leqq \frac{1}{\theta}$
\end{enumerate}
For the first claim, we note 
\begin{multline}
 \B = \sum_{k=0}^\infty \theta (1-\theta)^k\mathscr{P}(\Pi)^k \notag \\
\leqq_{ \textrm{\sffamily \textsc{Elementwise}}} \sum_{k=0}^\infty \theta (1-\theta)^k\Pi^k = \theta\big [ \mathbb{I} - (1 -\theta)\Pi \big ]^{-1}\notag
\end{multline}
The result then follows by noting that $\theta\big [ \mathbb{I} - (1 -\theta)\Pi \big ]^{-1}$ is  a stochastic matrix for all $\theta \in(0,1)$.
For the second claim, denoting $\boldsymbol{e} =[ 1 \cdots 1]^T$,  we conclude from stochasticity of $\Pi$:
\begin{gather}
 \big (\Pi - \P \big )  \boldsymbol{e}= \big [ \mathbb{I} - \P \big ]\boldsymbol{e} = \big [ \mathbb{I} - (1-\theta)\P \big ]\boldsymbol{e} - \theta\P\boldsymbol{e}\notag\\
\begin{split}
\Rightarrow\big [ \mathbb{I} - (1-\theta)\P \big ]^{-1}\big (\Pi - &\P \big ) \boldsymbol{e}
\\=  \boldsymbol{e}   &- \P \theta\big [ \mathbb{I} - (1-\theta)\P \big ]^{-1}\boldsymbol{e}
\end{split}
\notag \\
\Rightarrow \frac{1}{1-\theta}B\boldsymbol{e} = \bigg \{  \mathbb{I} - \theta\big [ \mathbb{I} - (1-\theta)\P \big ]^{-1}\bigg \}\boldsymbol{e}
+\theta\boldsymbol{e} \label{eq65}
\end{gather}
Since $B$ is a non-negative matrix, it follows from Eq.~\eqref{eq65} that: 
\begin{gather}
 \infnrm{ \frac{1}{1-\theta}B} = 1- \min_i \Big \{ \theta\big [ \mathbb{I} - (1-\theta)\P \big ]^{-1}\Bigg \vert_i \Big \} + \theta \notag\\
\intertext{Noting that $\displaystyle\theta\big [ \mathbb{I} - (1-\theta)\P \big ]^{-1} = \theta + \theta\sum_{k=1}^\infty \left ((1-\theta)\P\right )^k$, }
\infnrm{ \frac{1}{1-\theta}B} \leqq 1  - \theta + \theta \Rightarrow\infnrm{ \frac{1}{1-\theta}B} \leqq 1 \Rightarrow \infnrm{ B} \leqq 1-\theta \notag \\
\Rightarrow \sum_{k=0}^\infty \infnrm{B}^k \leqq \frac{1}{1 - (1-\theta)}  =    \frac{1}{\theta}
\end{gather}
Noting that $\boldsymbol{\nu}_\theta^0 = \boldsymbol{\nu}_\theta$ and $\theta > 0$, we conclude from Eq.~\eqref{eq63}:
\begin{gather}
 \forall \eta > 0, \infnrm{\boldsymbol{\nu}_\theta^\eta - \boldsymbol{\nu}_\theta^0} \leqq \eta \frac{1}{\theta}
\end{gather}
which implies that $\boldsymbol{\nu}_\theta^\eta$  is upper semi-continuous w.r.t. $\eta$ at $\eta= 0$. This completes the proof.
\end{proof}

\begin{lem}\label{lemoptpartial}
For  any plant \G with an unobservability map $\p$:
 the control actions generated by Algorithm~\ref{PrAlg3}
is optimal in the sense that
\begin{gather}
\mathbf{E} \left
(\int_0^t\hat{\nu}_\theta^{\star}(\tauup)\mathrm{d}\tauup \right
)\geqq \mathbf{E} \left
(\int_0^t\hat{\nu}_\theta^{\#}(\tauup)\mathrm{d}\tauup\right ) \
\forall t \in [0,\infty), \forall \theta \in (0,1)
\end{gather}
where $\hat{\nu}_\theta^\star(t)$ and $\hat{\nu}_\theta^{\#}(t)$ are the
instantaneous measures at time $t$ for control actions generated by
Algorithm~\ref{PrAlg3} and an arbitrary policy
respectively.
\end{lem}
\begin{proof}
First, we note that it suffices to consider terminating plants \Gt such that $\theta \leqq \theta_{min}$
(See Definition~\ref{defthetamin}) for the purpose of defining the optimal supervision policy~\cite{CR07}. Algorithm~\ref{PrAlg3} specifies 
the optimal control policy for plants with termination probability $\theta$ when the set of entangled states is finite (Lemma~\ref{lemassum2}).
We claim that the result is true when this finiteness condition stated in Eq. \eqref{eqcont} is relaxed.
The argument is as follows: The optimal control policy as stated in Algorithm~\ref{PrAlg3}
for finite $Q_\mathscr{F}$ can be paraphrased as
\begin{itemize}
 \item \textit{ Maximize language measure for every state offline}
\item \textit{ Follow the measure gradient online}
\end{itemize}
Since  $\Crd(Q_\mathscr{F}^\eta) < \infty$, it follows from Lemma~\ref{lemassum2} that such a policy yields the optimal decisions for
an $\eta$-quantized approximation of \E for any $\eta > 0$. As we approach $\eta=0$, we note that it follows from continuity
that there exists $\eta_\star > 0$ such that the sequence of disabling decisions do not change for all $\eta \leqq \eta_\star$
implying that the optimally controlled transition sequence is identical for all $\eta \leqq \eta_\star$. Since
it is guaranteed by Definition~\ref{defetaapprox} that for identical transition sequences, 
quantized entangled states $\alpha_\eta^{[k]}$ are within $\eta$-balls
of actual entangled state $\alpha^{[k]}$ after the $k^{th}$ observation, we conclude
\begin{gather}
 \forall k, \forall \eta \in(0,\eta_\star], \infnrm{\alpha^{[k]}_{\eta} - \alpha^{[k]}} \leqq \eta
\end{gather}
It therefore follows that for any control policy, we have
\begin{multline}
 \forall \eta \in (0,\eta_\star], \ 
\left\vert\int_0^t\hat{\nu}_\theta^{\eta}(\tauup)\mathrm{d}\tauup - \int_0^t\hat{\nu}_\theta(\tauup)\mathrm{d}\tauup\right \vert \\
\leqq \int_0^t \Big \vert \langle \alpha_\eta^{[k]},\boldsymbol{\nu}_\theta^{\eta}\rangle
 - \langle \alpha^{[k]},\boldsymbol{\nu}_\theta\rangle  \Big \vert\mathrm{d}\tauup 
\leqq \eta \left (1 + \frac{1}{\theta} + \frac{ 1}{\theta^2}\right )t
\end{multline}
implying that $\int_0^t\hat{\nu}_\theta^{\eta}(\tauup)\mathrm{d}\tauup$ is semi-continuous from above at $\eta = 0$ which completes the proof.
\end{proof}
\begin{prop}\label{propoptpartial}
 Algorithm~\ref{Algopoobsopt} correctly implements the optimal control policy for an arbitrary finite state plant \G with specified unobservability map $\p$.
\end{prop}
\begin{proof}
 We first note that Algorithm \ref{Algopoobsopt} is a detailed  restatement of Algorithm~\ref{PrAlg3} with the exception of the normalization step in
Lines 20 and 22. On account of non-negativity of any entangled state $\alpha$ and the fact $\alpha \neq \boldsymbol{0}$ (See Lemma~\ref{lemobsstate}), we have:
\begin{gather}
\mathop{sign}\left (\alpha \big ( \Gamma^\sigma  - \Gamma^\sigma_\mathscr{D} \big )\right ) 
= \mathop{sign}\left(\mathcal{N}\left ( \alpha \right )\big ( \Gamma^\sigma  - \Gamma^\sigma_\mathscr{D} \big )\right )
\end{gather}
which verifies the the normalization steps.
The result then follows immediately from Lemma~\ref{lemoptpartial}.
\end{proof}
\begin{rem}
The normalization steps in Algorithm~\ref{Algopoobsopt} serve to mitigate numerical problems.  Lemma~\ref{lemobsstate} guarantees that the entangled state $\alpha \neq \boldsymbol{0}$. However, repeated right multiplication by the transition matrices may result in  entangled states with norms arbitrarily close to $0$ leading to 
numerical errors in comparing arbitrarily close floating point numbers. Normalization partially remedies this by ensuring that the entangled states used for the comparisons are 
sufficiently separated from $\boldsymbol{0}$. There is, however, still the issue of approximability and even with normalization, we may be needed to compare
arbitrarily close values. The next proposition addresses this by showing that, in contrast to MDP based models, the optimization algorithm for PFSA is indeed
$\Eps$-approximable~\cite{LGM01}, $i.e.$ deviation from the optimal policy is guaranteed to be small for small errors in value comparisons in Algorithm~\ref{Algopoobsopt}.
This further implies that the optimization algorithm is robust under small parametric uncertainties in the model as well as to errors arising from finite precision arithmetic in digital computer implementations.
\end{rem}
\vspace{0pt}
\begin{algorithm}[!ht]
 \small \SetLine
  \SetKwData{Left}{left}
  \SetKwData{This}{this}
  \SetKwData{Up}{up}
  \SetKwFunction{Union}{Union}
  \SetKwFunction{FindCompress}{FindCompress}
    \SetKw{Tr}{true}
   \SetKw{Tf}{false}
  \SetKwInOut{Input}{input}
  \SetKwInOut{Output}{output}
  \caption{Optimal Control under Partial Observation (Finalized Version)}\label{Algopoobsopt}
\Input{\G,$\p$}
\Output{Optimal Control Actions} 
\vspace{3pt} {\blue \hrule} \vspace{3pt}
\Begin( \tcc*[f]{{\blue \sffamily \textsc{Offline Execution}}}){
	Compute $\boldsymbol{\nu}_\star$\; 
	Set $\theta = \theta_{min}$\;
	Compute $M = \big [ \mathbb{I} - (1-\theta_{min})\mathscr{P}(\Pi) \big ]^{-1}$\;
	\For{$\sigma \in \Sigma$}{
		Compute $\Gamma^\sigma$\tcc*[r]{Algorithm~\ref{AlgGammaFNO}}
		Compute $\Gamma^\sigma_\mathscr{D}$\;
		Compute $T^\sigma = \big [\Gamma^\sigma - \Gamma^\sigma_\mathscr{D}\big ]\boldsymbol{\nu}_\star$\tcc*[r]{Column Vector}
	}
	$\mspace{-10mu}\displaystyle \begin{array}{cc}\mathrm{Initialize} \ \alpha_0 = [0 \cdots 1 \cdots 0]\\
             \mspace{20mu}\mathrm{(i_0^{th} \ element)} \uparrow
            \end{array}
	$\tcc*[r]{Init. state: $q_{i_0}$}
	Compute $\alpha = \alpha_0 M$\tcc*[r]{For $\omega$ s.t. $\p(q_i,\omega) = \epsilon$}
\vspace{3pt} {\red \hrule} \vspace{3pt}
	\While(\tcc*[f]{{\red \sffamily \textsc{Online Execution}}}){\Tr} 
	{ 
		\For{$\sigma \in \Sigma$ } 
		{
		\If{$\alpha T^\sigma < 0$}
		{
			Disable $\sigma$\tcc*[r]{{\color{ForestGreen}Control Action}}
		}
 		}
		Observe event $\sigma$\;   
		\eIf {$\sigma$ is disabled}
		{
			$\alpha = \mathcal{N}\big (\alpha \Gamma^\sigma_\mathscr{D} \big )$\;
		}
		{
			$\alpha = \mathcal{N}\big (\alpha \Gamma^\sigma \big )$\;
		}
	}
  }
\end{algorithm}
\vspace{0pt}
\begin{prop}\label{propapprox}
(Approximability) In a finite precision implementation of
Algorithm~\ref{Algopoobsopt}, with real numbers distinguished upto $\Eps > 0$,
$i.e.$,
\begin{gather}
 \forall a,b \in \mathbb{R},  \vert a -b \vert \leqq \Eps \Rightarrow a - b \equiv 0
\end{gather}
we have $\forall t \in [0,\infty), \forall \theta \in (0,1)$,
 \begin{gather}
 0 \leqq \mathbf{E} \left (\int_0^t\hat{\nu}_\theta^{\star}(\tauup)\mathrm{d}\tauup \right
)- \mathbf{E} \left
(\int_0^t\hat{\nu}_\theta^{\#}(\tauup)\mathrm{d}\tauup\right )  < \Eps
 \end{gather}
where $\hat{\nu}_\theta^{\star}(t)$ and $\hat{\nu}_\theta^{\#}(t)$ are
 the instantaneous measures at time $t$ for the exact ($i.e.$ infinite precision)
and approximate (upto $\Eps$-precision) implementations of the optimal policy respectively.
\end{prop}
\begin{proof}
Let \Gt be the underlying plant. First we consider 
the perfectly observable case, $i.e.$, 
 with every transition observable at the supervisory level.
Denoting the optimal and approximate measure vectors
obtained by Algorithm~\ref{Algorithm02} as $\boldsymbol{\nu}_\theta^\star$ and 
$\boldsymbol{\nu}_\theta^\#$, we claim:
\begin{gather}\label{eqclm}
 \boldsymbol{\nu}_\theta^\star-\boldsymbol{\nu}_\theta^\# \leqq_{\textrm{\sffamily \textsc{Elementwise}}} \Eps
\end{gather}
Using the algebraic structure of the Monotonicity Lemma~\cite{CR07} (Also see Lemma~\ref{propmonotone}), we obtain:
\vspace{-3pt}
\begin{align*}
 \boldsymbol{\nu}_\theta^\star-\boldsymbol{\nu}_\theta^\# = \theta
\big [ \mathbb{I} - (1-\theta) \Pi^\star \big 
]^{-1}(1-\theta) &\underbrace{\big [ \Pi^\star - \Pi^\# \big ]\boldsymbol{\nu}_\theta^\#} \\
& \mspace{50mu} \mathbf{M}
\end{align*}
We note that it follows from the exact optimality of $\boldsymbol{\nu}_\theta^\star$ that
\begin{gather}
  \boldsymbol{\nu}_\theta^\star-\boldsymbol{\nu}_\theta^\# \geqq_{\textrm{\sffamily \textsc{Elementwise}}} 0
\end{gather}
Denoting the $i^{th}$ row of the matrix $\mathbf{M}$ as $\mathbf{M}_i$, 
we note that $M_i$ is of the form $\sum_{j=1}^{\Crd(Q)} a_jb_j$ where 
\begin{gather}
 a_j \leqq \Pi_{ij} \label{eq69}\\
b_j =  \big \vert \boldsymbol{\nu}_\theta^\#\vert_i
- \boldsymbol{\nu}_\theta^\# \vert_j \big \vert 
\end{gather}
We note that the inequality in Eq. ~\eqref{eq69} follows from the fact 
that event enabling and disabling is a redistribution of the \textit{controllable}
part of the unsupervised transition matrix $\Pi$. Also, since
$\Pi^\#$, was obtained via $\Eps$-precision optimization, we have:
\begin{subequations}
\begin{gather}
 \left ( \boldsymbol{\nu}_\theta^\#\vert_i
> \boldsymbol{\nu}_\theta^\# \vert_j \right ) \bigwedge q_j
 \xrightarrow[\mathrm{disabled}]{\sigma}q_i  \Rightarrow \big \vert
 \boldsymbol{\nu}_\theta^\#\vert_i
- \boldsymbol{\nu}_\theta^\# \vert_j \big \vert \leqq \Eps \\
 \left ( \boldsymbol{\nu}_\theta^\#\vert_i
\leqq \boldsymbol{\nu}_\theta^\# \vert_j \right ) \bigwedge q_j
 \xrightarrow[\mathrm{enabled}]{\sigma}q_i  \Rightarrow \big \vert
 \boldsymbol{\nu}_\theta^\#\vert_i
- \boldsymbol{\nu}_\theta^\# \vert_j \big \vert \leqq \Eps
\end{gather}
\end{subequations}
It therefore follows from stochasticity of $\Pi$ that:
\begin{gather}
 \infnrm{M} < \infnrm{\Pi} \Eps = \Eps
\end{gather}
Hence , noting that $\infnrm{\theta
\big [ \mathbb{I} - (1-\theta) \Pi^\star \big 
]^{-1}}=1$, we have:
\begin{gather}
\infnrm{\boldsymbol{\nu}_\theta^\star-\boldsymbol{\nu}_\theta^\# } \leqq
\infnrm{\theta
\big [ \mathbb{I} - (1-\theta) \Pi^\star \big 
]^{-1}}\times \vert 1 - \theta \vert \times  \Eps < \Eps
\end{gather}
which proves the claim made in Eqn.~\eqref{eqclm}. It then follows 
from Lemma~\ref{lemremmeas}, that for the perfectly observable case we have
$\forall t \in [0,\infty), \forall \theta \in (0,1)$,
 \begin{gather}
 0 \leqq \mathbf{E} \left (\int_0^t\hat{\nu}_\theta^{\star}(\tauup)\mathrm{d}\tauup \right
)- \mathbf{E} \left
(\int_0^t\hat{\nu}_\theta^{\#}(\tauup)\mathrm{d}\tauup\right )  < \Eps
 \end{gather}
We recall that for a finite entangled state set \Qe, the entangled transition system can 
be viewed as a perfectly observable terminating plant 
(See Remark~\ref{rem4p3}) with possibly partial controllability implying
that we must apply the Generalized Monotonicity Lemma (See Lemma~\ref{propmonotone}). Noting that 
the above argument is almost identically applicable for
the Generalized Monotonicity Lemma,  it follows 
that the above result is true for any non-trivial 
unobservability map on the underlying plant
$G_\theta$ satisfying $\Crd(\QE) < \infty$. The extension to the
general case of infinite \Qe then follows from the application 
of the result to $\eta$-approximations of the entangled transition 
system for $\eta \leqq \eta_\star$ (See Lemma~\ref{lemoptpartial} for explanation of the bound $\eta_\star$) and recalling the continuity 
argument stated in Lemma~\ref{lemoptpartial}. This completes the proof.
\end{proof}
\vspace{3pt}

The performance of MDP or POMDP based models is computed as the total reward garnered by the agent in the course of operation. The analogous
notion for PFSA based modeling is  the expected value of integrated instantaneous characteristic $\int_0^t \hat{\chiup}(\tauup)\mathrm{d}\tauup$
(See Definition~\ref{definstchar}) as a function of operation time.
\vspace{3pt}
\begin{prop}
 \label{propmaxchar}(Performance Maximization:) The optimal control policy stated in 
Algorithm~\ref{Algopoobsopt} maximizes
infinite horizon performance in the sense of maximizing the expected integrated instantenous state characteristic (See Definition~\ref{definstchar}), $i.e.$,
\begin{gather}\label{eqchimax2}
\forall t \in[0,\infty), \ \mathbf{E} \left
(\int_0^t\hat{\chiup}^{\star}(\tauup)\mathrm{d}\tauup \right
)\geqq \mathbf{E} \left
(\int_0^t\hat{\chiup}^{\#}(\tauup)\mathrm{d}\tauup \right
) 
\end{gather}
where the instantenous characteristic, at time $t$, for the optimal ($i.e.$ as defined by Algorithm~\ref{Algopoobsopt})
and an arbitrary supervision policy is denoted by  $\hat{\chiup}^{\star}(t)$
and $\hat{\chiup}^{\#}(t)$ respectively. 
\end{prop}
\begin{proof}
 We recall that the result is true for the case of perfect observation (See 
Eq.~\eqref{eqchimax}).
Next we recall from Remark~\ref{rem4p3}, that if 
the unobservability map is non-trivial, but has a finite $Q_\mathscr{F}$, 
then the entangled transition system $\mathscr{E}_{(G,\p)} $ 
can be viewed as a \textbf{perfectly observable} terminating model with uniform 
termination probability $\theta$. It therfore follows, that for such cases, we have:
\begin{gather}\label{eqchimaxEE}
\forall t \in[0,\infty), \ \mathbf{E} \left
(\int_0^t\hat{\chiup}_\mathscr{E}^{\star}(\tauup)\mathrm{d}\tauup \right
)\geqq \mathbf{E} \left
(\int_0^t\hat{\chiup}_\mathscr{E}^{\#}(\tauup)\mathrm{d}\tauup \right
) 
\end{gather}
We recall from the definition of entangled transition systems (See Definition~\ref{defentangtran}),
\begin{gather}
 \chi_\mathscr{E}(t) = \langle \alpha(t) ,\boldsymbol{\chi}\rangle
\end{gather}
where $\alpha(t)$ is the entangled state at time $t$, which in turn implies
that we have:
\begin{gather}
 \mathbf{E}(\chi_\mathscr{E}) = \langle \mathbf{E}(\alpha), \boldsymbol{\chi}\rangle
\end{gather}
Since $\mathbf{E}(\alpha)\vert_i$ is the expected sum of 
conditional probabilities of strings terminating on state $q_i$ of the
underlying plant, we conclude that $\mathbf{E}(\alpha)$ is  in fact
the stationary state probability vector corresponding to the underlying plant.
Hence it follows that 
$\mathbf{E}(\chi_\mathscr{E}) = \mathbf{E}(\chi)$ implying that 
for non-trivial unobservability maps that guarantee $\QE < \infty$, we have
\begin{gather}\label{eqchimax3}
\forall t \in[0,\infty), \ \mathbf{E} \left
(\int_0^t\hat{\chiup}^{\star}(\tauup)\mathrm{d}\tauup \right
)\geqq \mathbf{E} \left
(\int_0^t\hat{\chiup}^{\#}(\tauup)\mathrm{d}\tauup \right
) 
\end{gather}
The general result for infinite entangled state sets ($i.e.$
for unobservability maps which fail to gurantee $\QE < \infty$)
 follows from applying the above 
result to $\eta$-approximations (See Definition~\ref{defetaapprox})
of the entangled transition system and recalling the continuity result of Lemma~\ref{lemoptpartial}.
\end{proof}
\subsection{Computational Complexity}\label{subseccomplexity}
Computation of the supervision policy for an underlying plant with a non-trivial
unobservability map requires computation of $\boldsymbol{\nu}_\star$ (See Step 2 of Algorithm~\ref{Algopoobsopt}), $i.e.$, we need to execute Algorithm~\ref{Algorithm02}
first. It was conjectured and validated via extensive simulation in \cite{CR07} that
Algorithm~\ref{Algorithm02} can be executed with polynomial asymptotic runtime complexity.
Noting that each of the remaining steps of Algorithm~\ref{Algopoobsopt} can be executed 
with worst case complexity of $n \times n$ matrix inversion (where $n$ is the
size of the state set $Q$ of the underlying model), we conclude that 
the overall runtime complexity of 
proposed supervision algorithm 
is polynomial in number of underlying model states. Specifically, we have the 
following result:
\begin{prop}\label{propcomplexity}
 The runtime complexity of the offline portion of Algorithm~\ref{Algopoobsopt} ($i.e.$ upto line number 11) is same as that of Algorithm~\ref{Algorithm02}.
\end{prop}
\begin{proof}
The asymptotic runtime complexity of Algorithm~\ref{Algorithm02}, as shown in \cite{CR07}, is $M(n\times n)\times O(I)$ where 
$M(n\times n)$ is the complexity of $n\times n$ matrix inversion and $O(I)$ is the asymptotic bound on the number of iterations on Algorithm~\ref{Algorithm02}.
The proof is completed by noting that the  complexity of executing lines 3 to 11 of Algorithm~\ref{Algopoobsopt} is $M(n\times n)$.
\end{proof}
\vspace{3pt}
\begin{rem}
It is immediate that the online portion of Algorithm~\ref{Algopoobsopt} has the runtime complexity of Matrix-Vector multiplication.
It follows that the measure-theoretic optimization of partially observable plants is no harder to solve that those with perfect observation.
\end{rem}

The results of this section establish the following facts:
\begin{enumerate}
 \item Decision-theoretic processes modeled in the PFSA framework 
can be efficiently optimized via maximization of the corresponding
language measure.
\item The optimization problem for infinite horizon problems is shown to be 
$\Eps$-approximable,
and  the solution procedure presented in this paper is robust to modeling uncertainties and
computational approximations. This is a significant advantage over 
POMDP based modeling, as discussed in details in Section~\ref{secintroneg}.
\end{enumerate}
%
%
\section{Verification In Simulation Experiments}\label{secvalid}
The theoretical development of the previous sections
is next validated on two simple decision problems.

The first example consists of a four state mission execution model. 
The underlying plant is illustrated in Figure~\ref{figsim}.
The physical interpretation of the states and events is enumerated in
Tables~\ref{tabsim1} and \ref{tabsim2}. $G$ is the ground, initial or mission 
abort state. We assume the mission to be important; hence abort is
assigned a negative characteristic value of $-1$. $M$ represents
correct execution and therefore has a positive characteristic of $0.5$.
The mission moves to state $E$ on encountering possible
system faults (event $d$) from states $G$ and $M$. Any further
system faults or an attempt to execute
the next mission step under such  error conditions results
in a transition to the critical state $C$. The only way to correct the 
situation is to execute fault recovery protocols denoted by $r$. 
However, execution of $r$ from
the correct mission execution state $M$ results in an abort.  
Occurrence of system faults $d$ are uncontrollable from every state.
Furthermore, under system criticality, we have sensor failure 
resulting in unobservability of further system faults and success
of recovery attempts, $i.e.$, the events $d$ and $r$ are unobservable from state $C$.
\begin{figure}[!h]
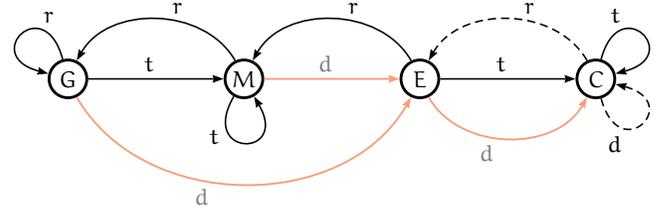

\centering
\VCDraw[1.3]{%
\begin{VCPicture}{(-1,-2.5)(10,2)}
\State[G]{(0,0)}{G} \State[M]{(3,0)}{M}
\State[E]{(6,0)}{E} \State[C]{(9,0)}{C}
\EdgeL{G}{M}{t}
\LoopNW{G}{r}
\DimEdge\SetEdgeLineColor{Melon}\EdgeL{M}{E}{d} \SetEdgeLineColor{black}\RstEdge
\EdgeL{E}{C}{t}
\LoopS{M}{t}
\LoopNE{C}{t}
\SetLArcAngle{60}
\DimEdge\SetEdgeLineColor{Melon}\LArcR{E}{C}{d} \SetEdgeLineColor{black}\RstEdge
\LArcR{M}{G}{r}
\LArcR{E}{M}{r}
\SetEdgeLineStyle{dashed}
\LoopSE{C}{d}
\LArcR{C}{E}{r}
\SetEdgeLineStyle{solid}
\SetLArcAngle{60}\DimEdge\SetEdgeLineColor{Melon}
\LArcR{G}{E}{d} \SetEdgeLineColor{black}\RstEdge
\SetLArcAngle{30}
\end{VCPicture}}
\caption{Underlying plant model with four states $Q = \{G,M,E,C\}$ and alphabet $\Sigma = \{t,r,d\}$: unobservable
 transitions are denoted by dashed
arrows ($ \boldsymbol{-- \mspace{-5mu}\rightarrow} $); uncontrollable but 
observable transitions are shown dimmed ($\color{Melon} \leftarrow $).}\label{figsim}
\end{figure}
\begin{table}[!ht]
\begin{minipage}{3.5in}
\caption{State Descriptions, Event Occurrence Probabilities \& Characteristic Values}\label{tabsim1}
\centering\begin{tabular}{||c|l|ccc|c||}
\hline & \sffamily \textsc{\bf Physical Meaning} & t &r & d & $\boldsymbol{\chi}$%
\\ \hline \hline $G$ & \sffamily \textsc{Ground/Abort} &$0.8$ & $0.05$ & $0.15$ & $-1.00 $%
\\\hline $M$ & \sffamily \textsc{Correct Execution} &$0.5$ & $0.30$ & $0.20$ & $\phantom{-}0.50 $%
\\ \hline $E$ & \sffamily \textsc{System Fault} &$0.5$ & $0.20$ & $0.30$ & $-0.20$%
\\ \hline $C$ & \sffamily \textsc{System Critical}&$0.1$ & $0.10$ & $0.80$ & $-0.25$%
\\ \hline
\end{tabular}
\end{minipage}\vspace{10pt}
\begin{minipage}{3.5in}
 \caption{Event Descriptions}\label{tabsim2}
\centering\begin{tabular}{||c|l|}
\hline & \sffamily \textsc{\bf Physical Meaning} \\ \hline %
\hline t & \sffamily \textsc{Execution of  Next Mission Step/Objective Successful}\\%
\hline r & \sffamily \textsc{Execution of Repair/Damage Recovery Protocol} \\ %
\hline d & \sffamily \textsc{System Fault Encountered}\\ %
\hline%
\end{tabular}
\end{minipage}
\end{table}\vspace{0pt}

The event occurrence probabilities are tabulated in Table~\ref{tabsim1}. We note 
that the probabilities of successful execution of mission steps (event $t$) and
damage recovery protocols (event $r$) are both small under system criticality in state 
$C$. Also, comparison of the event probabilities from states $M$ and $E$ 
reveals that the probability of encountering further errors is higher once some error has already occurred and the
probability of successful repair is  smaller. 

We simulate the controlled execution of the above described mission under the following three 
strategies:
\begin{enumerate}
 \item Null controller: No control enforced
\item Optimal control under perfect observation: Control enforced using Algorithm~\ref{Algoonlineopt} given that all transitions
are observable at the supervisory level
\item Optimal control under partial observation: Control enforced using Algorithm~\ref{Algopoobsopt} given the
above described unobservability map
\end{enumerate}

The optimal renormalized measure vector of the system under full observability is computed
to be $[ -0.0049 \ -0.0048 \ -0.0049 \ -0.0051]^T$ . Hence we observe  in Figure~\ref{figcontrol} that the
gradient of the instantaneous measure under perfect observation converges to around $0.005$. We note that the gradient for the 
instantaneous measure under partial observation converges close to the former value. The null controller,
of course, is the significantly poor.
\begin{figure}[!ht]
\flushleft
 \includegraphics[width=3in]{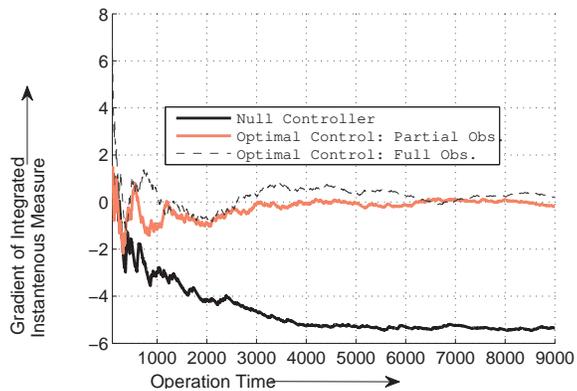}
\caption{Gradient of integrated instantaneous measure as a function of operation time}\label{figcontrol}
\end{figure}

The performance of the various control strategies are compared based on the expected value of the
integrated instantenous characteristic $\mathbf{E} \left
(\int_0^t\hat{\chiup}(\tauup)\mathrm{d}\tauup \right
)$.
The simulated results are shown in Figure~\ref{figcontrolchi}. 
The null controller performs worst;  and the optimal control strategy under perfect observation performs best. As expected
the strategy in which we blindly use the optimal control for perfect observation 
(Algorithm~\ref{Algoonlineopt}) under the given non-trivial unobservability map is 
exceedingly poor  and close-to-best performance is recovered using the optimal control algorithm under partial observation.
\begin{figure}[!ht]
\centering
 \includegraphics[width=3in]{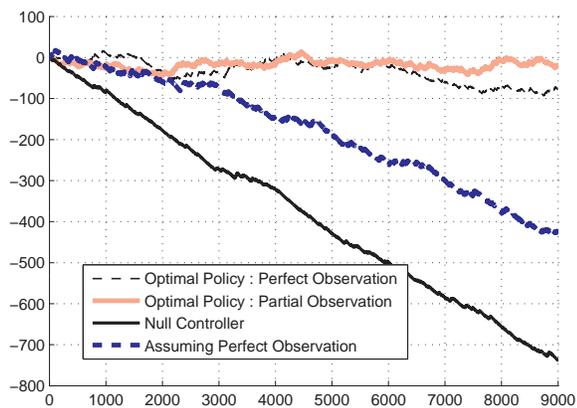}
\caption{Performance as a function of operation time}\label{figcontrolchi}
\end{figure}
%

The second example is one that originally appeared in the context of POMDPs in \cite{Cass94}. The
physical specification of the problem is as follows: The player is given a choice between 
opening one of two closed doors; one has a reward in the room behind it, the other has a tiger.
Entering the latter incurrs penalty in the form of bodily injury. The player can also choose
to \textit{listen} on the doors; and attempt to figure out which room has the tiger. The game resets
after each play; and the tiger and the reward is randomly placed in the rooms at the beginning of each such play.
Listening on the doors doesnot enable the player to accurately determine the location of the tiger; it merely
makes her odds better. However, listening incurrs a penalty; it costs the player if she chooses to listen.
The scenario is pictorially illustrated in the top part of Figure~\ref{figsimT}. We model the physical situation 
in the PFSA framework as shown in the bottom part of Figure~\ref{figsimT}.

\begin{figure}[!h]
\centering
 \includegraphics[width=1in]{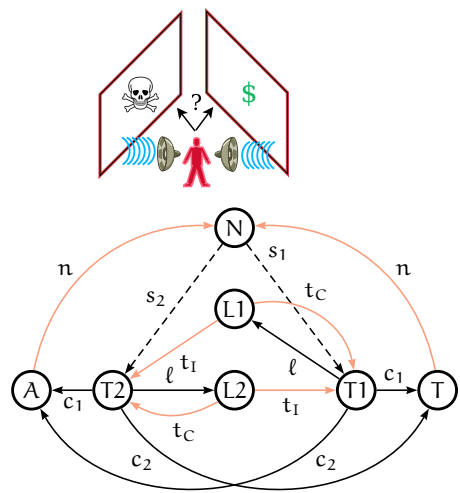}
\VCDraw[0.9]{%
\begin{VCPicture}{(-7,-3.5)(5,5)}
\State[A]{(-5,0)}{A} \State[T2]{(-3,0)}{T2}
\State[L2]{(0,0)}{L2} \State[T1]{(3,0)}{T1} 
\State[T]{(5,0)}{T} \State[L1]{(0,2)}{L1} \State[N]{(0,4)}{N}
\EdgeL{T2}{A}{c_1}
\EdgeL{T2}{L2}{\ell}
\SetEdgeLineColor{Melon}
\EdgeL{L2}{T1}{}\LabelR{t_I}
\SetEdgeLineColor{black}
\EdgeL{T1}{T}{c_{1}}
\SetEdgeLineStyle{dashed}
\EdgeL{N}{T2}{}\LabelR[0.5]{s_{2}}
\EdgeL{N}{T1}{}\LabelL[0.15]{s_{1}}
\SetEdgeLineStyle{solid}
\SetEdgeLineColor{Melon}
\EdgeL{L1}{T2}{}\LabelL{t_I}
\SetEdgeLineColor{black}
\EdgeL{T1}{L1}{}\LabelL{\ell}
\SetLArcAngle{30}
\SetEdgeLineColor{Melon}
\LArcL{L2}{T2}{t_C}
\SetEdgeLineColor{black}
\SetLArcAngle{30}
\SetLArcAngle{60}
\LArcL[0.15]{T1}{A}{c_2}
\LArcR[0.15]{T2}{T}{c_2}
\SetLArcAngle{30}
\SetLArcAngle{40}
\SetEdgeLineColor{Melon}
\LArcL{A}{N}{n}
\LArcR{T}{N}{n}
\SetLArcAngle{30}
\SetLArcAngle{45}
\LArcL{L1}{T1}{t_C}
\SetEdgeLineColor{black}
\SetLArcAngle{30}
\end{VCPicture}}
\caption{TOP: Illustration of the physical scenario, BOTTOM: Underlying plant model with seven states  and eight alphabet symbols: unobservable
 transitions are denoted by dashed
arrows ($ \boldsymbol{-- \mspace{-5mu}\rightarrow} $); uncontrollable but 
observable transitions are shown dimmed ($\color{Melon} \rightarrow $).}\label{figsimT}
\end{figure}
\begin{table}[!ht]
\begin{minipage}{1.5in}
\caption{State Descriptions}\label{tabsimT1}
\centering\begin{tabular}{||c|l|ccc|c||}
\hline & \sffamily \textsc{\bf Physical Meaning} %
\\ \hline \hline $N$ & \sffamily \textsc{Game Init} %
\\\hline $T1$ & \sffamily \textsc{Tiger in 1} %
\\ \hline $T2$ & \sffamily \textsc{Tiger in 2} %
\\ \hline $L1$ & \sffamily \textsc{Listen: Tiger in 1}%
\\ \hline $L2$ & \sffamily \textsc{Listen: Tiger in 2}%
\\ \hline $T$ & \sffamily \textsc{Tiger Chosen}%
\\ \hline $A$ & \sffamily \textsc{Award Chosen}%
\\ \hline
\end{tabular}
\end{minipage}\hspace{2pt}
\begin{minipage}{2in}
 \caption{Event Descriptions}\label{tabsimT2}
\centering\begin{tabular}{||c|l|}
\hline & \sffamily \textsc{\bf Physical Meaning} \\ \hline %
\hline $s_1$ & \sffamily \textsc{Tiger Placed in 1} {\red (unobs.)}\\%
\hline $s_2$ & \sffamily \textsc{Tiger Placed in 2} {\red (unobs.)}\\ %
\hline $\ell$ & \sffamily \textsc{Choose Listen} {\blue (cont.)}\\ %
\hline $t_{c}$ & \sffamily \textsc{Correct Determination}\\%
\hline $t_{I}$ & \sffamily \textsc{Incorrect Determination} \\ %
\hline $c_1$ & \sffamily \textsc{Choose 1} {\blue (cont.)}\\ %
\hline $c_2$ & \sffamily \textsc{Choose 2} {\blue (cont.)}\\ %
\hline $n$ & \sffamily \textsc{Game Reset}\\ %
\hline%
\end{tabular}
\end{minipage}
\end{table}\vspace{0pt}
  \begin{table}[!ht]
\caption{Event Occurrence Probabilities \& Characteristic Values}\label{tabsimT3}
\centering
\small
$\begin{array}{||c||c|cccccccc||ccccc}\hline
 & \chi  & s_1 \msp & s_2 \msp & \ell \msp & t_C \msp & t_I \msp & c_1 \msp & c_2 \msp &n \\\hline
 N  & \phantom{-}0.00  &     0.5    \msp &0.5         \msp &0         \msp &0         \msp &0                  \msp &0         \msp &0         \msp &0\\
    T1  & -0.25  &     0         \msp &0    \msp &0.33         \msp &0         \msp &0             \msp &0.33    \msp &0.33         \msp &0\\
   T2  & -0.25  &      0         \msp &0    \msp &0.33        \msp &0         \msp &0             \msp &0.33    \msp &0.33         \msp &0\\
   L1  & -0.75  &      0         \msp &0         \msp &0    \msp &0.8    \msp &0.2                 \msp &0         \msp &0         \msp &0\\
   L2  & -0.75  &      0         \msp &0         \msp &0         \msp &0.8         \msp &0.2             \msp &0         \msp &0         \msp &0\\
   T  & -1.00  &      0         \msp &0         \msp &0         \msp &0         \msp &0                  \msp &0         \msp &0    \msp &1\\
  A  &\phantom{-}1.00  &       0         \msp &0         \msp &0         \msp &0         \msp &0                  \msp &0         \msp &0    \msp &1\\\hline
 \end{array}$
\end{table}\vspace{0pt}

The PFSA has seven states $Q=\{N,T1,T2,L1,L2,T,A\}$ and eight alphabet symbols $\Sigma = \{ s_1,s_2, \ell,t_C,t_I,c_1,c_2,n\}$.
The physical meanings of the states and alphabet symbols are enumerated in Tables~\ref{tabsimT1} and \ref{tabsimT2} respectively.
The characteristic values and the event generation probabilities are tabulated in Table~\ref{tabsimT3}.
States $A$ and $T$ have characteristics of $1$ and $-1$ to reflect award and bodily injury. The listening states $L1$ and $L2$ also have negative
characteristic ($-0.75$) in accordance with the physical specification. An interesting point is the assignment of negative characteristic to 
the states $T1$ and $T2$; this prevents the player from choosing to disable \textit{all} controllable
moves from those states. Physically, this precludes the possibility that the player chooses to \textit{not} play at all and 
sits in either of those states forever; which may turn out to be the optimal course of action if the states $T1$ and $T2$ are not negatively weighted.

Figure~\ref{figcontrolmapT} illustrates the difference in the event disabling patterns resulting from the different strategies. We note that
the  the optimal controller under
perfect observation never disables  event $\ell$ (event no. 3), since the player never needs to \textit{listen} if she already knows which room has the
reward. In case of partial observation, the player decides to selectively listen to improve her odds. Also, note that the optimal policy under
 partial observation enables events significantly more often as compared to the optimal policy under perfect observation. The game actually proceeds via different 
routes in the two cases; hence it does not make sense to compare the 
control decisions after a given number of observation ticks; and  
the differences in the event disabling patterns must be interpreted only in an overall statistical sense.

We compare the simuation results in Figures~\ref{figcontrolT} and \ref{figcontrolchiT}. We note that in contrast to the
first example, the performance obtained for the optimally supervised partially observable case is significantly lower compared
to the situation under full observation. This arises from the physical problem at hand; it is clear 
that it is impossible in this case to have comparable performance in the two cases since the 
possibility of incorrect choice is significant and cannot be eliminated. The expected entangled state and the stationary probability vector
on the underlying model states is compared in Figure~\ref{figcompent} as an illustration for the result in Proposition~\ref{propmaxchar}.

\begin{figure}[t]
\centering
 \includegraphics[width=3.5in]{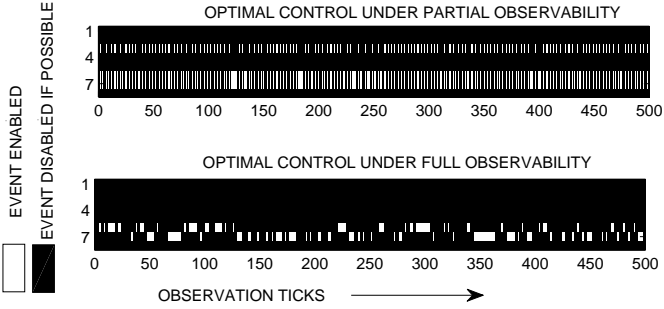}
\caption{Control maps as a function of observation ticks}\label{figcontrolmapT}
\end{figure}
\begin{figure}[!ht]
\flushleft
 \includegraphics[width=3in]{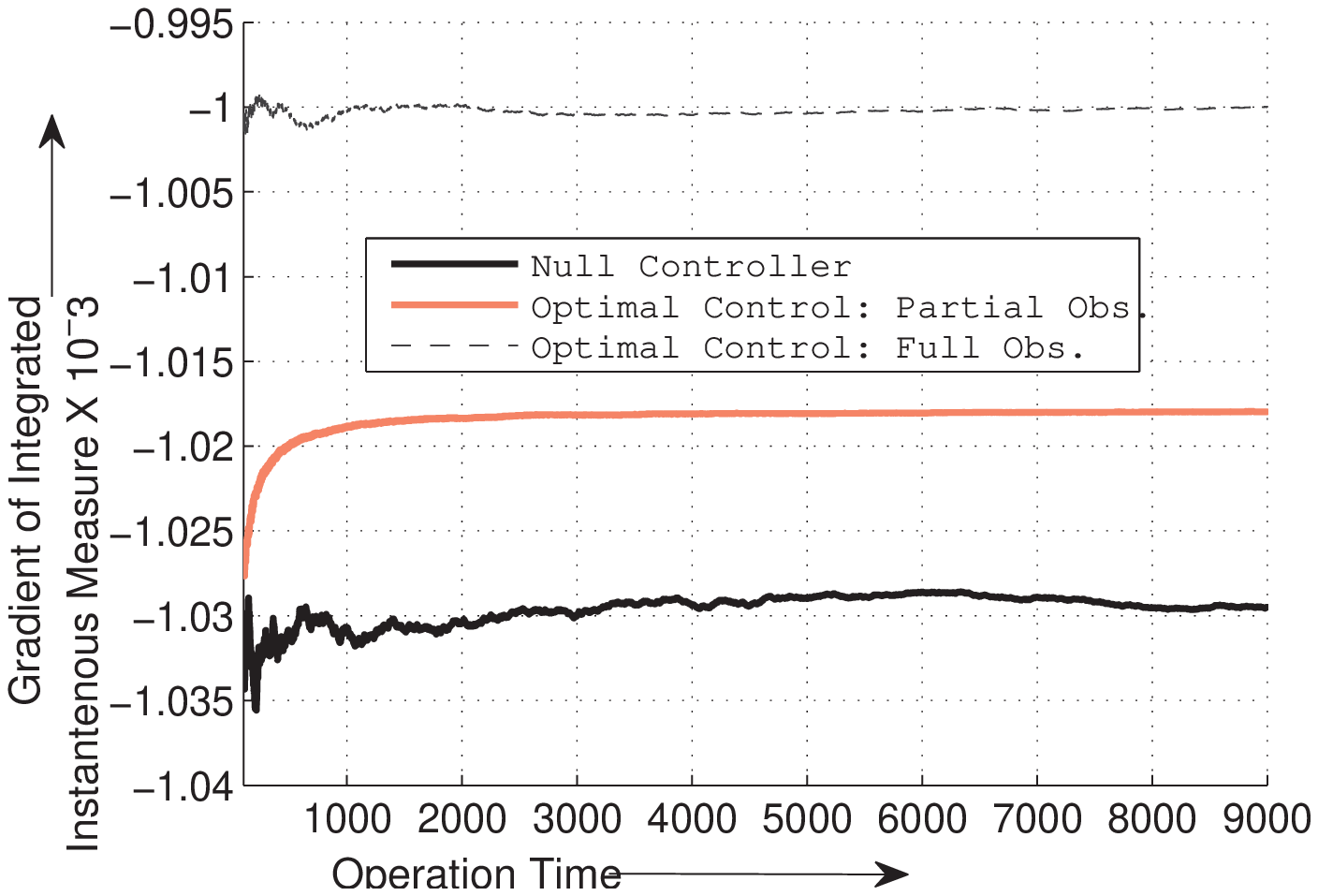}
\caption{Gradient of integrated instantaneous measure as a function of operation time}\label{figcontrolT}
\end{figure}

\begin{figure}[!ht]
\centering
 \includegraphics[width=3in]{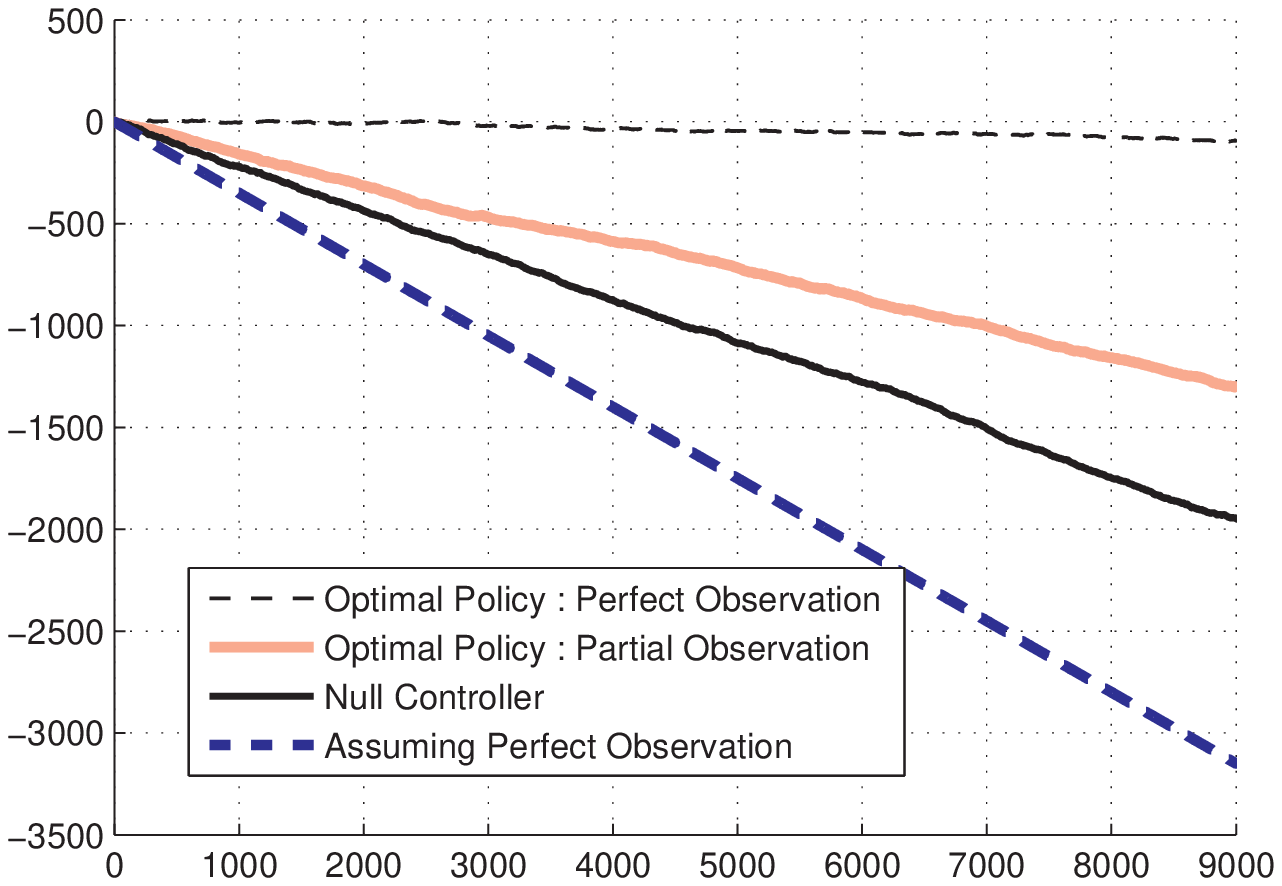}
\caption{Performance as a function of operation time}\label{figcontrolchiT}
\end{figure}
\begin{figure}[!ht]
\centering
 \includegraphics[width=3in]{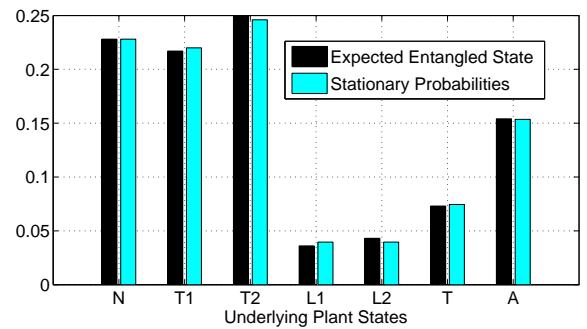}
\caption{Comparison of expected entangled state with the stationary probability vector on the underlying plant states for the optimal policy under partial observation}\label{figcompent}
\end{figure}
\section{Summary, Conclusions \& Future Work}\label{secsum}
In this paper we present an alternate framework based on probabilistic finite state machines (in the sense of Garg~\cite{G92,G92-2})
for modeling partially observable decision problems and
establish key advantages of the proposed approach over the current state of art. Namely, we show that,
the PFSA framework results in approximable
problems, $i.e.$, small changes in the model parameters or small numerical  errors result
 in small deviation in the obtained solution. Thus one is guranteed to obtain near optimal implementations
of the proposed supervision algorithm in a computationally efficient manner. This is a significant improvement over the
current state of art in POMDP analysis; several negative results exist that imply it is impossible to obtain 
a near-optimal supervision policy for arbitrary POMDPs in an efficient manner, unless certain complexity classes collapse (See detailed discussion
in Section~\ref{secintroneg}). The key tool used in this paper is the recently reported notion of renormalized measure of probabilistic regular languages.
We extend the measure theoretic optimization technique for perfectly observable probabilistic finite state automata to obtain an online implementable
supervision policy for finite state underlying plants, for which one or more transitions are unobservable at the supervisory level. It is further shown that
the proposed supervision policy maximizes the infinite horizon performance in a sense very similar to that generally used in the POMDP framework;
in the latter the optimal policy maximizes the total reward garnered by the plant in the course of operation, while in the former, it is shown, that
the expected value of the integrated instantaneous state characteristic is maximized. Two simple decision problems are included  as examples to illustrate the
theoretical development.
\subsection{Future Work}\label{subsecfutwork}
Future work will address the following areas:
\begin{enumerate}
\item Validation of the proposed algorithm in real-life systems with special emphasis of probabilistic robotics, and 
detailed comparison with the POMDP based approach with respect to computational complexity. 
 \item Generalization of the proposed technique to handle unobservability maps with unbounded memory; $i.e.$, unobservability maps that
result in infinite state phantom automata.
\item Adaptation of the proposed approach to solve finite horizon decision problems.
\end{enumerate}
}
\useRomanappendicesfalse
\appendices
\renewcommand{\thesection}{\Alph{section}}
\renewcommand{\thesectiondis}{\Alph{section}}
\renewcommand{\thesubsectiondis}{\Alph{section}.\arabic{subsection}.}
\renewcommand{\thesubsubsectiondis}{\Alph{section}.\arabic{subsection}.\arabic{subsubsection}}
\renewcommand{\thesubsection}{\Alph{section}.\arabic{subsection}}
\renewcommand{\thesubsubsection}{\Alph{section}.\arabic{subsection}.\arabic{subsubsection}}
\renewcommand{\thesection}{\Alph{section}}
\renewcommand{\thedefn}{\thesection.\arabic{defn}}
\renewcommand{\theprop}{\thesection.\arabic{prop}}
\renewcommand{\thethm}{\thesection.\arabic{thm}}
\renewcommand{\thelem}{\thesection.\arabic{lem}}
\renewcommand{\thecor}{\thesection.\arabic{cor}}
\section{Generalized Monotonicity Lemma}\label{appmono}
The following proposition is a slight generalization of the corresponding
result reported in \cite{CR07}  required to handle cases where
the effect of event disabling is not always a self-loop
at the current state but produces a pre-specified
reconfiguration, $e.g.$,
\begin{quote}
 \textit{Disabling $q_i\xrightarrow{\sigma}q_j$ results in
$q_i \xrightarrow{\sigma} q_k$ }
\end{quote}
Note that for every state $q_i \in Q$, 
it is pre-specified where each event $\sigma$ will terminate on been disabled. This 
generalization is critical to address the partial controllability issues arising from
partial observation at the supervisory level.
\begin{prop} \label{propmonotone} \textbf{(Monotonicity)}
Let \Gt be  reconfigured to $G_\theta^\# = (Q,\Sigma,\delta^\#,(1-\theta)\widetilde{\Pi}^\#,\boldsymbol{\chi},\mathscr{C})$ as follows:
$\forall i,j,k \in \{1,\,2,\,\cdots \,,\,n\}$, the $(i,j)^{th}$
element $\Pi_{ij}^{\#}$ and the $(i, k)^{th}$ element
$\Pi_{ik}^{\#}$ of $\Pi^{\#}$ are obtained as:
\begin{align}\left.
\begin{array}{c}
\Pi_{ij}^{\#} = \Pi_{ij} + \beta_{ij} \\ \Pi_{ik}^{\#} = \Pi_{ik} -
\beta_{ij}\end{array} \right \} \ & \textrm{if} \
\mu_j > \mu_k \ \textrm{with} \ \beta_{ij} > 0 \nonumber \\
\left.
\begin{array}{c}
\Pi_{ij}^{\#} = \Pi_{ij} \\ \Pi_{ik}^{\#} = \Pi_{ik}
\end{array} \right \} \ & \textrm{if} \ \mu_j = \mu_k
 \\
\left.\begin{array}{c}
\Pi_{ij}^{\#} = \Pi_{ij} - \beta_{ij} \\ \Pi_{ik}^{\#} = \Pi_{ik} +
\beta_{ij}\end{array} \right \} \ & \textrm{if} \ \mu_j < \mu_k \
\textrm{with} \ \beta_{ij} > 0  \nonumber
\end{align}
Then for the respective measure vectors  be $\boldsymbol{\nu}_{\theta}$ and
$\boldsymbol{\nu}^{\#}_\theta$, 
\begin{gather} \boldsymbol{\nu}^{\#}_\theta \geqq_{\textrm{\sffamily \textsc{Elementwise}}}
\boldsymbol{\nu}_{\theta} \ \ \forall \theta \in (0,1)
\end{gather} with equality holding if and only if $\Pi ^{\#}=\Pi.$
\end{prop} \vspace{3pt}
\begin{proof} From the definition of renormalized measure (Definition~\ref{defrenormmeas}), we have
\begin{align*}\label{optifull-eqmtl}
\boldsymbol{\nu}^{\#}_\theta-\boldsymbol{\nu}_\theta &=
\theta\left[{I-(1-\theta)\Pi ^{\#}}\right]^{-1}-\theta\left[ {I-(1-\theta)\Pi}\right]^{-1}\boldsymbol{\chi}\nonumber \\
 &=\left[I-(1-\theta)\Pi^{\#}\right]^{-1}(1-\theta)\left(\Pi^{\#}-\Pi\right)\boldsymbol{\nu}_{\theta}
\end{align*}
Defining the matrix $\Delta \triangleq
\Pi^{\#}-\Pi$, and the $i^{th}$ row of
$\Delta$ as $\Delta_i$, it follows that
\begin{gather}
{\Delta_i} \boldsymbol{\nu}_{\theta}=\sum\limits_j {\Delta _{ij} }\boldsymbol{\nu}_{\theta}\vert_j
= \sum_j  \beta_{ij}\Gamma_{ij} 
\intertext{where}
\Gamma_{ij} = \left \{
\begin{array}{cc} (\boldsymbol{\nu}_{\theta}\vert_k - \boldsymbol{\nu}_{\theta}\vert_j) & \textrm{if} \
\boldsymbol{\nu}_{\theta}\vert_k > \boldsymbol{\nu}_{\theta}\vert_j \\ 0 & \textrm{if} \ \boldsymbol{\nu}_{\theta}\vert_k = \boldsymbol{\nu}_{\theta}\vert_j \\
(\boldsymbol{\nu}_{\theta}\vert_j - \boldsymbol{\nu}_{\theta}\vert_k) & \textrm{if} \ \boldsymbol{\nu}_{\theta}\vert_k < \boldsymbol{\nu}_{\theta}\vert_j
\end{array} \right. \Longrightarrow  \Gamma_{ij} \geqq 0 \ \forall i,j \nonumber
\end{gather}
Since $\forall j, \ \sum_{i=1}^n\Pi_{ij} = \sum_{i=1}^n\Pi^{\#}_{ij} = 1$, it follows from non-negativity of
$\Pi$, that $[I-(1-\theta)\Pi^{\#}]^{-1}\geqq_{\textrm{\sffamily \textsc{Elementwise}}}
\mathbf{0}$. Since $\beta_{ij}  > 0 \ \forall \ i,j $, it follows
that $ {\Delta_i} \boldsymbol{\nu }_\theta \ \geq \ 0 \ \forall i \
\Rightarrow \boldsymbol{\nu}^{\#}_\theta\geqq_{\textrm{\sffamily \textsc{Elementwise}}} \boldsymbol{\nu }_\theta$. For
$\boldsymbol{\nu}_{\theta}\vert_j \ne 0$ and $\Delta$ as defined above, $\Delta_i
\boldsymbol{\nu }_\theta=0$ if and only if $\Delta =0$. Then,
$\Pi^{\#}=\Pi$ and
$\boldsymbol{\nu}^{\#}_\theta= \boldsymbol{\nu}_\theta$.
\end{proof}\vspace{0pt}
\section{Pertinent Algorithms For Measure-theoretic Control}\label{appen}
This section enumerates the pertinent algorithms for computing the
optimal supervision policy for a perfectly observable plant \G. For proof of correctness the 
reader is referred to \cite{CR07}.

In Algorithm~\ref{Algorithm01}, we use the following notation:
 $M_0 = \IPq$ , $M_1 = \Bigg [\mathbb{I}-\IPq \Bigg ]$,  $M_2 = \mathop{inf}_{\alpha
\ne 0}\bigg | \bigg| \big [\mathbb{I} - \mathbf{P} + \alpha
\mathscr{P}\big ]^{-1} \bigg |\bigg |_{\infty}$
\vspace{0pt}
Also, as defined earlier, $\Q$ is the stable probability distribution which can be computed using
methods reported widely in the literature~\cite{St97}.

\renewcommand{\thesection}{\Alph{section}}
\renewcommand{\thedefn}{\thesection.\arabic{defn}}
\renewcommand{\theprop}{\thesection.\arabic{prop}}
\renewcommand{\thethm}{\thesection.\arabic{thm}}
\renewcommand{\thelem}{\thesection.\arabic{lem}}
\restylealgo{linesnumbered,ruled,vlined,noend}
\begin{algorithm}[H]
 \footnotesize 
  \SetKwData{Left}{left}
  \SetKwData{This}{this}
  \SetKwData{Up}{up}
  \SetKwFunction{Union}{Union}
  \SetKwFunction{FindCompress}{FindCompress}
  \SetKwInOut{Input}{input}
  \SetKwInOut{Output}{output}
  \SetKw{Tr}{true}
   \SetKw{Tf}{false}
  \caption{Computation of Optimal Supervisor}\label{Algorithm02}
\Input{$\mathbf{P}, \ \boldsymbol{\chi}, \ \mathscr{C}$}
\Output{Optimal set of disabled transitions
$\mathscr{D}^{\star}$} \Begin{ Set
$\mathscr{D}^{[0]}=\emptyset$, $\widetilde{\Pi}^{[0]}=\widetilde{\Pi} $, $\theta^{[0]}_{\star} = 0.99$, $k \ = \ 1$\;
 \While{($ \texttt{Terminate}$ == \Tf)}{
 Compute $\theta_{\star}^{[k]}$\tcc*[r]{Algorithm~\ref{Algorithm01}}
 Set $\widetilde{\Pi}^{[k]} = \frac{1-\theta_{\star}^{[k]}}{1-\theta_{\star}^{[k-1]}}\widetilde{\Pi}^{[k-1]}$\;
 Compute $\boldsymbol{\nu}^{[k]}$ \;
 \For{$j = 1$ \textbf{to} $n$}{
\For{$i = 1$ \textbf{to} $n$}{
 Disable  all controllable
  $q_i \xrightarrow[]{\sigma} q_j$ s.t.
$\boldsymbol{\nu}^{[k]}_j < \boldsymbol{\nu}^{[k]}_i $ \;
 Enable all controllable   $q_i \xrightarrow[]{\sigma} q_j$
 s.t.
$\boldsymbol{\nu}^{[k]}_j \geqq \boldsymbol{\nu}^{[k]}_i $ \;
} } Collect all disabled transitions in $\mathscr{D}^{[k]}$\;
\eIf{$\mathscr{D}^{[k]} ==
\mathscr{D}^{[k-1]}$}{$\texttt{Terminate} =$ \Tr \;}{$k \ = \
k \ + \ 1$ \;}
 }
 $\mathscr{D}^{\star} \ = \ \mathscr{D}^{[k]}$
 \tcc*[r]{Optimal disabling set}
  }\vspace{0pt}
\end{algorithm}
\begin{algorithm}[H]
 \footnotesize 
  \SetKwInOut{Input}{input}
  \SetKwInOut{Output}{output}
  \SetKw{Tr}{true}
   \SetKw{Tf}{false}
  \caption{Computation of the Critical Lower Bound $\theta_{\star}$ }\label{Algorithm01}
\Input{$\mathbf{P}, \ \boldsymbol{\chi}$}
\Output{$\theta_{\star}$} \Begin{ Set $\theta_{\star} =
1$, $\theta_{curr} = 0$, Compute $\Q$ , $M_0$ ,
 $M_1 $,
  $M_2 $\;
 \For{$j = 1$ \textbf{to} $n$}{
\For{$i = 1$ \textbf{to} $n$}{ \eIf{$\left (\Q
\boldsymbol{\chi} \right )_i - \left (\Q \boldsymbol{\chi}
\right )_j \neq 0$}{$\theta_{curr} = \frac{1}{8M_2}\big \vert
\left (\Q \boldsymbol{\chi} \right )_i - \left (\Q
\boldsymbol{\chi} \right )_j \big \vert $}{
\For {$r=0$ \textbf{to} $n$}{ \eIf{$\left
(M_0\boldsymbol{\chi} \right )_i \neq \left
(M_0\boldsymbol{\chi} \right )_j$}{\textbf{Break}\;}{\If
{$\left ( M_0 M_1^r \boldsymbol{\chi} \right )_i \neq \left (
M_0 M_1^r \boldsymbol{\chi} \right )_j$}{\textbf{Break}\;} } }
\eIf{$r==0$}{$\theta_{curr} = \frac{\vert \left \{ (M_0 -\Q
)\boldsymbol{\chi} \right \}_i - \left \{ (M_0 -\Q
)\boldsymbol{\chi} \right \}_j \vert}{8M_2} $\;} {\eIf { $r >
0$
 \textbf{AND} $ r \leq n $ }{ $ \theta_{curr} = \frac{\vert
\left (M_0 M_1\boldsymbol{\chi} \right )_i - \left (M_0 M_1
\boldsymbol{\chi} \right )_j \vert }{2^{r+3}M_2} $ \;}{ $
\theta_{curr} = 1 $ \;}}}
$\theta_{\star}$ = $\mathrm{min} (
\theta_{\star} , \theta_{curr} ) $ \;
 } } }
\end{algorithm}
%
%
\begin{algorithm}[H]
\restylealgo{linesnumbered,ruled,noline,noend}
 \footnotesize 
  \SetKwData{Left}{left}
  \SetKwData{This}{this}
  \SetKwData{Up}{up}
  \SetKwFunction{Union}{Union}
  \SetKwFunction{FindCompress}{FindCompress}
  \SetKwInOut{Input}{input}
  \SetKwInOut{Output}{output}
  \SetKw{Tr}{true}
   \SetKw{Tf}{false}
  \caption{Computation of Phantom Automaton}\label{AlgoPh}
\Input{$Q$, $\Sigma$, $\widetilde{\pi}$, Unobservability map $\p$}
\Output{$\mathscr{P}(\Pi)$}
\Begin{
Set $\widetilde{\pi}^{\mathscr{P}} = \widetilde{\pi}$\;
\For{$i=1$ \textbf{to} $n$}{
\For{$j=1$ \textbf{to} $m$}{
\If{$\p(q_i,\sigma_j) = \sigma_j$}{
$\tilde{\pi}^{\mathscr{P}}_{ij} = 0$\tcc*[r]{Delete
transition}}
}}
\For{$i=1$ \textbf{to} $n$}{
\For{$j=1$ \textbf{to} $n$}{
$\mathscr{P}(\Pi)_{ij} =
\sum_{k:\delta(q_i,\sigma_k) =
q_j}\tilde{\pi}^{\mathscr{P}}_{ik}$\; }} }
\end{algorithm}
\begin{algorithm}[H]
 \restylealgo{linesnumbered,ruled,noline,noend}
 \footnotesize 
  \SetKwData{Left}{left}
  \SetKwData{This}{this}
  \SetKwData{Up}{up}
  \SetKwFunction{Union}{Union}
  \SetKwFunction{FindCompress}{FindCompress}
  \SetKwInOut{Input}{input}
  \SetKwInOut{Output}{output}
  \caption{Petri Net observer}\label{Alg2}
\dontprintsemicolon
\Input{$\langle G,p \rangle$}
\Output{Petri net observer} \Begin{I. Create a place $q_j$ for
each state $q_j$ in $\langle G,p \rangle$;\\
II. The set of transition labels is $\Sigma$;\\
\For{each \textbf{observable} transition $q_j
\xrightarrow[\sigma]{} q_k$ in $\langle G,p \rangle$ }{ I. Set
the initial state in $\langle G,p \rangle$ to $q_k$;\\
II. Compute
$\overline{Q}(\epsilon)$;\\
III. Add a transition labeled $\sigma$ from the place $q_j$
with
output arcs to all places $q_l \in \overline{Q}(\epsilon)$;\\
    }
    \For{each place $q_j$ in the net}{
    \For{each event $\sigma \in \Sigma$}{
    \If{there is no transition with label $\sigma$ from $q_j$}{
     I. Add a flush-out arc with label $\sigma$ from $q_j$}
      }
    }
  }
\end{algorithm}
%
\begin{algorithm}[H]
\restylealgo{linesnumbered,ruled,noline,noend}
 \footnotesize 
  \SetKwData{Left}{left}
  \SetKwData{This}{this}
  \SetKwData{Up}{up}
  \SetKwFunction{Union}{Union}
  \SetKwFunction{FindCompress}{FindCompress}
  \SetKwInOut{Input}{input}
  \SetKwInOut{Output}{output}
  \caption{Online computation of possible states}\label{Alg3}
\dontprintsemicolon
\Input{Petri net observer, Observed sequence $\omega = \tau_1\tau_2 \ldots \tau_r $}
\Output{$\overline{Q}(\omega)$} \Begin{I. Compute the initial
marking for the observer as follows: \\
\hspace{10pt} a. Compute $\overline{Q}(\epsilon)$;\\
\hspace{10pt} b. Put a token in each place $q_j \in
\overline{Q}(\epsilon)$;\\ \For{$j = 1$ to $r$}{ I. Fire all
enabled transitions labeled $\tau_j$;\\
\For{ each place $q_j$ in the observer}{ \If{ number of tokens
in $q_j$ > 0 }{ I. Normalize the number of tokens in $q_j$ to
$1$.
   }
  }
 }
II. $\overline{Q}(\omega) = \big \{ q_j \ \vert q_j \
\textrm{has one token} \big \}$ ;
 }
\end{algorithm}
\bibliographystyle{siam}
\bibliography{FULL_BIB}

\begin{thebibliography}{10}

\bibitem{AK01}
{\sc A.~Atrash and S.~Koenig}, {\em Probabilistic planning for behavior-based
  robots}, in Proceedings of the International FLAIRS conference (FLAIRS),
  2001, pp.~531--535.

\bibitem{BR97}
{\sc R.~Bapat and T.~Raghavan}, {\em Nonnegative matrices and Applications},
  Cambridge University Press, 1997.

\bibitem{Berman1979}
{\sc A.~Berman and R.~J. Plemmons}, {\em Nonnegative matrices in the
  mathematical science}, Academic Press, New York, 1979.

\bibitem{BRS96}
{\sc D.~Burago, M.~D. Rougemont, and A.~Slissenko}, {\em On the complexity of
  partially observed markov decision processes}, Theoretical Computer Science,
  157 (1996), pp.~161--183.

\bibitem{Cass94}
{\sc A.~R. Cassandra}, {\em Optimal policies for partially observable markov
  decision processes}, tech. rep., Providence, RI, USA, 1994.

\bibitem{CK98}
{\sc A.~R. Cassandra}, {\em Exact and approximate algorithms for partially
  observable markov decision processes}, PhD thesis, Providence, RI, USA, 1998.
\newblock Adviser-Leslie Pack Kaelbling.

\bibitem{C-PhD}
{\sc I.~Chattopadhyay}, {\em Quantitative control of probabilistic discrete
  event systems}, PhD Dissertation, Dept. of Mech. Engg. Pennsylvania State
  University, {http{://} etda.libraries.psu.edu {/} theses / approved /
  WorldWideIndex / ETD-1443},  (2006).

\bibitem{CR06a}
{\sc I.~Chattopadhyay and A.~Ray}, {\em A language measure for partially
  observed discrete event systems}, Int. J. Control, 79 (2006), pp.~1074--1086.

\bibitem{CR06}
\leavevmode\vrule height 2pt depth -1.6pt width 23pt, {\em Renormalized measure
  of regular languages}, Int. J. Control, 79 (2006), pp.~1107--1117.

\bibitem{CRg07}
{\sc I.~Chattopadhyay and A.~Ray}, {\em Generalized projections in finite state
  automata \& decidability of state determinacy}, American Control Conference
  (ACC),  (2007), pp.~5664--5669.

\bibitem{CR08}
{\sc I.~Chattopadhyay and A.~Ray}, {\em Structural transformations of
  probabilistic finite state machines}, International Journal of Control, 81
  (2008), pp.~820--835.

\bibitem{CR07}
\leavevmode\vrule height 2pt depth -1.6pt width 23pt, {\em
  Language-measure-theoretic optimal control of probabilistic finite-state
  systems}, Int. J. Control, 80 (August,2007), pp.~1271--1290.

\bibitem{G92-2}
{\sc V.~Garg}, {\em Probabilistic lnaguages for modeling of \textsc{DED}s},
  Proceedings of 1992 IEEE Conference on Information and Sciences,  (Princeton,
  NJ, March 1992), pp.~198--203.

\bibitem{G92}
\leavevmode\vrule height 2pt depth -1.6pt width 23pt, {\em An algebraic
  approach to modeling probabilistic discrete event systems}, Proceedings of
  1992 IEEE Conference on Decision and Control,  (Tucson, AZ, December 1992),
  pp.~2348--2353.

\bibitem{G01}
{\sc M.~Gribaudo, M.~Sereno, A.~Horvath, and A.~Bobbio}, {\em Fluid stochastic
  petri nets augmented with flush-out arcs: Modelling and analysis.}, Discrete
  Event Dynamic Systems,  (2001), pp.~97--117.

\bibitem{Hans98}
{\sc E.~A. Hansen}, {\em Finite-memory control of partially observable
  systems}, PhD thesis, Amherst, MA USA, 1998.
\newblock Director-Shlomo Zilberstein.

\bibitem{HMU01}
{\sc J.~E. Hopcroft, R.~Motwani, and J.~D. Ullman}, {\em Introduction to
  Automata Theory, Languages, and Computation, 2nd ed.}, Addison-Wesley, 2001.

\bibitem{KHW95}
{\sc N.~Kushmerick, S.~Hanks, and D.~S. Weld}, {\em An algorithm for
  probabilistic planning}, Artif. Intell., 76 (1995), pp.~239--286.

\bibitem{LGM01}
{\sc C.~Lusena, J.~Goldsmith, and M.~Mundhenk}, {\em Nonapproximability results
  for partially observable markov decision processes}, Journal of Artificial
  Intelligence Research, 14 (2001), p.~2001.

\bibitem{MHC99}
{\sc O.~Madani, S.~Hanks, and A.~Condon}, {\em On the undecidability of
  probabilistic planning and infinite-horizon partially observable markov
  decision problems}, in AAAI '99/IAAI '99: Proceedings of the sixteenth
  national conference on Artificial intelligence and the eleventh Innovative
  applications of artificial intelligence conference innovative applications of
  artificial intelligence, Menlo Park, CA, USA, 1999, American Association for
  Artificial Intelligence, pp.~541--548.

\bibitem{MR91}
{\sc D.~Mcallester and D.~Rosenblitt}, {\em Systematic nonlinear planning}, in
  In Proceedings of the Ninth National Conference on Artificial Intelligence,
  1991, pp.~634--639.

\bibitem{MA98}
{\sc J.~Moody and P.~Antsaklis}, {\em Supervisory Control ofDiscrete Event
  Systems Using Petri Nets}, Kluwer Academic, 1998.

\bibitem{P71}
{\sc A.~Paz}, {\em Introduction to probabilistic automata (Computer science and
  applied mathematics)}, Academic Press, Inc., Orlando, FL, USA, 1971.

\bibitem{PW92}
{\sc J.~S. Penberthy and D.~S. Weld}, {\em Ucpop: A sound, complete, partial
  order planner for adl}, Morgan Kaufmann, 1992, pp.~103--114.

\bibitem{PW93}
{\sc J.~Peng and R.~J. Williams}, {\em Efficient learning and planning within
  the dyna framework}, in Adaptive Behavior, 1993, pp.~437--454.

\bibitem{Put90}
{\sc M.~Puterman}, {\em Handbook of Operations Research}, vol.~2, North Holland
  Publishers, 1990, ch.~Markov Decision Processes, pp.~331--434.

\bibitem{R63}
{\sc M.~Rabin}, {\em Probablistic automata}, Information and Control, 6 (1963),
  pp.~230--245.

\bibitem{RW87}
{\sc P.~J. Ramadge and W.~M. Wonham}, {\em Supervisory control of a class of
  discrete event processes}, SIAM J. Control and Optimization, 25 (1987),
  pp.~206--230.

\bibitem{R05}
{\sc A.~Ray}, {\em Signed real measure of regular languages for discrete-event
  supervisory control}, Int. J. Control, 78 (2005), pp.~949--967.

\bibitem{RPP05}
{\sc A.~Ray, V.~Phoha, and S.~Phoha}, {\em Quantitative measure for discrete
  event supervisory control}, Springer, New York, 2005.

\bibitem{R88}
{\sc W.~Rudin}, {\em Real and Complex Analysis, 3rd ed.}, McGraw Hill, New
  York, 1988.

\bibitem{S78}
{\sc E.~J. Sondik}, {\em The optimal control of partially observable markov
  processes over the infinite horizon: Discounted costs}, Operations Research,
  26 (1978), pp.~282--304.

\bibitem{St97}
{\sc W.~Stewart}, {\em Computational Probability: Numerical methods for
  computing stationary distribution of finite irreducible Markov chains},
  Springer, New York, 1999.

\bibitem{W93}
{\sc D.~J. White}, {\em Markov Decision Processes}, Wiley, 1993.

\bibitem{Z01}
{\sc W.~Zhang and G.~M.~J. Golin}, {\em Algorithms for partially observable
  markov decision processes}, tech. rep., Hong Kong University of Science and
  Technology, 2001.

\end{thebibliography}
\end{document}